\documentclass[11pt,twoside,reqno]{amsart}
\title{Combinatorial decomposition theorem for Hitchin systems via zonotopes}
\author{Mirko Mauri, Luca Migliorini and Roberto Pagaria}
\usepackage[utf8]{inputenc}
\usepackage[T1]{fontenc}
\usepackage{amsthm}
\usepackage{mathrsfs}

\usepackage{mathtools}
\usepackage{graphicx}
\usepackage{longtable}
\usepackage{float}
\usepackage{wrapfig}
\usepackage{rotating}
\usepackage{amsmath}
\usepackage{color}
\usepackage{textcomp}
\usepackage{marvosym}
\usepackage{wasysym}
\usepackage{amssymb}
\usepackage{hyperref}
\usepackage{tikz}
\usepackage{tikz-cd}
\usepackage{float}
\usepackage{stmaryrd}
\usepackage{enumerate}
\usepackage{caption}
\usepackage{subcaption}
\usepackage{multirow}
\usepackage{fancyhdr}
\usepackage{enumitem} 
\usetikzlibrary{matrix}
\usetikzlibrary{shapes.multipart}
\usetikzlibrary{arrows}
\usetikzlibrary{arrows.meta}
\usepackage{amsthm,amsmath,amsfonts,amscd}
\usepackage{tikz-3dplot}
\usetikzlibrary{decorations.markings}
\usetikzlibrary{arrows}
\usepackage{parskip}
\allowdisplaybreaks
\newcommand{\Z}{\mathbb{Z}}

\newcommand{\QQ}{\mathbb{Q}}
\newcommand{\RR}{\mathbb{R}}

\newcommand{\ZZ}{\mathbb{Z}}

\newcommand{\Gl}{\operatorname{GL}}

\newcommand{\IC}{\operatorname{IC}}
\newcommand{\IH}{I\!H}

\newcommand{\nbar}{\underline{n}}
\newcommand{\mbar}{\underline{m}}
\newcommand{\lbar}{\underline{l}}
\newcommand{\kbar}{\underline{k}}
\newcommand{\Sbar}{\underline{S}}

\newcommand{\MDol}{{M}}

\newcommand{\Stab}{\operatorname{Stab}}

\newcommand{\codim}{\operatorname{codim}}
\newcommand{\id}{\operatorname{id}}
\newcommand{\rank}{\operatorname{rank}}

\newcommand{\Aut}{\operatorname{Aut}}
\newcommand{\Pic}{\operatorname{Pic}}

\newcommand{\coker}{\mathrm{coker}}
\newcommand{\Jacob}{\mathrm{Pic}^0}
\newcommand{\qb}{\underline{q}}
\newcommand{\Fl}{\mathcal{L}} 

\DeclareMathOperator{\rvol}{vol}
\DeclareMathOperator{\Half}{Half}
\DeclareMathOperator{\Ind}{Ind}
\DeclareMathOperator{\sgn}{sgn}

\DeclareMathOperator{\lcm}{lcm}

\usepackage{scalerel}

\usepackage{verbatim}

\begingroup
\makeatletter
\@for\theoremstyle:=definition,remark,plain\do{%
\expandafter\g@addto@macro\csname th@\theoremstyle\endcsname{%
\addtolength\thm@preskip\parskip
}%
}
\endgroup
\usepackage{graphicx}
\usepackage[capitalise]{cleveref}
\newtheorem{thm}{Theorem}[section]
\newtheorem{lem}[thm]{Lemma}
\newtheorem{cor}[thm]{Corollary}
\newtheorem{defn}[thm]{Definition}
\newtheorem{prop}[thm]{Proposition}
\newtheorem{conj}[thm]{Conjecture}
\newtheorem{quest}[thm]{Question}

\theoremstyle{definition}
\newtheorem{notation}[thm]{Notation}
\newtheorem{exa}[thm]{Example}
\newtheorem{rmk}[thm]{Remark}
\crefname{thm}{Theorem}{Theorems}
\Crefname{thm}{Theorem}{Theorems}
\Crefname{thm}{Theorem}{Theorems}
\Crefname{thm}{Theorem}{Theorems}
\crefname{lem}{Lemma}{Lemmas}
\Crefname{lem}{Lemma}{Lemmas}
\crefname{Conjecture}{Conjecture}{Conjectures}
\Crefname{Conjecture}{Conjecture}{Conjectures}
\crefname{Corollary}{Corollary}{Corollaries}
\Crefname{Corollary}{Corollary}{Corollaries}
\crefname{Claim}{Claim}{Claims}
\Crefname{Claim}{Claim}{Claims}
\crefname{Proposition}{Proposition}{Propositions}
\Crefname{Proposition}{Proposition}{Propositions}
\crefname{Remark}{Remark}{Remarks}
\Crefname{Remark}{Remark}{Remarks}
\crefname{Definition}{Definition}{Definitions}
\Crefname{Definition}{Definition}{Definitions}
\crefname{Example}{Example}{Examples}
\Crefname{Example}{Example}{Examples}
\crefname{Exercise}{Exercise}{Exercises}
\Crefname{Exercise}{Exercise}{Exercises}

\newtheoremstyle{plain2}    
   {}            
   {}            
   {\itshape}    
   {}            
   {\bfseries}   
   {.}           
   {5pt plus 1pt minus 1pt}  
   {{\thmnumber{#1} \thmname{#2}{\thmnote{ (#3)}}}}          

\newcommand{\TBC}[1]{\footnote{{\bf{\color{red} TBC $<<$}} #1 \bf{\color{red} $>>$ }}}

\begin{document}
\maketitle
\begin{abstract} 
    We determine the summands of the decomposition theorem for the Hitchin system for $\Gl_n$, in arbitrary degree, over the locus of reduced spectral curves. The key ingredient is an equivariant formula for lattice point counts in graphical zonotopes. 
\end{abstract}

\section{Introduction}
Intersection cohomology is a topological invariant well-adapted to investigate the symmetries of singular spaces. Nowadays, it is an established tool in Hodge theory, but also in representation theory and combinatorics. For instance, the symmetries of the $h$-polynomial of a polytope can be explained in terms of dualities of the intersection cohomology of the associated toric variety; see \cite{Fieseler91, Braden2006}. Via the geometric Satake correspondence, the irreducible representations of a reductive group $G$ can be realized as the intersection cohomology of affine Schubert varieties for the Langlands dual group ${^L}G$; see \cite{Ginz1995, MV2007}. 
Another example is the calculation of the Frobenius character of the Orlik--Terao algebra by using the intersection cohomology ring of an hypertoric variety; see \cite{BradenProudfoot2009,MoseleyProudfootYoung2017,Pagaria2022}.
Viceversa, the intersection cohomology of toric varieties is controlled by the combinatorics of their moment polytopes, as one may expect given the combinatorial nature of these varieties; see again \cite{Fieseler91}.

It may sound less obvious that the intersection cohomology of moduli spaces with no a priori combinatorial origin, like spaces of Higgs bundles, admits a combinatorial characterization. In this paper, we show how to obtain cohomological information for these spaces via lattice point count in rational polytopes called zonotopes. 

\subsection{Decomposition theorem for the Hitchin fibration}
We consider the \emph{Dolbeault moduli space} $M(n,d)$ parametrising semistable Higgs bundles $(\mathcal{E}, \phi)$ on a compact Riemann surface $C$ with $\rank(\mathcal{E})=n$, and $\deg(\mathcal{E})=d$; see \S\ref{sec:Hitchincompactified}. One can argue that the entire study of nonabelian Hodge theory may be thought of as the study of the
geometry of $M(n,d)$. Dolbeault moduli spaces play a central role in various other fields of mathematics
like mirror symmetry, Teichmüller theory, knot theory and in the geometric Langlands program. 

A key feature of the geometry of $M(n,d)$ is that it is endowed with a projective fibration called \emph{Hitchin fibration} 
\[
    \chi(n,d)\colon  \MDol(n,d) \to A_{n}\coloneqq \bigoplus^{n}_{i=1} H^0(C, \omega_C^{\otimes i})
\]
which assigns to $(\mathcal{E}, \phi)$ the characteristic polynomial $\mathrm{char}(\phi)$ of the Higgs field $\phi$; see again \S\ref{sec:Hitchincompactified}. Set $A^{\mathrm{red}}_{n} \subseteq A_{n}$ the locus of reduced characteristic polynomial, and $M(n,d)^{\mathrm{red}} \coloneqq \chi^{-1}(n,d)(A^{\mathrm{red}}_{n})$.

The intersection cohomology of $\MDol(n,d)$, denoted $\IH^*(\MDol(n,d), \QQ)$, is the global cohomology of the intersection complex $\IC(\MDol(n,d), \QQ)$ on $\MDol(n,d)$, or of 
the direct image complex $R \chi(n,d)_* \IC(\MDol(n,d), \QQ)$ on the base $A_n$. The advantage of the second complex is that we may decompose it in smaller complexes called \emph{Ng\^{o} strings}. On the reduced locus $A^{\mathrm{red}}_{n}$, the Ng\^{o} strings are indexed by a subset of partitions $\underline{n}$ of $n$, and they are denoted $\mathscr{S}(\mathscr{L}_{\underline{n}}(d))$. 
 Any string $\mathscr{S}(\mathscr{L}_{\underline{n}}(d))$ is a combination of the cohomology of an abelian group scheme independent of $d$, and a local system $\mathscr{L}_{\underline{n}}(d)$ depending on $d$; see \S \ref{sec:Ngo}. 
 
 These local systems $\mathscr{L}_{\underline{n}}(d)$ should be considered as the combinatorial part of the intersection cohomology of $\MDol(n,d)$: a partition $\underline{n}$ corresponds to a choice of a zonotope in $\RR^{|\underline{n}|} \simeq \ZZ^{|\underline{n}|} \otimes \RR$, and the degree $d$ determines a translation of the zonotope. The local system $\mathscr{L}_{\underline{n}}(d)$ is a subrepresentation of a group of automorphisms of the zonotope on the lattice points contained in its translate. 

The determination of these representations is a subtle task, and involves a delicate combination of combinatorics and algebraic geometry that we carried out in this paper. In \S \ref{sec:otherwork} and \cite[\S 1]{MM2022} we hint at the geometric significance of the result. Here we summarize the main output, and discuss it in detail in the next section.

\begin{thm}[Effective decomposition theorem for the Hitchin fibration]\label{thm:effectivedecomposition} For any partition $\underline{n}=\{n_1, \ldots, n_r\}$ of $n$, let $S_{\underline{n}} \subseteq A^{\mathrm{red}}_{n}$ be the closure of the locus the characteristic polynomials $\mathrm{char}(\phi)$ whose irreducible factors have degree $n_i$, and $g_{\underline{n}}: \mathscr{A}^{\times}_{\underline{n}} \to S^\times_{\underline{n}}$ be the relative Jacobian of the simultaneous normalization of the spectral curves of equation $\mathrm{char}(\phi)=0$, with $\mathrm{char}(\phi)$ in a dense open set $S^\times_{\underline{n}} \subseteq S^{\circ}_{\underline{n}}= S_{\underline{n}} \setminus \bigcup_{\underline{m} \leq \underline{n}} S_{\underline{m}}$. Denote by $\Stab(\underline{n})$ the subgroup of the symmetric group on $r$ elements stabilizing $\underline{n}$.

 Then there is an isomorphism in the bounded derived category of algebraic mixed Hodge modules $D^bMHM_{\text{alg}}(A^{\mathrm{red}}_{n})$, or in the bounded derived category $D^b(A^{\mathrm{red}}_n)$ of $\QQ$-constructible
complexes on $A^{\mathrm{red}}_n$ (ignoring the
Tate shifts):
\begin{equation}\label{eq:effectivedt}
    R\chi(n,d)_*\IC({M(n,d)}, \QQ)|_{A^{\mathrm{red}}_n}\simeq \bigoplus_{\substack{\underline{n}\vdash \, n \\ \text{non-}d\text{-integral} \\ \text{or trivial}}} \IC(S_{\underline{n}}, Rg_{\underline{n}, *}\QQ_{\mathscr{A}^{\times}_{\underline{n}}} \otimes \mathscr{L}_{\underline{n}}(d))[-2  \mathrm{codim} S_{\underline{n}}](-\mathrm{codim} S_{\underline{n}}),
\end{equation} 
where $\mathscr{L}_{\underline{n}}(d)$ is a polarizable variation of pure Hodge structures of weight zero and of Hodge-Tate type on $S^\times_{\underline{n}}$. 

The monodromy of $\mathscr{L}_{\underline{n}}(d)$ actually factors through a representation of the stabilizer $\Stab(\underline{n})$, i.e.\ for any $a \in S^\times_{\underline{n}}$ we have
\[\pi_1(S^\times_{\underline{n}}) \twoheadrightarrow \pi_1(S^\circ_{\underline{n}}) \twoheadrightarrow \Stab(\underline{n}) \xrightarrow{} \Aut(\mathscr{L}_{\underline{n}}(d)_a);\]
see \cref{rmk:represent}. The $\Stab(\underline{n})$-representation admits a clear combinatorial characterization as described in \cref{thm:combchar}, and its rank is computed in \cref{thm:rank}.
\end{thm}

\subsection{Supports and combinatorial local systems of the Hitchin fibration}\label{sec:mainresults_intro} In order to explain and prove \cref{thm:effectivedecomposition}, we first introduce the notion of $d$-integral partition.
\begin{defn}
 A partition $\underline{n} = \{n_i\}$ of $n = \sum_i n_i$ is \emph{$d$-integral} if $d n_i/n \in \mathbb{Z}$. 
\end{defn}
\begin{exa}\label{ex:d-integral}
If $\gcd(n,e)=1$, the only $e$-integral partition of $n$ is the trivial one, i.e.\ $\{n\}$. On the contrary, any partition of $n$ is $0$-integral.
\end{exa}
Such partitions arise naturally while studying the singularities of $M(n,d)^{\mathrm{red}}$. Indeed, a singular point of $M(n,d)^{\mathrm{red}}$ corresponds to a strictly polystable Higgs bundle	$(E_1, \phi_1)\oplus \ldots \oplus (E_r, \phi_r)$ with $r>1$, where $(\mathcal{E}_i, \phi_i)$ are distinct stable Higgs bundles of slope $\deg E_i/\rank E_i = d/n$. In fact, $\{\rank E_i\}$ is a $d$-integral partition of $n$. Let $M^{\circ}_{\underline{n}}(d)$ be the locus of polystable Higgs bundles of multirank $\underline{n}$. Then the space $M(n,d)^{\mathrm{red}}$ admits a Whitney stratification by partition type
	\[
	M(n,d)^{\mathrm{red}} = \bigsqcup_{\underline{n} 
	} M^{\circ}_{\underline{n}}(d),
	\]
	where the sum runs over the $d$-integral partitions of $n$; see also \cite[\S 1.1]{MM2022}.
	
For any partition $\underline{n}=\{n_i\}$ of $n$, let $S_{\underline{n}} \subseteq A_{n}$ be the closure of the polynomials whose irreducible factors have degree $n_i$; see also \S \ref{sec:Ngo}. In \cite[Cor. 3.8]{MM2022}, the first and second authors showed that the supports of the Ng\^{o} strings of $R \chi(n,d)_* \IC(\MDol(n,d), \QQ)|_{A^{\mathrm{red}}_n}$ range among the subvarieties $S_{\underline{n}}$; see \cref{thm:Ngostring}. Here we determine which of them actually support a non-zero Ng\^{o} string.

\begin{thm}[Supports in arbitrary degree]\label{thm:Hodge-to-singular-for-supports}
The supports of $R \chi(n,d)_* \IC(\MDol(n,d), \QQ)|_{A^{\mathrm{red}}_n}$ properly contained in $A^{\mathrm{red}}_n$ are all and only the subvarieties $S_{\underline{n}}$ such that $\underline{n}$ is not a $d$-integral partition of $n$. 
\end{thm}
\begin{proof}
It follows from \cref{thm:combchar}, \cref{prop:countsphere}, and \cref{positivityofsum}.
\end{proof}
Note that the closure of the images of $M^{\circ}_{\underline{n}}(d)$ are all and only the subvarieties $S_{\underline{n}}$ such that $\underline{n}$ is indeed $d$-integral. Therefore, we have the following dichotomy: either the subvariety $S_{\underline{n}}$ has a Hodge-theoretic interpretation as summand of $R \chi(n,d)_* \IC(\MDol(n,d), \QQ)$, or it admits a description in terms of the singularities of $M(n,d)$. 

Therefore, \cref{ex:d-integral} implies that we have two extreme cases: the coprime and the zero case. If $\gcd(n,e)=1$, $M(n,e)$ is smooth and $R \chi(n,e)_* \IC(\MDol(n,e), \QQ)|_{A^{\mathrm{red}}_n}$ has the maximal number of supports as the degree varies; see also \cite[\S 1.1]{MM2022}. On the contrary, when $d=0$, $M(n,0)$ is singular with the maximal number of strata in the stratification by partition type, but no summand of $R \chi(n,0)_* \IC(\MDol(n,0), \QQ)|_{A^{\mathrm{red}}_n}$ is supported on a proper subvariety of $A^{\mathrm{red}}_n$. We could informally say that, despite their different origin, there exists a \emph{conservation law} between summands of the decomposition theorem and singular strata as the degree $d$ varies. In particular, we obtain an alternative proof of \cite[Thm 1.4]{MM2022}, which is actually a key step in the proof of \cref{thm:Hodge-to-singular-for-supports} and \ref{thm:rank}.

\begin{thm}[Full support in degree zero]\label{thm:fullsupport}
The complex $R \chi(n,0)_* \IC(\MDol(n,0), \QQ)|_{A^{\mathrm{red}}_{n}}$ has full support, i.e.\
\[
R\chi(n,0)_*\IC(\MDol(n,0), \QQ)|_{A^{\mathrm{red}}_n} \simeq \mathscr{S}(\QQ_{A_{n}})|_{A_n^{\mathrm{red}}}.
\]
\end{thm}

If $\gcd(n,e)=1$, the rank of $\mathscr{L}_{\underline{n}}(e)$ has been interpreted as the cohomology of a cographic matroid in \cite[Thm 6.11]{deCataldoHeinlothMigliorini19}, and as the top cohomology of a symplectic resolution of a normal slice of $M^{\circ}_{\underline{n}}(0)$ in \cite[Thm 1.1]{MM2022}. Here we propose a new combinatorial description that holds in arbitrary degree. The monodromy of the local system $\mathscr{L}_{\underline{n}}(d)$ is a representation of $\Stab(\underline{n})$, i.e.\ the subgroup of the symmetric group on $\ell(\underline{n})=r$ elements stabilizing $\underline{n}$. If $\underline{n}=\{n_1, \ldots, n_r\} \vdash n$, set \begin{equation}\label{eq:vectoromega}\omega_{\underline{n}}(d) \coloneqq \bigg(d \frac{n_1}{n}, \ldots, d \frac{n_r}{n}\bigg) \in \QQ^{r}.\end{equation} Then the poset of non-$\omega_{\underline{n}}(d)$-integral partition of the set $\underline{n}$, denoted  $\overline{\Pi}_{\omega_{\underline{n}}(d)}$ (cf  \cref{defn:partitionsintegral} and \ref{defn:terminologyposet}), admits an action of $\Stab(\underline{n})$.

\begin{thm}[Combinatorial characterization of $ \mathscr{L}_{\underline{n}}(d)$]\label{thm:combchar}
Let $\underline{n}$ be a non-trivial partition of $n$.
The local system $\mathscr{L}_{\underline{n}}(d)$ can be identified as a $\Stab(\underline{n})$-representation with the top reduced homology of the order complex of the poset $\overline{\Pi}_{\omega_{\underline{n}}(d)}$, twisted by the sign representation 
\[\mathscr{L}_{\underline{n}}(d) \simeq \sgn \otimes \widetilde{H}_{\ell(\underline{n})-3}(\Delta(\overline{\Pi}_{\omega_{\underline{n}}(d)})).\] In particular, the rank of $\mathscr{L}_{\underline{n}}(d)$ is the number of spheres of the order complex $\Delta(\overline{\Pi}_{\omega_{\underline{n}}(d)})$.
\end{thm}
Note that the order complex $\Delta(\overline{\Pi}_{\omega_{\underline{n}}(d)})$ is homotopic equivalent to a wedge of spheres. This follows from the LEX-shellability of the poset $\hat{\Pi}_{\omega_{\underline{n}}(d)}$; see \cref{defn:LEXshellability}, \cref{prop:LEXshell} and \cref{prop:spheres}. \begin{thm}[\cref{prop:LEXshell}; LEX-shellability]
The poset $\hat{\Pi}_{\omega_{\underline{n}}(d)}$ is LEX-shellable.
\end{thm}

If $\gcd(n,e)=1$, the poset $\Pi_{\omega_{\underline{n}}(e)}$ is simply the poset of all partitions of $[r]$. This poset appears in \cite[Lemma 6.18]{deCataldoHeinlothMigliorini19} as the lattice of flats of a cographic matroid, and it was indeed used to compute the cohomology of the matroid. Note however that in arbitrary degree the poset $\Pi_{\omega_{\underline{n}}(d)}$ may fail to be semimodular, so in general there is no associated matroid.  

The Hodge-to-singular correspondence in \cite{MM2022} provides a recursive formula to determine the rank of $\mathcal{L}_{\underline{n}}(d)$ as explained in \cite[\S 10.1]{MM2022}.
Due to the new characterization \Cref{thm:combchar}, we can now provide a closed formula.

\begin{thm}[\cref{prop:countsphere}; Formula for $\operatorname{rk} \mathscr{L}_{\underline{n}}(d)$]\label{thm:rank} 
The rank of the local system $\mathscr{L}_{\underline{n}}(d)$ is
\begin{equation}\label{eq:coefficient}
    \sum_{\substack{\lambda=\{\lambda_i\} \vdash [r] \\ \lambda \; \omega_{\underline{n}}(d)\textnormal{-integral}}} 
    (-1)^{\ell(\lambda) -1} \prod_{i=1}^{\ell(\lambda)} (\lvert \lambda_i \rvert-1)!,
\end{equation}
where the sum runs over the $\omega_{\underline{n}}(d)$-integral partitions of the set $\underline{n}$; see \cref{defn:partitionsintegral}.
\end{thm}

By direct inspection, we show that the sum in \eqref{eq:coefficient} vanishes when $\omega$ is an integral vector. Otherwise, we prove that $\Delta(\overline{\Pi}_{\omega_{\underline{n}}(d)})$ contains at least a sphere, which by \cref{thm:combchar} forces the sum in \eqref{eq:coefficient} to be positive; see \cref{positivityofsum}. These facts together give \cref{thm:Hodge-to-singular-for-supports}.

\subsection{Strategy}\label{sec:strategy}
The key input for the proof of the previous statements is that any local system $\mathscr{L}_{\underline{n}}(d)$ arises from a graphical zonotope.

According to the decomposition theorem for the Hitchin fibration (\cref{thm:Ngostring}), the local systems $\mathscr{L}_{\underline{n}}(d)$ are direct summands of $R^{2c(n,g)} \chi(n,d)_*\QQ_{\MDol(n,d)}$. The stalk of $R^{2c(n,g)} \chi(n,d)_*\QQ_{\MDol(n,d)}$ at the point $a \in A^{\mathrm{red}}_{n}$ has a basis indexed by the irreducible components of $\chi(n,d)^{-1}(a)$. Via the BNR-correspondence, $\chi(n,d)^{-1}(a)$ is a compactified Jacobian parametrising semistable rank 1 torsion free sheaves on the spectral curve $C_{a}$; see \S\ref{sec:Hitchincompactified}. In particular, the irreducible components of the compactified Jacobian of a reduced spectral curve are indexed by the degrees of the restrictions of stable line
bundles to the irreducible components of $C_{a}$; see \cref{prop:compJacobiansI}. The stability condition imposes  numerical inequalities on the multidegree of a line bundle on $C_{a}$, which are defining equations for the graphical zonotope $Z_{\Gamma}+\omega$ for a suitable graph $\Gamma$ and translation vector $\omega$; see \cref{thm:geomvscombin}.  
\begin{figure}[t]\begin{center}
\begin{tikzpicture}[scale=0.7]
\node (a) at (-4, -3) {\color{red} $M(n,0)$};
\node  at (-5.25, -4.5) {\color{red} $\chi(n,0)$};
\node (b) at (-4, -6) {$A_{n}$};
\node (c) at (10, -3) {\color{blue}  $M(n,1)$};
\node  at (11.25, -4.5) {\color{blue} $\chi(n,1)$};
\node (d) at (10, -6) {$A_{n}$};
\draw[->] (a)--(b);
\draw[->] (c)--(d);
  \draw (1,0)--(2,0)--(2,1)--(1,2)--(0,2)--(0,1)--(1,0);
  \filldraw[black] (0,0) circle (3pt);
  \filldraw[black] (1,0) circle (3pt);
 \filldraw[black] (2,0) circle (3pt);
  \filldraw[black] (0,1) circle (3pt);
  \filldraw[red] (1,1) circle (3pt);
 \filldraw[black] (2,1) circle (3pt);
   \filldraw[black] (0,2) circle (3pt);
   \filldraw[black] (1,2) circle (3pt);
  \filldraw[black] (2,2) circle (3pt);
  \draw[dashed] (5,0)--(6,0)--(6,1)--(5,2)--(4,2)--(4,1)--(5,0);
    \draw (5+0.3,0+0.3)--(6+0.3,0+0.3)--(6+0.3,1+0.3)--(5+0.3,2+0.3)--(4+0.3,2+0.3)--(4+0.3,1+0.3)--(5+0.3,0+0.3);
    \filldraw[black] (4,0) circle (3pt);
    \filldraw[black] (5,0) circle (3pt);
   \filldraw[black] (6,0) circle (3pt);
     \filldraw[black] (4,1) circle (3pt);
        \filldraw[blue] (5,1) circle (3pt);
       \filldraw[blue] (6,1) circle (3pt);
        \filldraw[black] (4,2) circle (3pt);
           \filldraw[blue] (5,2) circle (3pt);
          \filldraw[black] (6,2) circle (3pt);
\end{tikzpicture}\\
\vspace{-3.75 cm}
\includegraphics[width=180pt, trim={0 0 0 0}, clip]{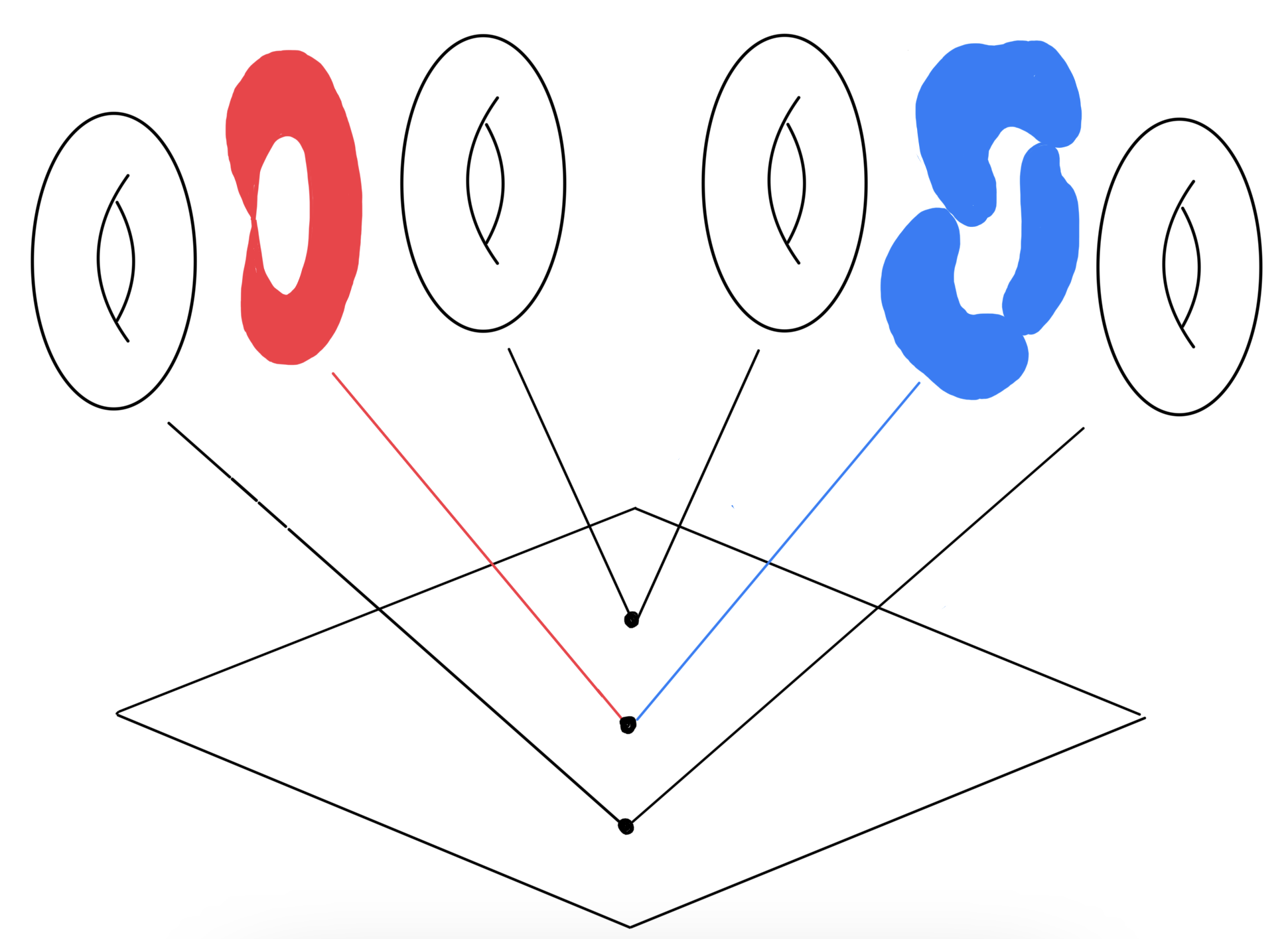}\\
\end{center}
\caption{The number of irreducible components of the Hitchin fibration $\chi(n,d) \colon \MDol(n,d) \to A_n$ is the number of lattice points in the interior of a graphical zonotope translated by a vector depending on the degree. The picture does not correspond to any particular Dolbeault moduli space: the fibers are usually higher dimensional, and the components much more numerous - we shave them off for illustrative purposes.}
\end{figure}
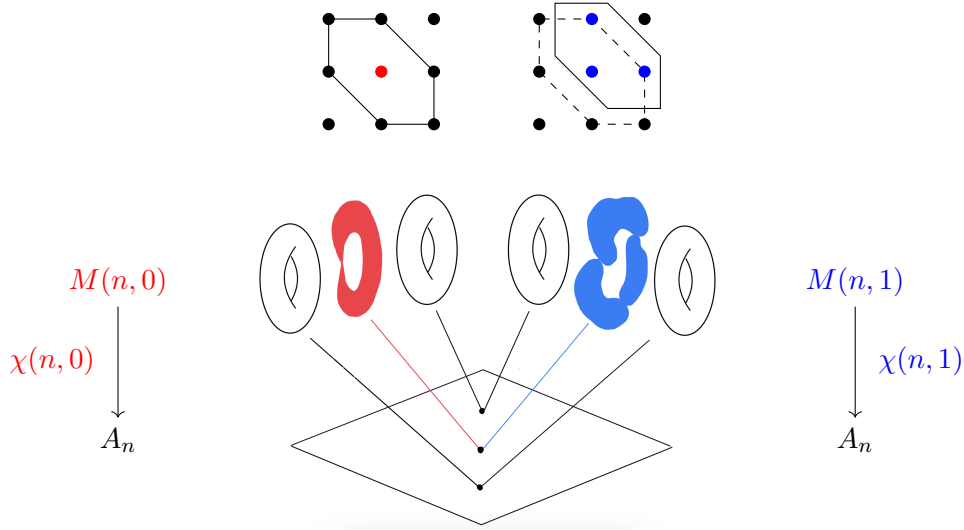
A crucial ingredient is the following combinatorial version of the decomposition theorem.

\begin{thm}[\cref{thm:repr}; Combinatorial decomposition theorem]
  For any graph $\Gamma$ and for any $\Aut(\Gamma)$-invariant vector $\omega \in \QQ^r$, $\mathcal{C}(Z_{\Gamma}+\omega)$ is the rational vector space with basis the set of integral points of the interior of the zonotope $Z_{\Gamma}+\omega$. The following $\Aut(\Gamma)$-representations are isomorphic
    \begin{equation}\label{eq:representationsII}
        \mathcal{C}(Z_{\Gamma}+\omega)= \mathcal{C}(Z_{\Gamma}) \oplus \bigoplus_{\underline{S} \in \Fl_\omega /\Aut(\Gamma)} \Ind_{\Stab (\underline{S})}^{\Aut (\Gamma)} \big( \alpha_{\underline{S}} \otimes \widetilde{H}_{\ell(\underline{S})-3}(\Delta(\overline{\Fl}_{\omega, \geq \underline{S}})) \otimes \mathcal{C}(Z_{\Gamma_{\underline{S}}})\big).
    \end{equation}
\end{thm}
We refer to \S \ref{sec:permutation} for details about the notation. Here we emphasise the formal analogy between \eqref{eq:effectivedt} and \eqref{eq:representationsII}. Indeed, in the same way as $R \chi(n,d)_* \IC(\MDol(n,d), \QQ)|_{A^{\mathrm{red}}_{n}}$ is a direct sum of Ng\^{o} strings  $\mathscr{S}(\mathscr{L}_{\underline{n}}(d))$ indexed by non-$d$-integral partitions $\underline{n}$ of $n$, the number of lattice points in a translated zonotope is the sum of the lattice points contained in the faces of some non-translated faces of the zonotopes corresponding to non-$d$-integral partitions too. 

\subsection{Other work and broad context}\label{sec:otherwork}
\cref{thm:effectivedecomposition} is a refinement of the celebrated the Ng\^{o} support theorem \cite[\S 7]{Ngo2010}. Let $A^{\mathrm{ell}}_n$ be the locus of irreducible characteristic polynomials. In our context, the Ng\^{o} support theorem asserts that $R \chi(n,d)_* \IC(\MDol(n,d), \QQ)|_{A^{\mathrm{ell}}_{n}}$ has full support, and it is independent of the degree $d$. In hindsight, we can now say that $A^{\mathrm{ell}}_n$ is the biggest open set of $A_n$ avoiding summands of $R \chi(n,d)_* \IC(\MDol(n,d), \QQ)$ supported on proper subvarieties of $A_n$ as the degree $d$ varies. If $\gcd(n,d)=1$, \cref{thm:effectivedecomposition} is then the main result of \cite{deCataldoHeinlothMigliorini19}. 

As pointed out in \cite[\S 1.2]{MM2022},  \cref{thm:effectivedecomposition} implies that $R \chi(n,d)_* \IC(\MDol(n,d), \QQ)|_{A^{\mathrm{red}}_{n}}$ depends only on the $\gcd(n,d)$, and so $\IH^*(M(n,d)^{\mathrm{red}}, \QQ)$ does. The independence of the cohomology ring $H^*(M(n,e), \QQ)$ for $\gcd(n,e)=1$ is classical and observed for instance in \cite[Remark 4.8]{Hausel13}; see also \cite[Theorem 1.7]{GWZ20}, \cite[Corollary 1.2]{Mellit2020} and \cite{Yu2022}. The independence of $R \chi(n,e)_* \IC(\MDol(n,e), \QQ)|_{A^{\mathrm{red}}_{n}}$ for $\gcd(n,e)=1$ has already been proved using vanishing cycle techniques in \cite[Theorem 0.5, Remark 4.9]{MaulikShen2020I}; see also \cite[Theorem 1.1, Example 5.18]{KK2021} and the comment \cite[\S 1.5.(2)]{KM2023}. The latest and most refined result \cite[Theorem 0.1]{dCMSZ2021} asserts that a canonical isomorphism $H^*(M(n,e), \QQ) \simeq H^*(M(n,e'), \QQ)$ for $\gcd(n,e)=\gcd(n,e')=1$ preserves simultaneously the ring structure, tautological classes, and perverse filtrations. 

In the singular case, the dependence of $\IH^*(M(n,d), \QQ)$ on the degree has been investigated in \cite[Corollary 1.9]{MM2022}, and settled unconditionally in \cite[Corollary 14.9]{DHSM2022}, by studying the geometry and representation theory of local singularities of $\MDol(n,d)$. The viewpoint of our paper is instead global and combinatorial. Observe also that \cite{DHSM2022} does not prescribe the shape of the summands of $R \chi(n,d)_* \IC(\MDol(n,d), \QQ)$. After \cref{thm:effectivedecomposition}, the main open conjecture is whether there are summands supported on the complement of $A^{\mathrm{red}}_{n}$. As a corollary of \cite{DHSM2022}, it suffices to check it in degree zero, and this is indeed equivalent to the following generalization of \cref{thm:fullsupport}.
\begin{conj}[Full support conjecture]
The complex $R \chi(n,0)_* \IC(\MDol(n,0), \QQ)$ has full support on the whole basis of the Hitchin fibration.
\end{conj} In other words, we ask whether the complex $R\chi(n,0)_*\IC(\MDol(n,0), \QQ)$ is determined by its restriction to any open subsets of $A_{n}$, in particular the locus over which the Hitchin fibration is smooth. A positive answer to the full support conjecture would imply that $\IH^*(\MDol(n,0), \QQ)$ should be regarded as the primitive indecomposable building block of the cohomology of any other Dolbeault moduli space $\MDol(n,d)$.

The key geometric fact used in this paper is the characterization of the fibres of the Hitchin fibration as compactified Jacobians. In fact, the same proof of \cref{thm:effectivedecomposition} works if we replace the Hitchin system with other families of compactifid Jacobians, e.g.\ the Mukai systems on moduli space
of sheaves on K3 or abelian surfaces of Picard rank one; cf \cite[\S 10.3]{MM2022}. In particular, we obtain the following equivariant description of the irreducible components of compactified Jacobians, which is a result of independent interest, and generalizes a classical result by Raynaud \cite[Prop. 8.1.2]{Raynaud1970} and Melo--Rapagnetta--Viviani \cite[Cor. 5.14]{MRV2017} for fine compactified Jacobians. We refer to \S \ref{sec:Hitchincompactified} for details on the terminology.

\begin{defn}
 Let $X= \bigcup_{i \in I} X_i$ be a nodal curve with smooth irreducible components and dual graph $\Gamma(X)$, and $X' \to X$ be a partial normalization of $X$.  
 Note that the restriction $\Pic(X)\to \Pic(X')$ is always surjective, and the kernel is an algebraic torus of dimension $H^1(\Gamma(X), \Gamma(X'), \ZZ)$. We say that $X$ has no tails if the dual graph $\Gamma(X)$ does not collapse onto any proper subgraph, or equivalently $\Pic(X) \not \simeq \Pic(X')$ for any partial normalization $X'$ of $X$.
\end{defn}

\begin{thm} [Irreducible components of compactified Jacobians]
Let $X = \bigcup_{i \in I} X_i$ be a nodal curve with no tails and dual graph $\Gamma$. The top cohomology of the compactified Jacobian $\overline{J}_{\qb}(X)$ is the $\Aut(\Gamma)$-representation $\mathcal{C}(Z_{\Gamma}+\qb)$, i.e.\
\[H^{\mathrm{top}}(\overline{J}_{\qb}(X), \QQ) \simeq_{\Aut(\Gamma)} \mathcal{C}(Z_{\Gamma}+\qb).\]
\end{thm}
\begin{proof}
As $X$ is nodal, rank 1 torsion free sheaves are pushforwards of line bundles from partial normalization of $X$; see for instance \cite[Proposition 10.1]{OS1979} or \cite[Theorem]{Cook93}. Since $X$ has no tails, the locus of pushforwards of line bundles from non-trivial partial normalizations of $X$ is a proper subvariety of $\overline{J}_{\qb}(X)$, or equivalently line bundles are dense in $\overline{J}_{\qb}(X)$. Then the same proof of \cref{prop:compJacobiansI}.\eqref{item} and \eqref{item2} and \cref{thm:geomvscombin} gives the result.
\end{proof}

\subsection{Outline} The paper is divided into two parts: a combinatorial half (from \S\ref{sec:shellability} to \S\ref{sec:permutation}) and an algebro-geometric half (from \S\ref{sec:Hitchincompactified} to \S\ref{sec:mainresults_proof}). 
\begin{itemize}
    \item In \S \ref{sec:notation} we fix the notation of partitions, flats and posets used all over the text. 
    \item In \S \ref{sec:shellability} we prove the LEX-shellability of some posets of flats, see \cref{prop:LEXshell}, and compute the Betti numbers of the order complex $\Delta(\overline{\Pi}_{\omega})$; see \cref{prop:countsphere}. These are combinatorial results of independent interest.
    \item In \S \ref{sec:zonotope} we recall a formula due to Ardila, Beck and McWhirter to compute the number of lattice points in the interior of a translated zonotope; see \cite[Proposition 3.1]{ABM20}. We apply M\"{o}bius inversion to obtain a numerical version of the combinatorial decomposition theorem, namely \cref{thm:countformula}.
    \item In \S \ref{sec:graphicalzonotopes} we specialize the previous count to graphical zonotopes of a complete graph with multiple edges, and realize that the same combinatorial pattern governs both the coefficients of the lattice point count formula and the numbers of spheres of the order complex $\Delta(\overline{\Pi}_{\omega})$; see \cref{claim}.
    \item In \S \ref{sec:permutation} we promote the previous numerical identities to an isomorphism of representations. This is the most technical and delicate part of the paper. We achieve this by comparing characters, which in turns boils down to computing lattice points of fixed loci of automorphisms in graphical zonotopes. In general, the fixed loci are no longer graphical zonotopes, and this is why in \S \ref{sec:zonotope} we work with arbitrary zonotopes.
    \item In \S \ref{sec:Hitchincompactified} we establish the relation between irreducible components of the fibers of the Hitchin fibrations and lattice point counts in graphical zonotopes as described in the Strategy \ref{sec:strategy}.
    \item After an analysis of the stalks of Ng\^{o} strings in \S \ref{sec:Ngo}, we can finally prove \cref{thm:fullsupport} and \cref{thm:combchar} in \S \ref{sec:mainresults_proof}.
\end{itemize}

\subsection{Acknowledgement} We warmly thank Mark de Cataldo, Giulia Sacc\`{a}, Filippo Viviani, 
Giovanni Paolini and Junliang Shen for useful discussions. 
MM was supported by the Max Planck Institute for Mathematics, the University of Michigan, the Institute of Science and Technology Austria, and Universit\'{e} Paris Cit\'{e} and Sorbonne Universit\'{e}, CNRS, IMJ-PRG, F-75013 Paris,
France. This project has received funding from the European Union’s Horizon 2020 research and innovation
programme under the Marie Skłodowska-Curie grant agreement No 101034413. LM and RP were partially supported by PRIN Project 2017YRA3LK "Moduli and Lie theory". RP was partially supported by the Italian Ministry of University and Research (MUR) under the grant FIS 2, project “Chow Ring of Toric Arrangements” (FIS-2023-00808, CUP J53C25001900001). RP is a member of the INdAM research group GNSAGA.

\section{Notation}\label{sec:notation}

\subsection{Partitions of a set or of an integer}

\begin{defn}
 Let $I$ be a finite set. A \textbf{partition} of $I$ of length $\ell(\underline{S})=l$ is a set $\underline{S} \coloneqq \{S_1, \ldots, S_l\}$ of $l$ disjoint subsets $S_i \subseteq I$ such that $I = \bigsqcup^l_{i=1} S_i$. 
 
 Let $\underline{S}'$ be another partition of $I$. We say that $\underline{S} \leq \underline{S}'$ if any $S'_j$ is the union of some $S_i$.
 
 The \textbf{poset of partitions} of $I$ is denoted $(\Pi(I), <)$, or simply by $\Pi$. The minimum $\hat{0}$ of $\Pi$ is the collection of the singletons of $I$, while the maximum $\hat{1}$ is the trivial partition $\{I\}$ consisting of the single block $I$.
\end{defn}
\begin{defn}\label{defn:partition}
Let $n \in \ZZ_{> 0}$. A \textbf{partition} $\underline{n}$ of $n$ of length $\ell(\underline{n})=r$ is a multiset $\{n_1, \ldots, n_r\}$ of positive integers such that $\sum^{r}_{i=1} n_i =n$. 

Let $\underline{m}$ be another partition of $n$ of length $s$. We say that $\underline{n}\geq \underline{m}$, if there exists a partition $\underline{S}$ of length $r$ of the set $\underline{m}$ such that 
\[ \underline{n}=\{n_1, \ldots, n_r\}, \qquad \underline{m}=\{m_1, \ldots, m_s\}, \qquad n_i = \sum_{j \in S_i} m_j.\]
In particular, denote $\underline{n}_i$ the partition $\{m_j : j \in S_i\}$ of $n_i$. 
\end{defn}

\subsection{Flats of a graph}\label{sec:flatsgraph}
Let $\Gamma$ be a graph on the vertex set $V(\Gamma)=[r]$. Denote $E(\Gamma)$ the set of edges of $\Gamma$. 

\begin{defn}\label{def:graph}
A \textbf{flat} $\underline{S} = \{ S_1, \ldots, S_{s} \}$ of $\Gamma$ is a partition of $V(\Gamma)$ with the property that $S_i$ are the vertices of a connected subgraph of $\Gamma$. Denote
by $\Gamma(S_i)$ the subgraph of $\Gamma$ whose vertices are $S_i$, and whose edge set contains all edges whose endpoints lie in $S_i$.  By abuse of notation we continue to call flat the set of edges of the subgraphs $\Gamma(S_i)$, i.e.\ $\bigcup^s_{i=1} E(\Gamma(S_i))$. 
 
 The \textbf{poset of flats} of $\Gamma$, denoted $\Fl=\Fl(\Gamma)$, is the subposet of flats in the partition poset $(\Pi(V(\Gamma)), <)$. 
 \end{defn}
 
 Given a flat $\underline{S}$ of $\Gamma$, we can construct two associated graphs:
 \begin{itemize}
     \item $\Gamma_{\underline{S}} \coloneqq \bigsqcup^s_{i=1} \Gamma(S_i)$ is the subgraph of $\Gamma$ obtained by removing all edges whose vertices do not lie in $\underline{S}$;
     \item $\Gamma^{\underline{S}}$ is the \textbf{contracted (loopless) graph}, i.e.\ the graph obtained from $\Gamma$ by contracting all the edges whose vertices lie in $\underline{S}$ and then deleting all loops.
 \end{itemize}
 
 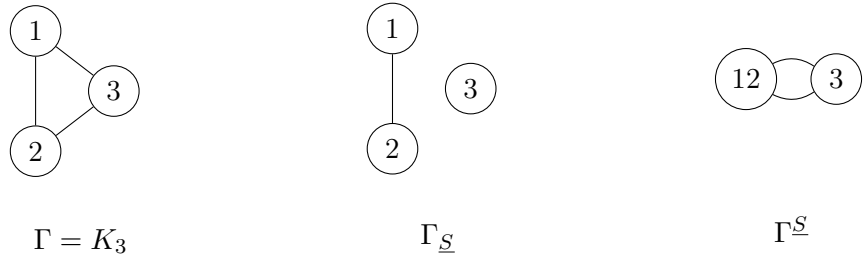
\begin{figure}
     \centering
     \begin{subfigure}[b]{0.3\textwidth}
     \centering
\begin{tikzpicture}[scale=0.8]
    \node[shape=circle,draw=black] (1) at (0,1) {1};
    \node[shape=circle,draw=black] (2) at (1.3,0) {3};
    \node[shape=circle,draw=black] (3) at (0,-1) {2};
    \path [-] (1) edge node[left] {} (2);
    \path [-](2) edge node[left] {} (3);
    \path [-](1) edge node[left] {} (3);
    \node[] () at (0.75,-2.5) {$\Gamma=K_3$};
\end{tikzpicture}
\end{subfigure}
\begin{subfigure}[b]{0.3\textwidth}
\centering
\begin{tikzpicture}[scale=0.8]
    \node[shape=circle,draw=black] (1) at (0,1) {1};
    \node[shape=circle,draw=black] (2) at (1.3,0) {3};
    \node[shape=circle,draw=black] (3) at (0,-1) {2};
    \path [-](1) edge node[left] {} (3);
    \node[] () at (0.75,-2.5) {$\Gamma_{\underline{S}}$};
\end{tikzpicture}
\end{subfigure}
\begin{subfigure}[b]{0.3\textwidth}
\centering
\begin{tikzpicture}[scale=0.8]
    \node[shape=circle,draw=black] (1) at (0,0) {12};
    \node[shape=circle,draw=black] (2) at (1.5,0) {3};
    \node[] at (0,-1.5) {};
    \path [-](1) edge[bend right] node[left] {} (2);
    \path [-](1) edge[bend left] node[left] {} (2);
    \node[] () at (0.75,-2.5) {$\Gamma^{\underline{S}}$};
    \node[] () at (0.75,-2.9) {};
\end{tikzpicture}
\end{subfigure}
\caption{Let $\Gamma$ be a complete graph with 3 vertices. The picture represents the graphs $\Gamma_{\underline{S}}$ and $\Gamma^{\underline{S}}$ associated to the flat $\underline{S}=12|3$.}
\end{figure}

We introduce an arithmetic condition on the flats of a graph that we call $\omega$-integrality.
\begin{defn}\label{defn:partitionsintegral} Let $\omega =(\omega_1, \ldots, \omega_r) \in \RR^r$.
A partition $\underline{S} = \{ S_1, \ldots, S_{l} \}$ of the set $[r]$ is \textbf{$\omega$-integral} if for all  $j=1, \dots, l$ we have 
\[\sum_{i \in S_j} \omega_i \in \Z.\] The poset $\Fl_{\omega}=\Fl_{\omega}(\Gamma) \subset \Fl(\Gamma)$ is the subposet of all non-$\omega$-integral flats of $[r]$.
\end{defn}
\begin{rmk}
Any partition $\underline{S}=\{S_1, \ldots, S_r\}$ of $\underline{m}$ is \textbf{$\omega_{\underline{m}}(d)$-integral} (cf eq.~\eqref{eq:vectoromega}) if and only if the corresponding partition $\underline{n} \coloneqq \{\sum_{j \in S_i} m_j \,|\, j=1, \ldots, r\}$ is $d$-integral.
 
 If the graph $\Gamma$ contains a complete graph on the vertices, then the poset of flats $\Fl(\Gamma)$ is the full poset of partitions of $V(\Gamma)$, denoted $(\Pi(V(\Gamma)), <)$. We simply write $\Fl=\Pi$. In particular, if $\gcd(d,n)=1$, then $\hat{\Fl}_{\omega_{\underline{n}}(d)} =\Fl=\Pi$; see also \cref{defn:terminologyposet}.
 On the other hand, if $d=0$, then $\Fl_{\omega}=\emptyset$.
\end{rmk}

\subsection{Posets}

\begin{defn}\label{defn:terminologyposet}
 A poset is \textbf{bounded} if it contains a least and greatest element $\hat{0}$ and $\hat{1}$. For a bounded poset $\mathcal{P}$ we denote by  $\overline{\mathcal{P}}=\mathcal{P} \setminus \{\hat{0}, \hat{1}\}$ and $\hat{\mathcal{P}}=\overline{\mathcal{P}} \cup \{\hat{0}, \hat{1}\}$. A poset is \textbf{pure} if all maximal chains have the same length. A poset is \textbf{graded} if it is finite, bounded and pure.
\end{defn}

\begin{defn}
 Given a poset $(\mathcal{P}, <)$ and $x,y \in \mathcal{P}$, the symbol $x \lessdot y$ denotes a \textbf{covering relation} and it means that $x<y$ and there is no $z \in \mathcal{P}$ such that $x<z<y$.
\end{defn}

Recall that the \textbf{M\"obius function} $\mu$ of a finite poset $\mathcal{P}$ assigns an integer to each interval $[x,y] \subseteq \mathcal{P}$ such that for any $x,y \in \mathcal{P}$ we have
\[\sum_{x\leq z\leq y}\mu(z,y)=\begin{cases} 1 \quad & \text{if }x=y,\\
0 & \text{otherwise}.
\end{cases}\]

\section{Shellability of certain posets of partitions}\label{sec:shellability}

\subsection{LEX-shellability}
Consider a graph $\Gamma$ on vertex set $[r]$ and a vector $\omega =(\omega_1, \ldots, \omega_r) \in \RR^r$ such that for each connected component $\Gamma'$ of $\Gamma$ the sum $\sum_{i \in \Gamma'} \omega_i \in \ZZ$ is an integer. The main result of this section is the LEX-shellability of the poset $\hat{\Fl}_\omega$ of non-$\omega$-integral partitions; see \cref{defn:partitionsintegral}. The notion of
LEX-shellability was first introduced in \cite{Kozlov1997}; see also \cite[\S 12.2.3]{KozlovBook}.
For the convenience of the reader, we recall here the definition.

\begin{defn}
Let $\mathcal{P}$ be a graded poset, and  $E(\mathcal{P}) = \{(x,y) \in \mathcal{P}\times\mathcal{P} | \, x \lessdot y\}$ be the set of covering relations of $\mathcal{P}$.

An \textbf{edge labelling} of $\mathcal{P}$ is a function $l \colon E(\mathcal{P}) \to \Lambda$, where $\Lambda$ is a poset. For any maximal chain $c =(x=x_0 \lessdot x_1 \lessdot \ldots \lessdot x_{k-1} \lessdot x_k = y)$ we consider the associated $k$-tuple $l(c) \coloneqq (l(x_0 \lessdot x_1), l(x_1 \lessdot x_2), \ldots,  l(x_{k-1} \lessdot x_{k})) \in \Lambda^{k}$. 
A chain $c$ is less than a chain $d$ if $l(c)$ lexicographically precedes $l(d)$.
\end{defn}

\begin{defn}\label{defn:LEXshellability}
    An edge labelling is a \textbf{LEX-labelling} if the following condition holds: for any interval $[x,t]$, any maximal chain $c$ in $[x,t]$, and any $y,z \in c$, with $x<y<z<t$, if $c_{|[x,z]}$ and $c_{|[y,t]}$ are lexicographically least in $[x,z]$ and $[y,t]$ respectively, then $c$ is lexicographical least in $[x,t]$.

    A poset that admits a LEX-labelling is \textbf{LEX-shellable}.
\end{defn}

In \cref{defn:LEXlabel} we define an edge labelling of $\hat{\Fl}_\omega(\Gamma)$, and then prove that it is indeed a LEX-labelling.

\begin{lem}
    If the graph $\Gamma$ is a forest,
    then there exists a unique minimal $\omega$-integral flat of $\Gamma$.
\end{lem}
\begin{proof}
    We may assume that $\Gamma$ is a tree.
    The removal of an edge $a \in E(\Gamma)$ disconnects the tree $\Gamma$ into two connected components $\Gamma_a \sqcup \Gamma^c_a$.
    The flat
    \[\underline{F} \coloneqq \big\{a \in E(\Gamma) \, | \, \sum_{i\in V(\Gamma_a)}\omega_i \not \in \ZZ\big\}\]
    is $\omega$-integral, minimal and unique with these properties, since any $\omega$-integral flat contains the edges $a \in E(\Gamma)$ such that
    $\sum_{i\in V(\Gamma_a)}\omega_i \not \in \ZZ$.
\end{proof}

Let us fix a total order $\prec$ on the edges of the graph $\Gamma$.
For any flat $\underline{S}$ we will define a total order $\prec_{\underline{S}}$ on the edges of $\Gamma^{\underline{S}}$.

We say that an edge $e \in E(\Gamma)$ is smaller than a flat $F \in \Fl(\Gamma)$ if the two endpoints of $e$ belong to the same block of the partition $F$.
Let $T\coloneqq T({\underline{S}})$ be the spanning forest of $\Gamma^{\underline{S}}$ which is lexicographically minimal with respect to the order $\prec$, and $\underline{F} \coloneqq \underline{F}({\underline{S}})$ be the unique minimal $\omega$-integral flat of $\Fl(T({\underline{S}})) \subseteq \Fl_{\geq \underline{S}}$.
We define the total order $\prec_{\underline{S}}$ on the edges of $\Gamma^{\underline{S}}$ such that
\begin{enumerate}
    \item for any $e_1 \in E(T)_{\not \leq \underline{F}}$, $e_2 \in E(T)_{ \leq \underline{F}}$ and $e_3 \in E(\Gamma^{\underline{S}}) \setminus E(T)$ 
    \[e_1 \prec_{\underline{S}} e_2 \prec_{\underline{S}} e_3;\]
    \item for any $e_4, e_5 \in E(T)_{\not \leq \underline{F}}$ (respectively in $E(T)_{ \leq \underline{F}}$ or $E(\Gamma^{\underline{S}}) \setminus E(T)$), 
    \[e_4 \prec_{\underline{S}} e_5 \quad \text{ if and only if } \quad e_4 \prec e_5.\]
\end{enumerate}
\begin{defn}\label{defn:LEXlabel}
The poset $\hat{\Fl}_\omega = \hat{\Fl}_\omega(\Gamma)$ is endowed with the labelling \[l \colon E(\hat{\Fl}_\omega) \to \bigsqcup_{\underline{S} \in \mathcal{\hat{\Fl}_\omega}}(E(\Gamma^{\underline{S}}),  \prec_{\underline{S}})
\]
which sends any covering relation $\underline{S}=\{S_1, S_2, S_3, \ldots, S_s\} \lessdot \underline{T}=\{S_1 \cup S_2, S_3, \ldots, S_s\}$ to the minimum with respect to the order $\prec_{\underline{S}}$ among the edges between the set of vertices  $S_1$ and $S_2$, i.e.\
\[l(\underline{S} \lessdot \underline{T})= \min_{\prec_{\underline{S}}} \{ e \in E(\Gamma^{\underline{S}}) \mid e \in E(\Gamma^{\underline{S}}_{\underline{T}}) \}.\]
\end{defn}

\begin{figure}
\centering
\begin{tikzpicture}[scale=0.8]
    \node[] () at (-0.6,2.6) {\small $\frac{1}{2}$};
    \node[] () at (-0.6,-0.6) {\small $\frac{1}{2}$};
    \node[] () at (2.6,-0.6) {\small $\frac{1}{2}$};
     \node[] () at (2.6,2.6) {\small $\frac{1}{2}$};
    \node[shape=circle,draw=black] (1) at (0,2) {1};
    \node[shape=circle,draw=black] (2) at (0,0) {2};
    \node[shape=circle,draw=black] (3) at (2,0) {3};
    \node[shape=circle,draw=black] (4) at (2,2) {4};
    \path [-] (1) edge node[left] {a} (2);
    \path [-](2) edge node[below] {b} (3);
    \path [-](3) edge node[right] {c} (4);
    \path [-](4) edge node[above] {d} (1);
    \node[] () at (1,-1.5) {$\Gamma=\Gamma^{\hat{0}}$};
    \node[] () at (1,-2.5) {$a \prec b \prec c \prec d$};
    \node[] () at (4.4,2.6) {\small $\frac{1}{2}$};
    \node[] () at (4.4,-0.6) {\small $\frac{1}{2}$};
    \node[] () at (7.6,-0.6) {\small $\frac{1}{2}$};
     \node[] () at (7.6,2.6) {\small $\frac{1}{2}$};
    \node[shape=circle,draw=black] (5) at (5,2) {1};
    \node[shape=circle,draw=black] (6) at (5,0) {2};
    \node[shape=circle,draw=black] (7) at (7,0) {3};
    \node[shape=circle,draw=black] (8) at (7,2) {4};
    \path [-, line width=2 pt] (5) edge node[left] {a} (6);
    \path [-, line width=2 pt](7) edge node[right] {c} (8);
  \path [-](6) edge node[below] {b} (7);
     \node[] () at (6,-1.5) {$\underline{F}(\hat{0})=ac \subset T( \hat{0})=abc$};
    \node[] () at (6,-2.5) {$b \prec_{\hat{0}} a \prec_{\hat{0}} c  \prec_{\hat{0}} d$};
    \node[shape=circle,draw=black, label=above:{\small $1$}] (10) at (11,2) {14};
    \node[shape=circle,draw=black] (12) at (10,0) {2};
    \node[shape=circle,draw=black] (11) at (12,0) {3};
    \path [-](10) edge node[right] {c} (11);
    \path [-](11) edge node[below] {b} (12);
    \path [-](12) edge node[left] {a} (10);
    \node[] () at (11,-1.5) {$\Gamma^{\underline{d}}$};
        \node[shape=circle,draw=black, label=above:{\small $1$}] (13) at (16,2) {14};
    \node[shape=circle,draw=black] (15) at (15,0) {2};
    \node[shape=circle,draw=black] (14) at (17,0) {3};
    \node[] () at (17.6,-0.6) {\small $\frac{1}{2}$};
    \node[] () at (14.4,-0.6) {\small $\frac{1}{2}$};
     \node[] () at (12.6,-0.6) {\small $\frac{1}{2}$};
      \node[] () at (9.4,-0.6) {\small $\frac{1}{2}$};
    \path [-, line width=2 pt](14) edge node[below] {b} (15);
    \path [-](13) edge node[left] {a} (15);
      \node[] () at (16,-1.5) {$F(\underline{d})=b \subset T(\underline{d})=ab$};
      \node[] () at (16,-2.5) {$a \prec_{\underline{d}} b \prec_{\underline{d}} c$};
\end{tikzpicture}
\caption{Let $\Gamma$ be a square whose edges are ordered by $a \prec b \prec c \prec d$ where $a$ and $c$ are opposite edges. Set $\omega = (\frac{1}{2}, \frac{1}{2}, \frac{1}{2}, \frac{1}{2})$. Consider the flat $\hat{0}=1|2|3|4$ and $\underline{d}=14|2|3$. Then $b \prec_{\hat{0}} a \prec_{\hat{0}} c  \prec_{\hat{0}} d$ and $a \prec_{\underline{d}} b \prec_{\underline{d}} c$.}
\end{figure}
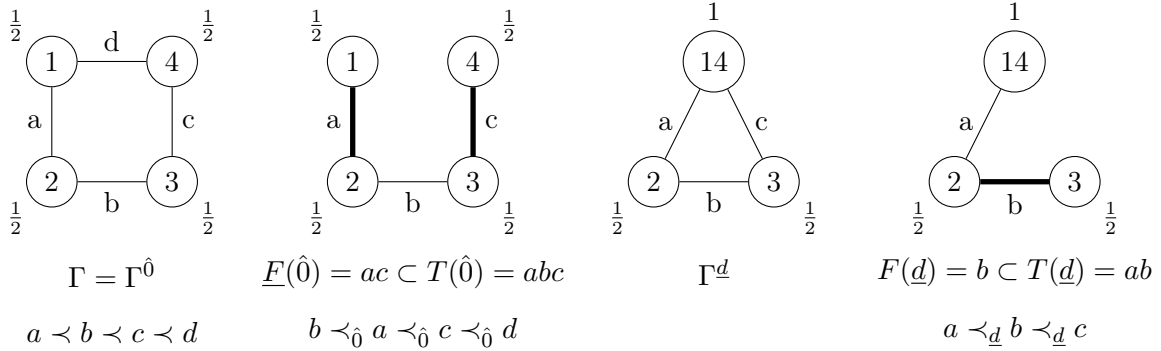

\begin{lem} \label{lemma:labelling_properties}
    Let $\Gamma$ be a graph with $r$ vertices,  
    and $\underline{S} < \underline{U}$ be two flats.
    Then $E(T({\underline{U}})) \subset E(T({\underline{S}}))$. 
    Moreover, if $a \in E(T({\underline{S}}))$ then $E(T({\underline{S} \vee a})) = E(T({\underline{S}})) \setminus \{a\}$ and 
    $\underline{F}({\underline{S} \vee a}) = \underline{F}({\underline{S}}) \vee a$.
\end{lem}
\begin{proof}
    The image of $E(T({\underline{S}}))$ under the contraction $\Gamma^{\underline{S}} \to \Gamma^{\underline{U}}$ is a spanning subgraph. In particular, the minimal spanning forest of $\Gamma^{\underline{U}}$ can be identified with a subset of $E(T({\underline{S}}))$.
    If $a \in E(T({\underline{S}}))$, then $E(T({\underline{S}})) \setminus \{a\}$ is a spanning forest in $\Gamma^{\underline{S} \vee a}$, and the equality $E(T({\underline{S} \vee a})) = E(T({\underline{S}})) \setminus \{a\}$ holds.
    Notice that $\Fl(T({\underline{S} \vee a}))= \Fl(T({\underline{S}}))_{\geq a}$, hence $\underline{F}({\underline{S} \vee a}) \geq  \underline{F}({\underline{S}}) \vee a$.
    On the other hand, $\underline{F}({\underline{S}}) \vee a$ is a $\omega$-integral flat in $\Fl(T({\underline{S} \vee a}))$, therefore by definition $\underline{F}({\underline{S} \vee a}) \leq  \underline{F}({\underline{S}}) \vee a$.
\end{proof}

\begin{thm} \label{prop:LEXshell}
    The labelling $l$ is a LEX-labelling for $\hat{\Fl}_\omega$. In particular, $\hat{\Fl}_\omega$ is LEX-shellable.
\end{thm}
\begin{proof}
    By \cite[Proposition 12.13]{KozlovBook}, it is enough to check the LEX-labelling condition for any $x \lessdot y \lessdot z < t$.
    Let $c$ be a maximal chain in $[x,t]$ such that $c_{|[x,z]}$ and $c_{|[y,t]}$ are lexicographically least.
    Replacing $\Gamma$ with $\Gamma_{t}^x$ we may assume $t=\hat{1}$ and $x=\hat{0}$.
    
    Let $a=l(x \lessdot y)$ and $b = l(y \lessdot z)$, so that $l(c)=(a,b,\dots)$.
    Since $c_{|[y,t]}$ is lexicographically least then $b \in E(T(y)) \subset E(T(x))$.
    We have $x \lessdot (x \vee b) \lessdot z$ and since $c_{|[x,z]}$ is lexicographically least then  $a \prec_x l(x \lessdot x \vee b) \preceq_x b$. In particular $a \in E(T(x))$.

    By contradiction, we assume that the chain $c$ is not lexicographically least in $[x,t]$.
    Hence there exists $d \prec_x a$ and $d \in E(T(x))$.
    Notice that $\{a,b,d\}$ is an independent set of distinct edges in $\Gamma^x$ because contained in $T(x)$, i.e.\ the poset $[x, x \vee a \vee b \vee d]$ is Boolean of rank $3$.
    
    Note that $x \vee a \vee d$ cannot be $\omega$-integral. Otherwise, $\underline{F}(x) \leq x \vee a \vee d$, and $d \leq \underline{F}(x)$ since $y=x \vee a$ is not $\omega$-integral.  
    The conditions $d \in E(T(x))_{\leq \underline{F}}$, $b \in E(T(x))$ and $d \prec_x b$ would imply $b \in E(T(x))_{\leq \underline{F}}$, so $x \vee b\leq \underline{F}(x) \leq x \vee a \vee d$, which is a contradiction since $[x, x \vee a \vee b \vee d]$ is Boolean.
    
    Therefore, $y\vee d= x \vee a \vee d \in \Fl_\omega$ and we are assuming that it is different from $z= x \vee a \vee b$.
    Since $c_{|[y,t]}$ is lexicographically least, we have $b \prec_y l(y \lessdot y \vee d) \preceq_y d$. 
    Moreover, $b \leq \underline{F}(y)= \underline{F}({x \vee a})=\underline{F}(x) \vee a$ if and only if $b \leq \underline{F}(x)$, and the same holds for $d$.
    This fact, together the inequalities $d \prec_x b$ and $b \prec_y d$, gives a contradiction.
\end{proof}

Suppose that $\Gamma$ contains a complete graph on the vertices. Then one can show that the poset $\hat{\Fl}_\omega=\hat{\Pi}_\omega$ associated to the complete graph satisfied the stronger condition of EL-shellability for any $\omega \in \RR^r$; see for instance \cite[\S 2.1.2]{KozlovBook} for a definition of EL-shellability.
\begin{quest}
Is the poset $\hat{\Fl}_\omega$ EL-shellable for any graph $\Gamma$ and any vector $\omega$?
\end{quest}

\subsection{Homotopy type of the order complex of non-\texorpdfstring{$\omega$}{ω}-integral partitions}

Recall that the \emph{order complex} $\Delta(\mathcal{P})$ associated to a poset $\mathcal{P}$ is a simplicial complex whose vertices are in correspondence with the set underlying $\mathcal{P}$, and whose faces corresponds to finite chains in $\mathcal{P}$. 

Recall that $\overline{\Fl}_\omega= {\Fl}_\omega \setminus \{ \hat{0}\}$. We denote by $\widetilde{H}_k(T)$ the reduced $k$-homology group of the topological space $T$ and by $\widetilde{\chi}(T)$ the reduced Euler characteristic.

\begin{defn}
 Given an edge labelling of a poset, a saturated
chain $c$ is \textbf{mediocre} if for any $x,y,z \in c$ such that $x \lessdot y \lessdot z$, the chain
$c|_{[x,z]}$ is not lexicographically least in $[x, z]$.
\end{defn}

\begin{prop}\label{prop:spheres}
Let $r>2$. The order complex $\Delta(\overline{\Fl}_\omega)$ has the homotopy type of the wedge of $(r-3)$-spheres.
The number of mediocre maximal chains in $\Pi_\omega$ equals the number of $(r-3)$-spheres in $\Delta(\overline{\Fl}_\omega)$, i.e.\ the rank of the reduced homology group $ \widetilde{H}_{r-3}(\Delta(\overline{\Fl}_\omega))$.
\end{prop}
\begin{proof}
It follows from \Cref{prop:LEXshell} and the general properties of LEX-shellability \cite[Theorem 12.15]{KozlovBook}.
\end{proof}

\begin{rmk}\label{rmk:baseinduction}
If $r=2$, $\Delta(\overline{\Fl}_\omega)$ is the empty set, but since $\tilde{H}_{-1}(\emptyset)=\mathbb{Z}$ we still have $-\mu(\hat{0},\hat{1})=1=\operatorname{rk} \tilde{H}_{-1}(\emptyset)$.
\end{rmk}

    \begin{prop}\label{prop:countsphere}
   The following equalities hold for $r>1$:
\[\operatorname{rk} \widetilde{H}_{r-3}(\Delta(\overline{\Fl}_\omega))=  |\widetilde{\chi}(\Delta(\overline{\Fl}_\omega))| = (-1)^{r-1}\sum_{\lambda \in \Fl \setminus \Fl_\omega } \mu_\Fl (\hat{0},\lambda)\]
Moreover, if $\Gamma$ contains a complete graph on the vertices, then
\[\operatorname{rk} \widetilde{H}_{r-3}(\Delta(\overline{\Pi}_\omega))=\sum_{\substack{\lambda \vdash [r] \\ \lambda \; \omega\textnormal{-integral}}} (-1)^{\lvert \lambda \rvert-1} \prod_{i=1}^{\lvert \lambda \rvert} (\lvert \lambda_i \rvert-1)!.\]
\end{prop}

 \begin{proof}
The first equality follows from \cref{prop:spheres} and \cref{rmk:baseinduction}.
For the second equality,
we use that $\mu_{\Fl} (\hat{0}, \lambda)=\mu_{\Fl_\omega} (\hat{0}, \lambda)$ for all flats $\lambda \in \Fl_\omega$.
Therefore, we obtain
\begin{align*}
 (-1)^{r-1}\sum_{\lambda \in \Fl \setminus \Fl_\omega } \mu_\Fl (\hat{0},\lambda) &=  (-1)^{r}  \sum_{\lambda \in \Fl_\omega} \mu_{\Fl} (\hat{0},\lambda) \\
&= (-1)^{r-1} \mu_{\hat{\Fl}_\omega}(\hat{0},\hat{1}) \\
&= (-1)^{r-1} \widetilde{\chi}(\Delta(\overline{\Fl}_\omega)).
\end{align*}  
The last equality is the Philip Hall’s theorem \cite[Proposition 3.8.6]{Stanley2012}; the other inequalities follow directly from the definition of M\"obius function. 

In the case of the complete graph, the M\"obius function of the partition poset $\Pi$ is 
\[\mu_{\Pi} (\hat{0},\lambda)= (-1)^{r-\lvert \lambda \rvert} \prod_{i=1}^{\lvert \lambda \rvert} (\lvert \lambda_i \rvert-1)!\]
for $\lambda \vdash [r]$; see \cite[Example 2.9]{Bjorner1980} or \cite[Example 3.10.4]{Stanley2012}.
\end{proof}

\begin{prop}\label{positivityofsum}
The sum in \eqref{eq:coefficient} is positive if and only if $\omega \not \in \mathbb{Z}^r$ and $r>1$.
\end{prop}
\begin{proof}
If $\omega \in \mathbb{Z}^r$, then the vanishing of the sum in \eqref{eq:coefficient} follows from \Cref{rmk:commentformula} (ii).
Alternatively, by definition of M\"obius function and \Cref{prop:countsphere} we have
\[     \sum_{\substack{\lambda\vdash [r] \\ \lambda \; \omega_{\underline{n}}(d)\textnormal{-integral}}} 
    (-1)^{\ell( \lambda )-1} \prod_{i=1}^{\ell( \lambda )} (\lvert \lambda_i \rvert-1)!= (-1)^{r-1}\sum_{\lambda \vdash [r]} \mu_\Pi(\hat{0},\lambda)=0 \qquad \text{ when }r>1. \]
If $\omega \not \in \mathbb{Z}^r$, we must verify the positivity of \eqref{eq:coefficient}. By \cref{prop:spheres} and \cref{prop:countsphere}, it suffices to exhibit a mediocre maximal chain $c$ in $\hat{\Pi}_\omega$.
We construct $c$ inductively by using a greedy algorithm: $\hat{0}$ is in $c$, and for any $x \in c$ we add the unique $y \gtrdot x$ such that $l(y \gtrdot x)$ is maximal with respect to $\prec_x$.

Given $x,t \in c$ with $x<t$ and $x \not \!\!\lessdot t$, we want to show that $c_{|[x,t]}$ is not lexicographically least in $[x,t]$. 
To this end, choose $y,z \in c$ such that $x \lessdot y \lessdot z$. As $\Gamma$ contains a complete graph on the vertices, $\Gamma^x_{z}$ contains a complete subgraph over 3 vertices with edges $\{a,b,d\}$.
In particular, the poset $\mathcal{L}(\Gamma^x_{z})$ contains the 3 atoms $y = x \vee a$, $x \vee b$, $x \vee d$.
If two of them are $\omega$-integral, then the third must be so too, but since $y$ is non-$\omega$ integral, then at least one of $x \vee b$ and $x \vee d$ is non $\omega$-integral too, say $x \vee b$. Then the chain $(c_{|[x,t]} \setminus y) \cup (x \vee b)$ obtained substituting $y$ with $x \vee b$ lexicographically precedes $c_{|[x,t]}$ by construction of the chain $c$.
Therefore the chain $c$ is mediocre, and the sum in \eqref{eq:coefficient} is positive.
\end{proof}

\section{Lattice points in zonotopes}\label{sec:zonotope} In this section, we recall the definition of zonotope, and we provide the key formula for the number of lattice points in the interior of its translates; see \cref{thm:countformula}. 

Recall that the Minkowski sum of several subsets $A, \ldots, B\in \RR^r$ is the locus of sums of vectors that belong to these subsets \[A+\ldots+ B := \{a +\ldots+ b \, |\,
a \in A, \ldots , b \in B\}.\] If $A, \ldots , B$ are convex and/or lattice polytopes, then so is their Minkowski sum. 

\begin{defn}
 Let $U$ be a finite set of points in $\RR^r \setminus \{0\}$. The \textbf{zonotope} generated by $U$ is the Minkowski sum of line segments $[0,u]$ with $u \in U$, i.e.\
 \[Z(U) \coloneqq \sum_{u \in U}[0,u].\]
\end{defn} 

The zonotope $Z(U)$ decomposes into the disjoint union of translates of
the half-open parallelepipeds
$\sum_{u \in W} [0,u)$ 
spanned by the linearly independent subsets $W \subseteq U$; see \cite[Theorem 54]{Shephard1974}. This observation allows us to compute the number of lattice points in the interior of a translate of $Z(U)$ in \cref{thm:Ehrhart_qp}, as proved in \cite[Proposition 3.1]{ABM20}. 

To state the result, we first adapt the definitions of \S \ref{sec:flatsgraph} to the present context, and recall the definition of Ehrhart quasi-polynomial.
\begin{defn}\label{defn:newflat}
 A \textbf{flat} of $U$ of rank $k$ is a maximal subset of $U$ spanning a vector subspace of rank $k$. The poset of flats of $U$ is denoted $\Fl=\Fl(U)$.
 
 Let $\ZZ^r \subset \RR^r$ be the standard lattice, and $\omega$ be a vector in $\RR^r$. Then a flat $S \subseteq U$ is \textbf{$\omega$-integral} if the the affine subspace obtained by translating the vector space spanned by $S$ by the translation vector $\omega$ contains lattice points, i.e.\ $(\omega + \langle S \rangle) \cap \mathbb{Z}^{r} \neq \emptyset$. The poset of $\omega$-integral flats is denoted $\Fl_{\omega}=\Fl_{\omega}(U)$.
 
 Given a set $W$ of independent vectors in $\ZZ^{r} \subset \RR^r$, the \textbf{relative volume} of the parallelepiped generated by $W$, denoted $\rvol (W)$, is the volume of the parallelepiped in $\RR^r$ normalized by the volume of a basis of the lattice $\langle W \rangle \cap \ZZ^{r}$, where $\langle W \rangle$ is the linear span of $W$ in $\mathbb{Q}^r$. Alternatively, it is the index of the sublattice generated by $W$ in $\langle W \rangle \cap \ZZ^{r}$.\end{defn}
 
\begin{defn}
The \textbf{Ehrhart quasi-polynomial}
\[ L_P(t) \coloneqq \left| tP \cap \ZZ^{r} \right|\] counts the number of lattice points in the dilated polytope $tP$.
\end{defn}  

\begin{thm}
\label{thm:Ehrhart_qp} \emph{\cite[Proposition 3.1]{ABM20}}
    Let $U \subset \mathbb{Z}^r$ be a finite set of integer vectors, $\omega \in \mathbb{Q}^r$ a rational vector, and $Z(U)$ be the zonotope generated by $U$.
    Then the Ehrhart quasi-polynomial of $Z(U)+\omega$ is
    \[ L_{Z(U)+\omega}(t)= \sum_{\substack{W \subseteq U \\ W \textnormal{ independent}}} \delta_{W,\omega}(t) \rvol (W) t^{\lvert W \rvert}, \]
    where
    \[ \delta_{W,\omega}(t)=\begin{cases}
        1 & \textnormal{if } (t\omega + \langle W \rangle) \cap \mathbb{Z}^{r} \neq \emptyset, \\
        0 & \textnormal{otherwise.}
    \end{cases}\]
\end{thm}

Given any polytope $P$ we denote by $\mathrm{Int}(P)$ the interior of $P$ in its linear span. The number of lattice points contained in the interior of the polytope $P$ is denoted
\begin{align*}
    C(P) \coloneqq \left|\mathrm{Int}(P) \cap \ZZ^{r} \right|.
\end{align*} By Ehrhart-Macdonald reciprocity \cite{Macdonald1971}, $C(P)$ is a value of the Ehrhart quasi-polynomial
$L_P(t)$ up to sign, namely 
\[ C(P) = (-1)^{\dim(P)}L_{P}(-1).\]
If $P$ is a translated zonotope $Z(U)+\omega$, then M\"obius inversion applied to \cref{thm:Ehrhart_qp} gives the combinatorial identity \eqref{eq:Z(U)} and \eqref{eq:claim}. 

\begin{thm}\label{thm:countformula}
    Let $U \subset \mathbb{Z}^r$ be a finite set of integer vectors, $\omega \in \mathbb{Q}^r$ a rational vector, and $Z(U)$ be the zonotope generated by $U$. Then we have
    \begin{equation}\label{eq:Z(U)}
        C(Z(U)+\omega) = \sum_{{S} \in \Fl} \left( \sum_{{T} \in \Fl_{ \geq {S}} \setminus \Fl_\omega} \mu_\Fl({S}, {T}) \right) C(Z({S})).
    \end{equation} 
\end{thm}
\begin{proof}
From \Cref{thm:Ehrhart_qp} we write
\[ L_{Z(U) +\omega}(t) = \sum_{S \in \Fl} \delta_{S,\omega} (t) \left(
\sum_{\substack{W\subseteq S \textnormal{ independent}\\ \langle W \rangle=\langle S \rangle
}} \rvol(W)\right) t^{\operatorname{rk}(S)}\]
where 
    \[ \delta_{{S},\omega}(t)=\begin{cases}
        1 & \textnormal{if } (t\omega + \langle S \rangle) \cap \mathbb{Z}^{r} \neq \emptyset, \\
        0 & \textnormal{otherwise,}
    \end{cases}\]
because the function $\delta_{W,\omega}(t)$ does not depend on $W$ but only on the flat $\langle W \rangle=\langle S \rangle$.

Let us define \[Q_{{S}}(t) \coloneqq t^{\operatorname{rk}(S)} \sum_{\substack{W\subseteq S \textnormal{ independent}\\ \langle W \rangle=\langle S \rangle
}} \rvol(W)\] and notice that it is independent of $\omega$.
For any flat ${T} \in \Fl$ we write
\[ L_{Z(T) +\omega}(t) = \sum_{{S} \leq {T}} \delta_{{S},\omega} (t) Q_{{S}}(t).\]
We specialize to the case $\omega=0$ so that $\delta_{{S},0}=1$, and we apply the M\"obius inversion formula to obtain
\[ Q_{{T}}(t) = \sum_{{S} \leq {T}} \mu_\Fl({S}, {T}) L_{Z(S)}(t).\]
Therefore, for any $\omega \in \RR^r$ we have 
\begin{align*}
    L_{Z(U) +\omega}(t) & = \sum_{{T} \in \Fl} \delta_{{T},\omega} (t) \left(\sum_{{S} \leq {T}} \mu_\Fl({S}, {T}) L_{Z(S)}(t) \right)  = \sum_{{S} \in \Fl} \left( \sum_{{T} \geq {S}} \mu_\Fl({S}, {T}) \delta_{{T},\omega} (t) \right) L_{Z(S)}(t).
\end{align*}
Specializing at $t=-1$ we have $\delta_{{T},\omega}(-1)=1$ if and only if ${T}$ is $\omega$-integral, and we obtain
\begin{align*}
    C(Z(U)+\omega) & = \sum_{{S} \in \Fl} \left( \sum_{ \substack{{T} \geq {S}\\ {T} \, \omega\textnormal{-integral}}} \mu_{\Fl}({S}, {T})\right) C(Z(S)). 
\end{align*}
This completes the proof.
\end{proof}

\section{Graphical zonotopes}\label{sec:graphicalzonotopes}
In this section, we specialize the lattice point count formula in \cref{thm:countformula} to the case of graphical zonotopes of a complete graph with multiple edges; see \cref{claim}.

\subsection{Definition and equations of graphical zonotopes}
Denote by $x_1, \ldots, x_r$ the real coordinates of the affine space $\RR^{r}$ with respect to an integer basis $v_1, \ldots, v_r$ of the standard lattice $\ZZ^{r} \subset \RR^{r}$. 

Let $\Gamma$ be a graph on the vertex set  $I \subset [r]$.
Denote $E(\Gamma)$ the set of edges of $\Gamma$.
To every $e \in E(\Gamma)$ one can associate the pair of endpoints $\{s(e),t(e)\}$. For any pair of vertices $J \subseteq I$, set \[y_J \coloneqq |\{e \in E(\Gamma)\,|\, \{s(e),t(e)\} = J\}|.\] 

\begin{defn}\label{def:graphical}
Denote $\Delta_K$ the convex hull of the points $v_k$ for all $k \in K \subseteq [r]$. The \textbf{graphical zonotope} $Z_{\Gamma}$ of $\Gamma$ in $\RR^r$ is the Minkowski sum of the line segments $[s(e), t(e)]$, with $e \in E(\Gamma)$, i.e.\
\[ Z_{\Gamma} \coloneqq \sum_{e \in E(\Gamma)} [s(e), t(e)]=\sum_{J \subseteq I, |J|=2} y_J \Delta_{J}.\]
 In case  $I = \{i\}$, we set $Z_{\Gamma}$ to be $v_i$.
\end{defn}

Equivalently, for $|I|>1$, the zonotope $Z_{\Gamma}$ is the lattice polytope in $\RR^{I} \coloneqq \mathrm{Span}\{v_i : i \in I\} \subseteq \RR^r$ defined by the system of equality and inequalities  
\begin{equation}\label{eq:inequequ}
  \sum_{i \in I} x_i = z_{I}, \quad \sum_{i \in K} x_i \geq z_K \quad \forall K \subset I,
\end{equation} 
where $z_K \coloneqq \sum_{J \subseteq K} y_J$; see \cite[Proposition 6.3]{Postnikov09}.

\begin{figure}[b]
\centering
\vspace{0.1 cm}
\begin{tikzpicture}[join=round, scale=0.5]
    \tikzstyle{conefill} = [fill=blue!20,fill opacity=0.4,text opacity=1]
    \filldraw[conefill](-12-4,0,-4-4-4)--node[midway] {$12|3$}(-10-4,-2,-4-4-4)--node[midway] {$1|23$}(-10-4,-4,-2-4-4)--node[midway] {$13|2$}(-12-4,-4,0-4-4)--node[midway] {$12|3$}(-14-4,-2,0-4-4)--node[midway] {$1|23$}(-14-4,0,-2-4-4)--node[midway] {$13|2$} cycle;
    \filldraw[conefill](2,0,4)--(0,2,4)--(0,4,2)--(2,4,0)--(4,2,0)--(4,0,2)--cycle;
    \filldraw[conefill](4,0,2)--(2,0,4)--(2,0,6)--(4,0,6)--(6,0,4)--(6,0,2)--cycle;
    \filldraw[conefill](0,4,2)--(0,2,4)--(0,2,6)--(0,4,6)--(0,6,4)--(0,6,2)--cycle;
    \filldraw[conefill](6,2,0)--(6,4,0)--(4,6,0)--(2,6,0)--(2,4,0)--(4,2,0)--cycle;
    \filldraw[conefill](2,0,6)--(0,2,6)--(0,2,4)--(2,0,4)--cycle;
    \filldraw[conefill](6,0,2)--(6,2,0)--(4,2,0)--(4,0,2)--cycle;
    \filldraw[conefill](0,6,2)--(2,6,0)--(2,4,0)--(0,4,2)--cycle;
    \filldraw[conefill](4,2,6)--(2,4,6)--(2,6,4)--(4,6,2)--(6,4,2)--(6,2,4)--cycle;
    \filldraw[conefill](6,0,4)--(6,2,4)--(6,4,2)--(6,4,0)--(6,2,0)--(6,0,2)--cycle;
     \filldraw[conefill](0,6,4)--(2,6,4)--(4,6,2)--(4,6,0)--(2,6,0)--(0,6,2)--cycle;
     \filldraw[conefill](2,0,6)--(0,2,6)--(0,4,6)--(2,4,6)--(4,2,6)--(4,0,6)--cycle;
     \filldraw[conefill](6,0,4)--(4,0,6)--(4,2,6)--(6,2,4)--cycle;
     \filldraw[conefill](0,6,4)--(0,4,6)--(2,4,6)--(2,6,4)--cycle;
     \filldraw[conefill](6,4,2)--(4,6,2)--(4,6,0)--(6,4,0)--cycle;
     \node () at (-12,2){$123$};
\end{tikzpicture}
\caption{Graphical zonotopes of the complete graph $\Gamma$ with $r=3$ or $4$ vertices respectively. Each face of $Z_{\Gamma}$ is a translate of $Z_{\Gamma_{\underline{S}}}$ for some flat $\underline{S} \in \Fl(\Gamma)$.}
\end{figure}
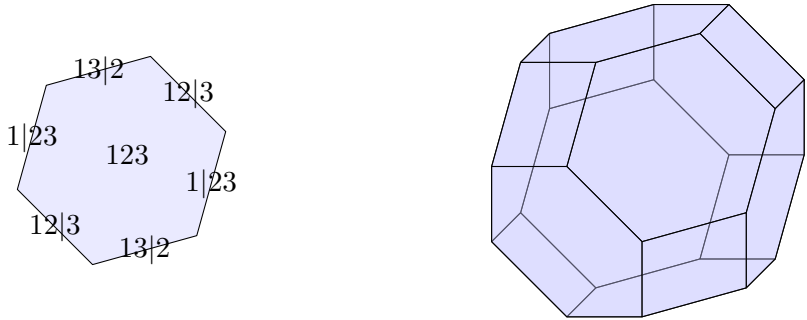
\subsection{Flats and faces}\label{rmk:2zonoto} Fix an orientation for the graph $\Gamma$. The finite set of vectors $U_{\Gamma} \coloneqq \{t(e) - s(e) \colon e \subseteq E(\Gamma)\} \subset \RR^r$ is the image of the basis $E(\Gamma)$ of $\ZZ[E(\Gamma)]$ via the boundary map
\begin{align*}
  \delta \colon \ZZ[E(\Gamma)] \to \ZZ[V(\Gamma)],\\
  e \mapsto t(e)-s(e).
\end{align*}
By construction, the graphical zonotope $Z_{\Gamma}$ is isomorphic to the zonotope $Z(U_{\Gamma}) \subseteq \RR[V(\Gamma)]$ via a translation by an integral vector (depending on the orientation). 

Since $\ker \delta = H^1(\Gamma, \ZZ)$ and $\coker \delta = H^0(\Gamma, \ZZ)$, there exist  bijective correspondences between
\begin{align*}
    \text{independent subsets of }U_{\Gamma} & \qquad \Longleftrightarrow \qquad \text{forests of } \Gamma, \text{ i.e. graphs } F \text{ with }H^1(F, \ZZ)=0,\\
    \Fl(U_{\Gamma}) \coloneqq \text{flats of }U_{\Gamma} & \qquad \Longleftrightarrow \qquad \Fl(\Gamma) \coloneqq \text{flats of }\Gamma;
\end{align*}
see Definitions \ref{def:graph} and  \ref{defn:newflat}. In particular, the latter correspondence assigns  the partition $\underline{S} \in \Fl(\Gamma)$ to the flat in $\Fl(U_{\Gamma})$ cut by the linear subspace 
\begin{equation}\label{eq:flatincoord}
    W_{\underline{S}} \coloneqq \big\{\underline{x} \in \mathbb{R}^r \mid \sum_{j \in S_i} x_j =0 \textnormal{ for all } S_i \in \underline{S} \big \}.
\end{equation} 
Note that the affine span of any face of the zonotope $Z(U)$ is an integral translation of the linear space spanned by a flat $S \in \Fl(U)$, and the given face is isomorphic to $Z(S)$, up to an integral translation. In particular, any face of a graphical zonotope $Z_{\Gamma}$ is isomorphic to 
\begin{equation} \label{eq:faceexpl}
    Z_{\Gamma_{\underline{S}}} = Z_{\Gamma(S_1)} \times \ldots \times Z_{\Gamma(S_{l})}
\end{equation}
for some flat $\underline{S}=\{S_1, \ldots, S_l\} \in \Fl(\Gamma)$ via an integral translation.
Viceversa, a flat $\underline{S} \in \Fl(\Gamma)$ with an acylic orientation of the contracted graph $\Gamma^{\underline{S}}$ determines uniquely a face of $Z_{\Gamma}$; see \cite[Proposition 1.3]{PPP2021}.

\begin{rmk} The factorization \eqref{eq:faceexpl} of the faces of $Z_{\Gamma}$ explains the occurrence of the factor $\prod^{\ell(\underline{S})}_{j=1}C(Z_{\Gamma(S_j)})$ in \eqref{eq:claim}, and it is a combinatorial analogue of the splitting of the normal slice in \eqref{eq:slice}.
\end{rmk}

\subsection{Lattice point counts in graphical zonotopes} We now specialize \cref{thm:countformula} to the case of the graphical zonotopes of a complete graph with multiple edges.
\begin{cor} \label{claim}
If $\Gamma$ contains a complete graph on the vertices, then for any $\omega \in \QQ^r$ we have
\begin{equation}
\label{eq:claim} 
    C(Z_{\Gamma}+\omega) = \sum_{\underline{S} \vdash [r]}
    \left( \sum_{\substack{\lambda \vdash \underline{S} \\ \lambda \; \omega_{\underline{S}}\textnormal{-integral}}} (-1)^{\ell(\lambda)-1} \prod_{i=1}^{\ell(\lambda)} (\lvert \lambda_i \rvert-1)! \right) 
    \prod^{\ell(\underline{S})}_{j=1}C(Z_{\Gamma(S_j)})
\end{equation}
where $\omega_{\underline{S}}=(\omega_{S_1}, \dots, \omega_{\ell(\underline{S})}) \in \RR^{\ell(\underline{S})}$, and $\omega_{S_i}= \sum_{j \in S_i} \omega_j$.
Moreover, if $\sum_{i=1}^r \omega_i \in \ZZ$, then
\begin{equation}
C(Z_{\Gamma}+\omega) = C(Z_{\Gamma}) + \sum_{\underline{S} \in \Pi_\omega} \operatorname{rk} \widetilde{H}_{\ell(\underline{S})-3}(\Delta(\Pi_{\omega_{\underline{S}}})) \prod^{\ell(\underline{S})}_{j=1} C(Z_{\Gamma(S_j)}).
\end{equation}
\end{cor}
\begin{proof}
It follows from \cref{thm:countformula}, \cref{rmk:2zonoto} and \cref{prop:countsphere}.
\end{proof}

\begin{rmk}\label{rmk:commentformula} $ $
\begin{enumerate}[label=(\roman*)]
\item \label{item:formulageneral} If $\omega\in \Delta_{[r]}$, then the only $\omega$-integral partition is the trivial one. Hence
\[C(Z_{\Gamma}+ \omega)=\sum_{\underline{S} \vdash [r]} (\ell(\underline{S})-1)! \prod_{j=1}^{ \ell(\underline{S})} C(Z_{\Gamma(S_j)}).\]
\item If $\omega\in \Z^r$, then any partition is $\omega$-integral. Therefore, for any partition $\underline{S} \vdash [r]$, the coefficient of $\prod^{\ell(\underline{S}) }_{j=1}C(\Gamma(S_j))$ in \eqref{eq:claim} is
\[
\sum_{\lambda \vdash \underline{S}} (-1)^{\ell(\lambda)-1} \prod_{i=1}^{\ell(\lambda)} (\lvert \lambda_i \rvert-1)! = \sum_{\sigma \in \mathfrak{S}_{k}} (-1)^{k-1-\ell(\sigma)},
\]
where $\ell(\sigma)$ is the length of the permutation $\sigma \in \mathfrak{S}_{k}$, i.e.\ the smallest number of transpositions whose product gives $\sigma$.
This coefficient vanishes for $\ell(\underline{S})>1$, so we obtain the trivial identity
\[ C(Z_{\Gamma}+\omega)=C(Z_{\Gamma([r])})=C(Z_{\Gamma}).\]
\item Let $\omega=d (\frac{n_1}{n}, \frac{n_2}{n}, \dots, \frac{n_r}{n})$ 
and $n=\sum_{i=1}^r n_i$.
 \cref{claim} gives in particular that 
\[ C(Z_{\Gamma}+\omega)= C\bigg(Z_{\Gamma}+  \gcd(d,n)\bigg(\frac{n_1}{n}, \frac{n_2}{n}, \dots, \frac{n_r}{n}\bigg)\bigg).\]
In other words, $C(Z_{\Gamma}+\omega)$ depends only on $\gcd(d,n)$ and $n_i$, for $i =1,\dots,r$.
\end{enumerate}
\end{rmk}
\begin{exa}\label{exa:4}
Let $K({n,e})$ be the graph on $n$-vertices with $e$ edges between any two distinct vertices. 
The graph $K({n,1})$ is the complete graph $K_n$. For any partition $\underline{n}\vdash n$, we denoted by $b(\underline{n})$ the number of partitions $\underline{S}=\{S_{1}, \ldots, S_{k}\}$ of the set $[n]$ whose partition type is $\underline{n}$, i.e.\ $\underline{n}=\{|S_{1}|, \ldots, |S_{k}|\}$.\footnote{The integer $b(\underline{n})$ has a clear geometric meaning: the stratum $S_{\underline{n}}$ defined in \S \ref{sec:Ngo} is not normal  in general, and $b(\underline{n})$ is the number of branches of $S_{\underline{n}}$ passing through the general point of $S_{1|\ldots|1}$.}  Set 
\[l(\underline{n}, \omega)\coloneqq \sum_{\substack{\lambda \vdash \underline{n} \\ \lambda \; \omega_{\underline{n}}\textnormal{-integral}}} (-1)^{\ell(\lambda)-1} \prod_{i=1}^{\ell(\lambda)} (\lvert \lambda_i \rvert-1)!.\]
Hence, we can rewrite \eqref{eq:claim} as follows \begin{equation}\label{eq:completegraphformula}
    C(Z_{K(n,e)}+\omega) = \sum_{\underline{n}=\{n_1, \ldots, n_k\} \vdash n}
    b(\underline{n}) \cdot l(\underline{n}, \omega) \cdot 
    \prod^{k}_{j=1}C(Z_{K(n_j, e)}).
\end{equation}
    
Fix for instance $n=4$, and let $\omega(d)=d\big(\frac{1}{4}, \frac{1}{4}, \frac{1}{4}, \frac{1}{4}\big) \in \RR^r$. We collect the values of $b(\underline{n})$ and $l(\underline{n}, \omega(d))$ in the Table \ref{table2}. \begin{table}[H]
\centering
\renewcommand{\arraystretch}{1}
   \begin{tabular}{c|ccccc}
   Type of $\underline{n} \vdash 4$ & \{4\} & \{3,1\} & \{2,2\} & \{2,1,1\} & \{1,1,1,1\}\\ \hline
 $b(\underline{n})$ & 1 & 4 & 3 & 6 & 1 \\
  $l(\underline{n}, \omega(0))$ & 1 & 0 & 0 & 0 & 0 \\
  $l(\underline{n}, \omega(1))$ & 1 & 1 & 1 & 2 & 6 \\
  $l(\underline{n}, \omega(2))$ & 1 & 1 & 0 & 1 & 3
    \end{tabular}
    \vspace{0.4 cm}
    \caption{Values of $b(\underline{n})$ and $l(\underline{n}, \omega(d))$ for $n=4$. The third row equals $(\ell(\underline{n})-1)!$, as explained in 
\cref{rmk:commentformula}.\ref{item:formulageneral}. To fill the firth row, note that the only $\omega(2)$-integral partitions are $12|34$, $13|24$, $14|23$, $1234$. The coefficients $l(\underline{n}, \omega(d))$ have been computed also in \cite[Table 1]{MM2022} using a different algebro-geometric approach.} \vspace{-1 cm}  \label{table2}
    \end{table}
    To determine $C(Z_{K(n,e)}+\omega)$, it is then sufficient to count the number of lattice points in the interior of the non-translated zonotope $Z_{K(m,e)}$ for $m \leq  4$. To this end, observe that the Tutte polynomial of the complete graph $K_{m+1}$ can be obtained recursively from the following formula \cite{Gessel1995,Pak}
\[T_m(x,y)= \sum_{k=1}^m \binom{m-1}{k-1} (x+y+y^2+\dots +y^{k-1})T_{k-1}(1,y) T_{m-k}(x,y) \] 

\begin{table}[p]
\vspace{1.5 cm}
    \centering
    \begin{tabular}{c|c|c|c}
        $m$ & $T_{m-1}(x,y)$ & $E_{K_{m}}$ & $C(K({m,e}))$ \\ \hline
        $1$ & $1$ & $1$ & $1$ \\
        $2$ & $x$ & $q+1$ & $e-1$ \\
        $3$ & $x^2+x+y$ & $3q^2+3q+1$ & $3e^2-3e+1$ \\
        $4$ & $x^3+3x^2+2x+4xy+2y+3y^2+y^3$ & $1+6q+15q^2+16q^3$ & $16e^3-15e^2+6e-1$
    \end{tabular}
    \vspace{0,5 cm}
    \caption{First values of Tutte polynomials, Ehrhart polynomials for the complete graphs, and $C(Z_{K({m,e})})$.}
    \label{tab:points_count}
\end{table}
\begin{figure}[p]
\begin{subfigure}[b]{0.3\textwidth}
\centering
\begin{tikzpicture}[join=round, scale=0.5]
    \tikzstyle{conefill} = [fill=blue!20,fill opacity=0.4]
    \filldraw[conefill](2,0,4)--(0,2,4)--(0,4,2)--(2,4,0)--(4,2,0)--(4,0,2)--cycle;
    \filldraw[conefill](4,0,2)--(2,0,4)--(2,0,6)--(4,0,6)--(6,0,4)--(6,0,2)--cycle;
    \filldraw[conefill](0,4,2)--(0,2,4)--(0,2,6)--(0,4,6)--(0,6,4)--(0,6,2)--cycle;
    \filldraw[conefill](6,2,0)--(6,4,0)--(4,6,0)--(2,6,0)--(2,4,0)--(4,2,0)--cycle;
    \filldraw[conefill](2,0,6)--(0,2,6)--(0,2,4)--(2,0,4)--cycle;
    \filldraw[conefill](6,0,2)--(6,2,0)--(4,2,0)--(4,0,2)--cycle;
    \filldraw[conefill](0,6,2)--(2,6,0)--(2,4,0)--(0,4,2)--cycle;
    \filldraw[conefill](4,2,6)--(2,4,6)--(2,6,4)--(4,6,2)--(6,4,2)--(6,2,4)--cycle;
    \filldraw[conefill](6,0,4)--(6,2,4)--(6,4,2)--(6,4,0)--(6,2,0)--(6,0,2)--cycle;
     \filldraw[conefill](0,6,4)--(2,6,4)--(4,6,2)--(4,6,0)--(2,6,0)--(0,6,2)--cycle;
     \filldraw[conefill](2,0,6)--(0,2,6)--(0,4,6)--(2,4,6)--(4,2,6)--(4,0,6)--cycle;
     \filldraw[conefill](6,0,4)--(4,0,6)--(4,2,6)--(6,2,4)--cycle;
     \filldraw[conefill](0,6,4)--(0,4,6)--(2,4,6)--(2,6,4)--cycle;
     \filldraw[conefill](6,4,2)--(4,6,2)--(4,6,0)--(6,4,0)--cycle;
     
     \node at (2,2,4) [circle,fill=black, scale=0.7] {};
     \node at (2,4,4) [circle,fill=black, scale=0.7] {};
     \node at (4,2,4) [circle,fill=black, scale=0.7] {};
     \node at (4,4,2) [circle,fill=black, scale=0.7] {};
     \node at (2,4,2) [circle,fill=black, scale=0.7] {};
     \node at (4,2,2) [circle,fill=black, scale=0.7] {};
     \node at (4,4,4) [circle, draw, scale=0.7] {};
     \node at (2,2,6) [circle, draw, scale=0.7] {};
     \node at (2,6,2) [circle, draw, scale=0.7] {};
     \node at (6,2,2) [circle, draw, scale=0.7] {};
     
    \node at (4,2,6) [circle, draw, scale=0.7] {};
     \node at (2,4,6) [circle, draw, scale=0.7] {};
     \node at (2,6,4) [circle, draw, scale=0.7] {};
     \node at (4,6,2) [circle, draw, scale=0.7] {};
     \node at (6,4,2) [circle, draw, scale=0.7] {};
     \node at (6,2,4) [circle, draw, scale=0.7] {};
\end{tikzpicture}
\vspace{0.3 cm}
\begin{gather*}
    \omega(0)=(0, 0, 0,0)\\
    C(Z_{K_4}+\omega(0))  =6
\end{gather*}
\end{subfigure}
\begin{subfigure}[b]{0.35\textwidth}
\centering
\begin{tikzpicture}[join=round, scale=0.5]
    \tikzstyle{conefill} = [fill=blue!20,fill opacity=0.4]
    \tikzstyle{conefilla} = [opacity=0.4, fill opacity=0]
    \filldraw[conefilla](2,0,4)--(0,2,4)--(0,4,2)--(2,4,0)--(4,2,0)--(4,0,2)--cycle;
    \filldraw[conefilla](4,0,2)--(2,0,4)--(2,0,6)--(4,0,6)--(6,0,4)--(6,0,2)--cycle;
    \filldraw[conefilla](0,4,2)--(0,2,4)--(0,2,6)--(0,4,6)--(0,6,4)--(0,6,2)--cycle;
    \filldraw[conefilla](6,2,0)--(6,4,0)--(4,6,0)--(2,6,0)--(2,4,0)--(4,2,0)--cycle;
    \filldraw[conefilla](2,0,6)--(0,2,6)--(0,2,4)--(2,0,4)--cycle;
    \filldraw[conefilla](6,0,2)--(6,2,0)--(4,2,0)--(4,0,2)--cycle;
    \filldraw[conefilla](0,6,2)--(2,6,0)--(2,4,0)--(0,4,2)--cycle;
    \filldraw[conefilla](4,2,6)--(2,4,6)--(2,6,4)--(4,6,2)--(6,4,2)--(6,2,4)--cycle;
    \filldraw[conefilla](6,0,4)--(6,2,4)--(6,4,2)--(6,4,0)--(6,2,0)--(6,0,2)--cycle;
     \filldraw[conefilla](0,6,4)--(2,6,4)--(4,6,2)--(4,6,0)--(2,6,0)--(0,6,2)--cycle;
     \filldraw[conefilla](2,0,6)--(0,2,6)--(0,4,6)--(2,4,6)--(4,2,6)--(4,0,6)--cycle;
     \filldraw[conefilla](6,0,4)--(4,0,6)--(4,2,6)--(6,2,4)--cycle;
     \filldraw[conefilla](0,6,4)--(0,4,6)--(2,4,6)--(2,6,4)--cycle;
     \filldraw[conefilla](6,4,2)--(4,6,2)--(4,6,0)--(6,4,0)--cycle;

    \filldraw[conefill](2.5,0.5,4.5)--(0.5,2.5,4.5)--(0.5,4.5,2.5)--(2.5,4.5,0.5)--(4.5,2.5,0.5)--(4.5,0.5,2.5)--cycle;
    \filldraw[conefill](4.5,0.5,2.5)--(2.5,0.5,4.5)--(2.5,0.5,6.5)--(4.5,0.5,6.5)--(6.5,0.5,4.5)--(6.5,0.5,2.5)--cycle;
    \filldraw[conefill](0.5,4.5,2.5)--(0.5,2.5,4.5)--(0.5,2.5,6.5)--(0.5,4.5,6.5)--(0.5,6.5,4.5)--(0.5,6.5,2.5)--cycle;
    \filldraw[conefill](6.5,2.5,0.5)--(6.5,4.5,0.5)--(4.5,6.5,0.5)--(2.5,6.5,0.5)--(2.5,4.5,0.5)--(4.5,2.5,0.5)--cycle;
    \filldraw[conefill](2.5,0.5,6.5)--(0.5,2.5,6.5)--(0.5,2.5,4.5)--(2.5,0.5,4.5)--cycle;
    \filldraw[conefill](6.5,0.5,2.5)--(6.5,2.5,0.5)--(4.5,2.5,0.5)--(4.5,0.5,2.5)--cycle;
    \filldraw[conefill](0.5,6.5,2.5)--(2.5,6.5,0.5)--(2.5,4.5,0.5)--(0.5,4.5,2.5)--cycle;
    \filldraw[conefill](4.5,2.5,6.5)--(2.5,4.5,6.5)--(2.5,6.5,4.5)--(4.5,6.5,2.5)--(6.5,4.5,2.5)--(6.5,2.5,4.5)--cycle;
    \filldraw[conefill](6.5,0.5,4.5)--(6.5,2.5,4.5)--(6.5,4.5,2.5)--(6.5,4.5,0.5)--(6.5,2.5,0.5)--(6.5,0.5,2.5)--cycle;
     \filldraw[conefill](0.5,6.5,4.5)--(2.5,6.5,4.5)--(4.5,6.5,2.5)--(4.5,6.5,0.5)--(2.5,6.5,0.5)--(0.5,6.5,2.5)--cycle;
     \filldraw[conefill](2.5,0.5,6.5)--(0.5,2.5,6.5)--(0.5,4.5,6.5)--(2.5,4.5,6.5)--(4.5,2.5,6.5)--(4.5,0.5,6.5)--cycle;
     \filldraw[conefill](6.5,0.5,4.5)--(4.5,0.5,6.5)--(4.5,2.5,6.5)--(6.5,2.5,4.5)--cycle;
     \filldraw[conefill](0.5,6.5,4.5)--(0.5,4.5,6.5)--(2.5,4.5,6.5)--(2.5,6.5,4.5)--cycle;
     \filldraw[conefill](6.5,4.5,2.5)--(4.5,6.5,2.5)--(4.5,6.5,0.5)--(6.5,4.5,0.5)--cycle;
     
     \node at (2,2,4) [circle,fill=black, scale=0.7] {};
     \node at (2,4,4) [circle,fill=black, scale=0.7] {};
     \node at (4,2,4) [circle,fill=black, scale=0.7] {};
     \node at (4,4,2) [circle,fill=black, scale=0.7] {};
     \node at (2,4,2) [circle,fill=black, scale=0.7] {};
     \node at (4,2,2) [circle,fill=black, scale=0.7] {};
     \node at (4,4,4) [circle, fill=black, scale=0.7] {};
     \node at (2,2,6) [circle, fill=black, scale=0.7] {};
     \node at (2,6,2) [circle, fill=black, scale=0.7] {};
     \node at (6,2,2) [circle, fill=black, scale=0.7] {};
     
    \node at (4,2,6) [circle, fill=black, scale=0.7] {};
     \node at (2,4,6) [circle, fill=black, scale=0.7] {};
     \node at (2,6,4) [circle, fill=black, scale=0.7] {};
     \node at (4,6,2) [circle, fill=black, scale=0.7] {};
     \node at (6,4,2) [circle, fill=black, scale=0.7] {};
     \node at (6,2,4) [circle, fill=black, scale=0.7] {};
\end{tikzpicture}
\vspace{0.3 cm}
\begin{gather*}
    \omega(1)=\bigg(\frac{1}{4},  \frac{1}{4},\frac{1}{4},\frac{1}{4} \bigg)\\
    C(Z_{K_4}+ \omega(1)) = 16
\end{gather*}
\end{subfigure}
\begin{subfigure}[b]{0.3\textwidth}
\centering
\begin{tikzpicture}[join=round, scale=0.5]
       \tikzstyle{conefill} = [fill=blue!20,fill opacity=0.4]
    \tikzstyle{conefilla} = [opacity=0.4, fill opacity=0]
    \filldraw[conefilla](2,0,4)--(0,2,4)--(0,4,2)--(2,4,0)--(4,2,0)--(4,0,2)--cycle;
    \filldraw[conefilla](4,0,2)--(2,0,4)--(2,0,6)--(4,0,6)--(6,0,4)--(6,0,2)--cycle;
    \filldraw[conefilla](0,4,2)--(0,2,4)--(0,2,6)--(0,4,6)--(0,6,4)--(0,6,2)--cycle;
    \filldraw[conefilla](6,2,0)--(6,4,0)--(4,6,0)--(2,6,0)--(2,4,0)--(4,2,0)--cycle;
    \filldraw[conefilla](2,0,6)--(0,2,6)--(0,2,4)--(2,0,4)--cycle;
    \filldraw[conefilla](6,0,2)--(6,2,0)--(4,2,0)--(4,0,2)--cycle;
    \filldraw[conefilla](0,6,2)--(2,6,0)--(2,4,0)--(0,4,2)--cycle;
    \filldraw[conefilla](4,2,6)--(2,4,6)--(2,6,4)--(4,6,2)--(6,4,2)--(6,2,4)--cycle;
    \filldraw[conefilla](6,0,4)--(6,2,4)--(6,4,2)--(6,4,0)--(6,2,0)--(6,0,2)--cycle;
     \filldraw[conefilla](0,6,4)--(2,6,4)--(4,6,2)--(4,6,0)--(2,6,0)--(0,6,2)--cycle;
     \filldraw[conefilla](2,0,6)--(0,2,6)--(0,4,6)--(2,4,6)--(4,2,6)--(4,0,6)--cycle;
     \filldraw[conefilla](6,0,4)--(4,0,6)--(4,2,6)--(6,2,4)--cycle;
     \filldraw[conefilla](0,6,4)--(0,4,6)--(2,4,6)--(2,6,4)--cycle;
     \filldraw[conefilla](6,4,2)--(4,6,2)--(4,6,0)--(6,4,0)--cycle;

    \filldraw[conefill](2+1,0+1,4+1)--(0+1,2+1,4+1)--(0+1,4+1,2+1)--(2+1,4+1,0+1)--(4+1,2+1,0+1)--(4+1,0+1,2+1)--cycle;
    \filldraw[conefill](4+1,0+1,2+1)--(2+1,0+1,4+1)--(2+1,0+1,6+1)--(4+1,0+1,6+1)--(6+1,0+1,4+1)--(6+1,0+1,2+1)--cycle;
    \filldraw[conefill](0+1,4+1,2+1)--(0+1,2+1,4+1)--(0+1,2+1,6+1)--(0+1,4+1,6+1)--(0+1,6+1,4+1)--(0+1,6+1,2+1)--cycle;
    \filldraw[conefill](6+1,2+1,0+1)--(6+1,4+1,0+1)--(4+1,6+1,0+1)--(2+1,6+1,0+1)--(2+1,4+1,0+1)--(4+1,2+1,0+1)--cycle;
    \filldraw[conefill](2+1,0+1,6+1)--(0+1,2+1,6+1)--(0+1,2+1,4+1)--(2+1,0+1,4+1)--cycle;
    \filldraw[conefill](6+1,0+1,2+1)--(6+1,2+1,0+1)--(4+1,2+1,0+1)--(4+1,0+1,2+1)--cycle;
    \filldraw[conefill](0+1,6+1,2+1)--(2+1,6+1,0+1)--(2+1,4+1,0+1)--(0+1,4+1,2+1)--cycle;
    \filldraw[conefill](4+1,2+1,6+1)--(2+1,4+1,6+1)--(2+1,6+1,4+1)--(4+1,6+1,2+1)--(6+1,4+1,2+1)--(6+1,2+1,4+1)--cycle;
    \filldraw[conefill](6+1,0+1,4+1)--(6+1,2+1,4+1)--(6+1,4+1,2+1)--(6+1,4+1,0+1)--(6+1,2+1,0+1)--(6+1,0+1,2+1)--cycle;
     \filldraw[conefill](0+1,6+1,4+1)--(2+1,6+1,4+1)--(4+1,6+1,2+1)--(4+1,6+1,0+1)--(2+1,6+1,0+1)--(0+1,6+1,2+1)--cycle;
     \filldraw[conefill](2+1,0+1,6+1)--(0+1,2+1,6+1)--(0+1,4+1,6+1)--(2+1,4+1,6+1)--(4+1,2+1,6+1)--(4+1,0+1,6+1)--cycle;
     \filldraw[conefill](6+1,0+1,4+1)--(4+1,0+1,6+1)--(4+1,2+1,6+1)--(6+1,2+1,4+1)--cycle;
     \filldraw[conefill](0+1,6+1,4+1)--(0+1,4+1,6+1)--(2+1,4+1,6+1)--(2+1,6+1,4+1)--cycle;
     \filldraw[conefill](6+1,4+1,2+1)--(4+1,6+1,2+1)--(4+1,6+1,0+1)--(6+1,4+1,0+1)--cycle;
     
     \node at (2,2,4) [circle,fill=black, scale=0.7] {};
     \node at (2,4,4) [circle,fill=black, scale=0.7] {};
     \node at (4,2,4) [circle,fill=black, scale=0.7] {};
     \node at (4,4,2) [circle,fill=black, scale=0.7] {};
     \node at (2,4,2) [circle,fill=black, scale=0.7] {};
     \node at (4,2,2) [circle,fill=black, scale=0.7] {};
     \node at (4,4,4) [circle,fill=black, scale=0.7] {};
     
     \node at (2,2,6) [circle, draw, scale=0.7] {};
     \node at (2,6,2) [circle, draw, scale=0.7] {};
     \node at (6,2,2) [circle, draw, scale=0.7] {};
     
    \node at (4,2,6) [circle, fill=black, scale=0.7] {};
     \node at (2,4,6) [circle, fill=black, scale=0.7] {};
     \node at (2,6,4) [circle, fill=black, scale=0.7] {};
     \node at (4,6,2) [circle, fill=black, scale=0.7] {};
     \node at (6,4,2) [circle, fill=black, scale=0.7] {};
     \node at (6,2,4) [circle, fill=black, scale=0.7] {};
\end{tikzpicture}
\vspace{0.3 cm}
\begin{gather*}
    \omega(2)=\bigg(\frac{1}{2},  \frac{1}{2},\frac{1}{2},\frac{1}{2} \bigg)\\
    C(Z_{K_4}+ \omega(2)) =13
\end{gather*}
\end{subfigure}
\label{fig}
\caption{Let $K_4$ be the complete graph with $4$ vertices. The pictures represent the 3-dimensional polytope $Z_{K_4} + \omega(d) \subset \RR^4$, i.e.\ the graphical zonotope of $Z_{K_4}$ translated by the vector $\omega(d)$. The lattice points in the interior of $Z_{K_4} + \omega(d)$ are the circles filled in black, and their number is $C(Z_{K_4} + \omega(d))$. The \emph{visible} faces of the zonotope, i.e.\ those that are not contained in the closure of faces whose outward normal unit vector $\mathfrak{n}$ in $\RR^3$ pairs negatively with $\mathfrak{v} \coloneqq (1,1,1)$, in symbols $\mathfrak{n} \cdot \mathfrak{v} < 0$, can be enumerated as follows: $4$ hexagons, $3$ rectangles, $12$ edges and $6$ points. The same multiplicities appear in \eqref{eq:completegraphformula}, namely
\begin{align*}
   \hspace{-2.5 cm} C(K_4+\omega(1)) =  C(K_4) + 4C(K_3) C(K_1)  + 3 C(K_2)^2  + 12C(K_2)C(K_1)^2 + 6C(K_1)^4.
\end{align*} }
\end{figure}
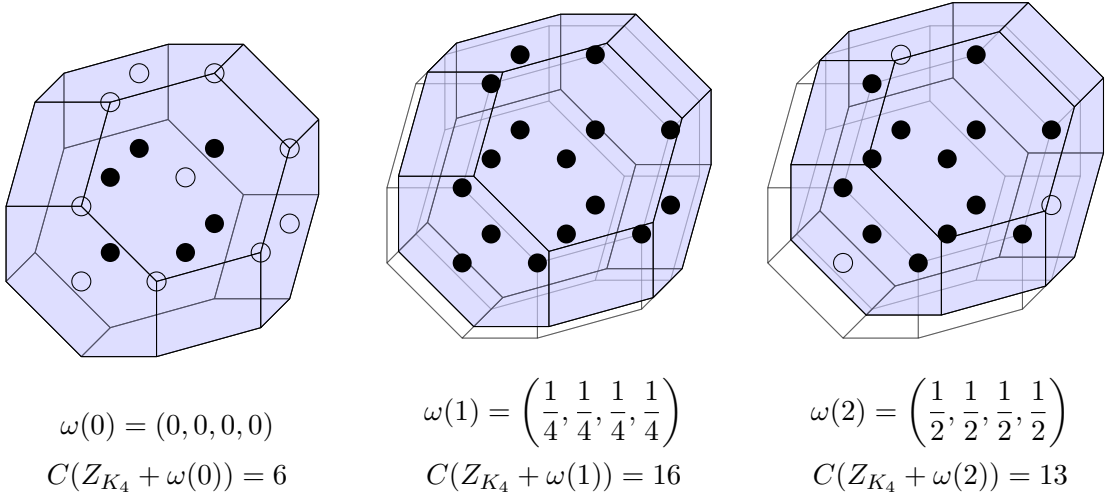

By \cite{Moci2012,DAdderioMoci2012}, the Tutte polynomial determines the Ehrhart polynomial for the zonotope $Z_{K_{m}}$ as follows 
\[E_{K_{m+1}}(q)= \lvert qZ_{K_{m+1}} \cap \mathbb{Z}^m \rvert = q^m T_m\Bigl(1+\frac{1}{q},1\Bigr). \]
Now by Ehrhart–Macdonald reciprocity we have
\[C(K(m,e))=(-1)^{m-1} E_{K_{m}}(-e).\] 
We report these  polynomials for $m \leq 4$ in \Cref{tab:points_count}. Finally, we obtain the desired lattice point count \begin{align*}
    C(K({4,e})+ \omega(0)) & = 16 e^3-15e^2+6e-1,\\
    C(K({4,e})+ \omega(1)) & = 16 e^3,\\
    C(K({4,e})+ \omega(2)) & = 16 e^3-3e^2.
\end{align*} \end{exa}
\section{Permutation representations on lattice points in graphical zonotopes}\label{sec:permutation}
Let $\Aut (\Gamma) < \mathfrak{S}_r$ be the group of automorphisms of the graph $\Gamma$ with $r$ vertices, i.e.\ the permutations of the vertices of $\Gamma$ mapping edges of $\Gamma$ to edges of $\Gamma$.
In this section, we promote the numerical identity \eqref{eq:claim} to an isomorphism of $\Aut (\Gamma)$-representations; see \cref{thm:repr}. 

From now on assume that $\omega \in \QQ^r$ is an $\Aut(\Gamma)$-invariant vector.
The group $\Aut (\Gamma)$ acts by permutations on the graphical zonotope $Z_{\Gamma}\subset \RR^{r}$, and so on the vector space with basis the set of integral points in $\mathrm{Int}(Z_{\Gamma} + \omega)$.
We denote the latter $\Aut (\Gamma)$-representation by $\mathcal{C}(\Gamma+\omega)$.

Furthermore, the group $\Aut (\Gamma) < \mathfrak{S}_r$ acts by permutations on the posets $\Fl$ and $\Fl_\omega$ as $\omega$ is $\Aut(\Gamma)$-invariant.
The stabilizer of $\underline{S} \in \Fl$, denoted $\Stab (\underline{S}) < \Aut (\Gamma)$, acts on $\Fl_{\geq \underline{S}}$ preserving $\Fl_{\omega, \geq \underline{S}}$.
Therefore, $\widetilde{H}_{\ell(\underline{S})-3}(\Delta(\Fl_{\omega, \geq \underline{S}}))$ admits a right action of $\Stab (\underline{S})$.

\begin{defn}
 Let $\Gamma$ be an oriented graph. Define the \textbf{orientation character} \[o_{\Gamma}\colon \Aut (\Gamma) \to \ZZ/2\ZZ \qquad \text{ by } \qquad o_{\Gamma}(\sigma) \coloneqq \prod_{e \in E(\Gamma)} o(e, \sigma),\]
where $o(e,\sigma)$ is $1$ if the map $\sigma|_{e} \colon e \to \sigma(e)$ is orientation-preserving and $-1$ otherwise. 
\end{defn}
Note that the definition of $o_\Gamma$ does not depend on the chosen orientation and it is indeed multiplicative, i.e.\ 
\begin{equation}\label{eq:multiplicativity}
    o_{\Gamma}(\sigma)o_{\Gamma}(\tau)=o_{\Gamma}(\sigma \tau).
\end{equation}
Indeed, let $\Gamma^d$ be the double of $\Gamma$, i.e.\ the oriented graph obtained from $\Gamma$ by replacing each
edge $e$ of $\Gamma$ with a pair of edges $e^+$ and $e^-$ having the same endpoints as $e$ but opposite orientations. 
Any $\sigma \in \Aut(\Gamma)$ define an automorphism $\sigma^d \in \Aut(\Gamma^d)$ given by \[\sigma^d(e^{\pm})=\begin{cases}
\sigma(e)^{\pm} & \text{ if } o(e, \sigma)=1,\\
\sigma(e)^{\mp} & \text{ if } o(e, \sigma)=-1.
\end{cases}\] Thus we can identify $\Aut(\Gamma)$ and $\Aut(\Gamma^d)$ as subgroups of $\mathfrak{S}_r$. Now, $o_{\Gamma}$ is the sign representation of the action by permutation of $\Aut(\Gamma)$ on the set of edges $E({\Gamma^{d}})$ of $\Gamma^d$
\[\Aut(\Gamma) \simeq \Aut(\Gamma^d) \to \mathfrak{S}(E({\Gamma^{d}})).\]
In the following, we also consider the sign representation $\sgn ( \sigma \colon V(\Gamma) \to V(\Gamma))$ with respect to the action on the vertex set.
\begin{defn}
Let $\underline{S} \in \Fl$ be a flat of $\Gamma$, $\Stab(\underline{S})< \Aut(\Gamma)$ be its stabilizer and $\Gamma^{\underline{S}}$ be the contracted (loopless) graph defined in \cref{def:graph}.
Define $\alpha_{\underline{S}}$ as the character of $\Stab(\underline{S})$ given by
\[ \alpha_{\underline{S}} = \sgn(\sigma\colon V(\Gamma^{\underline{S}}) \to V(\Gamma^{\underline{S}})) \otimes o_{\Gamma^{\underline{S}}}(\sigma).\]
\end{defn}

\begin{exa}
Consider the graph $\Gamma$
\begin{center}
\begin{tikzpicture}[thick,scale=0.6]
{
\tikzstyle{every node}=[circle, draw,
                        inner sep=1pt, minimum width=4pt]
        \node (1) at (0,1){1};
        \node (2) at (0,-1){2};
        \node (3) at (1.73,0){3};}
        \draw (1) -- (2) node[midway,left]{$e$};
        \draw (1) -- (3) node[midway, above]{$f$};
        \draw (2) -- (3) node[midway, below]{$f$};
\end{tikzpicture}
\end{center}
where $e$ and $f$ are the number of edges between two vertices. If  $e \neq f$, then $\Aut(\Gamma)= \Z/2\Z = \langle (12) \rangle$ and $a_\Gamma(\sigma)=(-1)^{e+1}$.
\end{exa}
\begin{thm} \label{thm:repr}
    For any graph $\Gamma$ and for any $\Aut(\Gamma)$-invariant vector $\omega \in \QQ^r$, we have
    \begin{equation}\label{eq:representations}
        \mathcal{C}(Z_{\Gamma}+\omega)= \mathcal{C}(Z_{\Gamma}) \oplus \bigoplus_{\underline{S} \in \Fl_\omega /\Aut(\Gamma)} \Ind_{\Stab (\underline{S})}^{\Aut (\Gamma)} \big( \alpha_{\underline{S}} \otimes \widetilde{H}_{\ell(\underline{S})-3}(\Delta(\overline{\Fl}_{\omega, \geq \underline{S}})) \otimes \mathcal{C}(Z_{\Gamma_{\underline{S}}})\big),
    \end{equation}
    with $\omega_{\underline{S}}=(\omega_{S_1}, \dots, \omega_{S_{\ell(\underline{S})}}) \in \RR^{\ell(\underline{S})}$ and $\omega_{S_i}= \sum_{j \in S_i} \omega_j$.
\end{thm}

\begin{rmk}
The isomorphism type of the induced representation 
\[\Ind_{\Stab (\underline{S})}^{\Aut (\Gamma)} \big( \alpha_{\underline{S}} \otimes \widetilde{H}_{\ell(\underline{S})-3} (\Delta(\overline{\Fl}_{\omega, \geq \underline{S}}))\] is independent of the choice of a representative $\underline{S}$ in its $\Aut (\Gamma)$-coset of $\Fl$.
\end{rmk}

\begin{exa}
Fix  $\omega = (\frac{1}{2},\frac{1}{2},0)$. Consider the graph $\Gamma$ below and its poset of non $\omega$-integral flats $\Fl_{\omega}$ 
\begin{center}
\begin{tikzpicture}[thick,scale=0.6]
{
\tikzstyle{every node}=[circle, draw,
                        inner sep=1pt, minimum width=4pt]
        \node (1) at (0,1){1};
        \node (2) at (0,-1){2};
        \node (3) at (1.73,0){3};}
        \draw (1) -- (2) node[midway,left]{2};
        \draw (1) -- (3) node[midway, above]{4};
        \draw (2) -- (3) node[midway, below]{4};
        \node (G) at (0.6,-2){$\Gamma$};
\end{tikzpicture} \hspace{1 cm}
\begin{tikzpicture}[thick]
        \node (1) at (0,0){$1|2|3$};
        \node (3) at (0,1){$13|2$};
        \node (4) at (1,1){$23|1$};
        \node (G) at (0.5,-0.75){$\Fl_{\omega}$};
        \draw  (3) -- (1) -- (4);
\end{tikzpicture}
\end{center}
Then there exists an isomorphism of $\Aut(\Gamma)$-representations 
\begin{equation*}
\mathcal{C}(Z_{\Gamma} + \omega) = \mathcal{C}(Z_{\Gamma}) \oplus \operatorname{Reg}^{\oplus 3} \oplus (\sgn \otimes \sgn \otimes 1).
\end{equation*}
where $\operatorname{Reg}$ and $\sgn$ denote respectively the regular and sign representation of $\Aut(\Gamma) = \langle (12) \rangle =\ZZ/2\ZZ$.
\end{exa}

\subsection{Fixed loci of zonotopes}

The cycle type of a permutation $\sigma \in \Aut(\Gamma)< \mathfrak{S}_r$ is the partition of $r$ consisting of the lengths $l_1, l_2, \dots, l_{|\sigma|}$ of the cycles $\sigma_1 \cdot\sigma_2 \cdot \ldots \cdot \sigma_{|\sigma|}$ of $\sigma$.
Define the quotient graph $\Gamma/\sigma$ as the loopless graph whose vertices are orbits of vertices of $\Gamma$ under the action of $\langle \sigma \rangle$ and whose edges are orbits of edges of $\Gamma$ with endpoints in different orbits of vertices.
Explicitly, the set of vertices is $[\ell(\sigma)]$ and the number of edges between $i$ and $j$ is
\[ x_{ij} \coloneqq \frac{1}{\lcm(l_i,l_j)}\sum_{ \substack{a \in \sigma_i, \, b \in \sigma_j}} y_{ab} \in \ZZ.\]
Define the translation vectors $\omega_{\sigma}$ and $t_{\sigma}=t_\sigma(\Gamma)$ in coordinates by
\begin{align*}
     \omega_{\sigma, i} \coloneqq \frac{1}{l_i} \sum_{a \in \sigma_i} \omega_a, \qquad
     t_{\sigma, i}  \coloneqq \frac{1}{l_i}\sum_{\{a,b\} \in \sigma_i} y_{ab}.
\end{align*} For brevity, we often write simply $t$ for $t_{\sigma}=t_\sigma(\Gamma)$ when this does not cause any confusion.

\begin{rmk}\label{rmk:half_int}
    Notice that $\omega_{\sigma, i}=\omega_j$ for any $j \in \sigma_i$, and \begin{align}\label{eq:tj}
        t_{\sigma, i} & =\frac{1}{l_i}\sum_{\{a,b\} \subseteq \sigma_i} y_{ab} = \sum_{\{a,b\} \subseteq \sigma_i} \bigg( \frac{1}{l_i} \sum_{b \neq \sigma^{l_i/2}(a)} y_{ab} + \frac{1}{l_i}\sum_{b = \sigma^{l_i/2}(a)} y_{ab}\bigg) \\
        & =  \begin{cases}
    \sum_{b \in \sigma_i} y_{ab}  & \text{ if }l_i \text{ is odd,}\\
    \bigg(\sum_{b \neq \sigma^{l_i/2}(a)} y_{ab}\bigg) + \frac{1}{2} y_{a\sigma^{l_i/2}(a)}  & \text{ if }l_i \text{ is even.}
    \end{cases} \nonumber
    \end{align} is a half-integer.
    Moreover, $t_{\sigma, i}$ is an integer if and only if $l_i$ is odd, or if $l_i$ and $ y_{a \sigma^{l_i/2}(a)}$ are both even. 
\end{rmk}

A permutahedron is endowed with a standard action of the symmetric group. The fixed locus of any such permutation is described in \cite[Proposition 2.11]{ASVMother}.\footnote{Observe that the description of the permutahedron chosen in \cite[Proposition 2.1]{ASVMother} differs from \cref{def:graphical} by a translation by an integral vector.} In the next lemma, we extend it to arbitrary graphical zonotopes. 

\begin{lem}\label{lem:fixedlocus}
    Let $\sigma \in \Aut(\Gamma)$ be a permutation with cycle lengths $l_1,l_2, \dots l_{|\sigma|}$.
    Then the fixed locus $Z_\Gamma^\sigma$ of the graphical zonotope $Z_\Gamma$ by the action of $\sigma$ is the zonotope
    \[ \pi(Z_\Gamma^\sigma) = \sum_{i < j \in [|\sigma|]} x_{ij} \left[ \frac{l_j}{\gcd(l_i,l_j)}e_i, \frac{l_i}{\gcd(l_i,l_j)}e_j \right] + \sum^{|\sigma|}_{i=1} t_{i} e_i  \]
    where $\{e_i\}_i$ is the standard basis of $\RR^{|\sigma|}$.
\end{lem}
\begin{proof}
    Consider the averaging map $\pi \colon \RR^r \to \RR^{|\sigma|}$ defined by 
    \[ \pi(\underline{x})= \left( \frac{1}{l_i} \sum_{a \in \sigma_i}x_a\right)_{i=1,\dots, |\sigma|}\]
    and the map $\iota \colon \RR^{|\sigma|} \to \RR^r $ defined by $\underline{y} \mapsto \underline{x}$ where $x_a = y_i$ if $a \in \sigma_i$.
    The set of $\sigma$-fixed points is $\operatorname{Im} \iota$ and $\pi \circ \iota = \id_{\RR^{|\sigma|}}$. 
    Hence $Z_\Gamma^\sigma \subseteq \iota(\pi(Z_\Gamma)))$.
    The other inclusion $Z_\Gamma^\sigma \supseteq \iota(\pi(Z_\Gamma)))$ follows from the convexity of $Z_\Gamma$ and the $\sigma$-invariance $\sigma(Z_\Gamma)=Z_\Gamma$.

    Finally, by linearity it is enough to compute $\pi([e_a,e_b])$. If $a,b \in \sigma_i$ then $\pi([e_a,e_b])= \frac{e_i}{l_i}$, otherwise $a \in \sigma_i, \, b \in \sigma_j$ and $\pi([e_a,e_b])= [\frac{e_i}{l_i},\frac{e_j}{l_j}]$.
    Notice that 
    \[ x_{ij} \left[ \frac{l_j}{\gcd(l_i,l_j)}e_i, \frac{l_i}{\gcd(l_i,l_j)}e_j \right] = \sum_{ \substack{a \in \sigma_i, \, b \in \sigma_j}} y_{ab} \left[ \frac{e_i}{l_i}, \frac{e_j}{l_j} \right].\]
\end{proof}

\begin{figure}
\begin{center}
\begin{tikzpicture}[join=round, scale=0.5]
    \tikzstyle{conefill} = [fill=blue,fill opacity=0.5]
    \filldraw[conefill](3,0,3)--(1,4,1)--(1,6,1)--(3,6,3)--(5,2,5)--(5,0,5)--cycle;
    \tikzstyle{conefill} = [fill=blue!0,fill opacity=0.4]
    \filldraw[conefill](2,0,4)--(0,2,4)--(0,4,2)--(2,4,0)--(4,2,0)--(4,0,2)--cycle;
    \filldraw[conefill](4,0,2)--(2,0,4)--(2,0,6)--(4,0,6)--(6,0,4)--(6,0,2)--cycle;
    \filldraw[conefill](0,4,2)--(0,2,4)--(0,2,6)--(0,4,6)--(0,6,4)--(0,6,2)--cycle;
    \filldraw[conefill](6,2,0)--(6,4,0)--(4,6,0)--(2,6,0)--(2,4,0)--(4,2,0)--cycle;
    \filldraw[conefill](2,0,6)--(0,2,6)--(0,2,4)--(2,0,4)--cycle;
    \filldraw[conefill](6,0,2)--(6,2,0)--(4,2,0)--(4,0,2)--cycle;
    \filldraw[conefill](0,6,2)--(2,6,0)--(2,4,0)--(0,4,2)--cycle;
    \filldraw[conefill](4,2,6)--(2,4,6)--(2,6,4)--(4,6,2)--(6,4,2)--(6,2,4)--cycle;
    \filldraw[conefill](6,0,4)--(6,2,4)--(6,4,2)--(6,4,0)--(6,2,0)--(6,0,2)--cycle;
     \filldraw[conefill](0,6,4)--(2,6,4)--(4,6,2)--(4,6,0)--(2,6,0)--(0,6,2)--cycle;
     \filldraw[conefill](2,0,6)--(0,2,6)--(0,4,6)--(2,4,6)--(4,2,6)--(4,0,6)--cycle;
     \filldraw[conefill](6,0,4)--(4,0,6)--(4,2,6)--(6,2,4)--cycle;
     \filldraw[conefill](0,6,4)--(0,4,6)--(2,4,6)--(2,6,4)--cycle;
     \filldraw[conefill](6,4,2)--(4,6,2)--(4,6,0)--(6,4,0)--cycle;
     
     \node at (2,2,4) [circle, draw, scale=0.7] {};
     \node at (2,4,4) [circle, draw, scale=0.7] {};
     \node at (4,2,4) [circle,fill=black, scale=0.7] {};
     \node at (4,4,2) [circle, draw, scale=0.7] {};
     \node at (2,4,2) [circle,fill=black, scale=0.7] {};
     \node at (4,2,2) [circle, draw, scale=0.7] {};
\end{tikzpicture}
\caption{Let $K_4$ be the complete graph with $4$ vertices, and $\sigma = (12) \in \mathfrak{S}_4$. The fixed locus $Z_{K_4}^\sigma$ of the graphical zonotope $Z_{K_4}$ by the action of $\sigma$ is the zonotope 
    \[\hspace{-1.5cm} \pi(Z_{K_4}^\sigma) = [e_1, 2e_2]+[e_1, 2e_3]+ [e_2,e_3]+\frac{1}{2}e_1. \]
Note that $C(Z_{K_4})=6$ and $C(Z^{\sigma}_{K_4})=2$.}
\end{center}
\end{figure}
\begin{lem}
\label{lemma:inv_order_complex}
For any vector $\omega \in \QQ^r$ and any $\sigma \in \mathfrak{S}_r$ such that $\sigma(\omega)=\omega$, we have
    \begin{equation}\label{eq:tophom}
        \chi_{\widetilde{H}_{r-3}(\Delta(\Fl_{\omega}))}(\sigma)= (-1)^{r-3}\tilde{\chi} (\Delta(\Fl_{\omega}^\sigma)) = (-1)^{r-3} \mu_{\Fl_{\omega}^\sigma} (\hat{0},\hat{1}),
    \end{equation}
where $\chi_V(\sigma)$ is the character of the representation $V$ evaluated at $\sigma$, and $\tilde{\chi}$ is the reduced Euler characteristic of a topological space.
\end{lem}
\begin{proof}
    Since the order complex $\Delta(\Fl_\omega)$ has the homotopy type of a wedge of $(r-3)$-spheres by \Cref{prop:spheres}, we apply the Lefschetz fixed-point theorem for simplicial complexes:
    \[\chi_{\widetilde{H}_{r-3}(\Delta(\Fl_{\omega}))}(\sigma)=(-1)^{r-3} \sum_{i} (-1)^i \chi_{\widetilde{H}_{i}(\Delta(\Fl_{\omega}))}(\sigma) = (-1)^{r-3}\tilde{\chi} (\Delta(\Fl_{\omega})^\sigma). \]
    
    We are left to prove that $\Delta(\Fl_{\omega})^\sigma= \Delta(\Fl_{\omega}^\sigma)$.
    Consider a $\sigma$-fixed point $p\in \Delta(\Fl_{\omega})$ and let $F$ be the unique face that contains $p$ in its relative interior.
    Since the $\sigma$-action is simplicial then $\sigma(F)=F$. Moreover, $\sigma$ acts on $\Fl_\omega$ by a rank-preserving automorphism and the vertices of $F$ have different rank as elements in $\Fl_\omega$.
    Hence, $F \in \Delta(\Fl_{\omega})^\sigma$, and this proves that the set of fixed points is the full subcomplex generated by $\sigma$-fixed vertices.
    The third equation in \eqref{eq:tophom} follows from the Hall's theorem \cite[Proposition 3.8.6]{Stanley2012}.
\end{proof}

\subsection{Combinatorial lemmas}
Fix a graph $\Gamma$,  an automorphism $\sigma \in \Aut(\Gamma)$ and a forest $F$ of $\Gamma/\sigma$.
The quotient map $q \colon \Gamma \to \Gamma/\sigma$ induces a morphism between the posets of flats 
\begin{align*}
   \pi \colon \Fl({\Gamma})^{\sigma} \to \Fl({\Gamma/\sigma}), \quad
   \underline{S}=\{S_i\} \mapsto \underline{T}=\{T_j\}, 
\end{align*}
where $T_j = q(S_i)$ for some $i$. Define 
\[\underline{B}= \underline{B}(F) \in \Fl({\Gamma/\sigma})\]
as the finest flat containing
$F$, or equivalently as the partition of $V(\Gamma/\sigma)$ whose blocks are the sets of vertices of the connected components of $F$.
For each block $B_j \in \underline{B}$ define 
\[m_j \coloneqq \min_{i \in B_j} v_2(l_i)\]
where $v_2$ is the $2$-adic valuation of the length $l_i$ of the cycle $\sigma_i$.
\color{black}
If $m_j>0$, then define 
the sets
\[A_{\underline{B},\sigma,j}=\{ \text{cycle type of the permutation } \sigma_i^{l_i/2} \mid  i \in B_j, \, v_2(l_i)=m_j, \, y_{a \sigma^{l_i/2}(a) } \text{ odd}\} \subseteq \Fl^{\sigma}\]
for any $a \in \sigma_i$. The set $A_{\underline{B},\sigma,j}$ does not depend on the choice of $a$.
Define also 
\[A_{\underline{B},\sigma}= \{ \vee A_{\underline{B},\sigma,j} \mid m_j>0, \, \lvert A_{\underline{B},\sigma,j}\rvert \text{ is odd} \} \subseteq \Fl^{\sigma}.\]
We write $a \sim_{\underline{S}} b$ if the vertices $a, b \in V(\Gamma)$ lies in the same block of the flat $\underline{S}$.

\begin{lem}\label{lem:all_or_nothing}
Let $\underline{S} \in \Fl^\sigma$ be a flat such that $\pi(\underline{S}) \geq \underline{B}$.
Then the set of flats in $A_{\underline{B},\sigma,j}$ smaller than or equal than $\underline{S}$, denoted $A_{\underline{B},\sigma,j}^{\leq \underline{S}}$, is empty or equal to $A_{\underline{B},\sigma,j}$.
\end{lem}
\begin{proof}
    Consider two vertices $i,k$ in $F$ which belongs to the same connected component $B_j$, and such that the type of $\sigma_i^{l_i/2}$ and of $\sigma_k^{l_k/2}$ are in $A_{\underline{B},\sigma,j}$. 
Note that:
    \begin{enumerate}
    \item for any $b \in \sigma_k$, there exists $a\in \sigma_i$ such that $a\sim_{\underline{S}} b$, since $\pi(\underline{S}) \geq \underline{B}$;
        \item $\sigma^{l_i/2}(a)=\sigma^{\lcm(l_i,l_k)/2}(a)$ for any $a \in \sigma_i$, since $v_2(l_i)=m_j=v_2(l_k)$;
        \item $\sigma^{l_k/2}(b)=\sigma^{\lcm(l_i,l_k)/2}(b)$ for any $b \in \sigma_k$, since $v_2(l_i)=m_j=v_2(l_k)$;
        \item $\sigma^{\lcm(l_i,l_k)/2}(a) \sim_{\underline{S}} \sigma^{\lcm(l_i,l_k)/2}(b)$, since $\underline{S}$ is $\sigma$-invariant.
    \end{enumerate}
    If the type of $\sigma_i^{l_i/2}$ is smaller or equal than $\underline{S}$, equivalently $a \sim_{\underline{S}} \sigma^{l_i/2}(a)$ for any $a \in \sigma_i$, then \[b \sim_{\underline{S}} a \sim_{\underline{S}} \sigma^{l_i/2}(a)=\sigma^{\lcm(l_i,l_k)/2}(a) \sim_{\underline{S}} \sigma^{\lcm(l_i,l_k)/2}(b)=\sigma^{l_k/2}(b)\]
    for any $b \in \sigma_k$, i.e.\ the type of $\sigma_k^{l_k/2}$ is smaller or equal than $\underline{S}$. We conclude that either all types of $\sigma_i^{l_i/2}$ in $A_{\underline{B},\sigma,j}$ are smaller or equal than $\underline{S}$, or none of them are.
\end{proof}

 By \cref{lem:fixedlocus}, the affine transformation $\RR^{|\sigma|}\to\RR^{|\sigma|}$, given by $x_i \mapsto l_i (x_i - t_{i})$, sends the zonotope $\pi(Z_{\Gamma}^\sigma) \simeq Z_{\Gamma}^\sigma$ to the graphical zonotope $Z_{\Gamma^{\underline{S}(\sigma)}}$, where $\underline{S}(\sigma)$ is the cycle type of the permutation $\sigma$ seen as partition of $V(\Gamma)$. In particular, the affine transformation exchanges the affine spaces supporting the faces of the two zonotopes. 
 Equivalently, any flat $\underline{B} \in \Fl(\Gamma/\sigma)$ determines a linear subspace $W_{\underline{B}}$ as in \eqref{eq:flatincoord}, which via the affine transformation above corresponds to the affine subspace
\begin{equation*}
    W^t_{\underline{B}} \coloneqq \big\{\underline{x} \in \mathbb{R}^{|\sigma|} \mid \sum_{i \in B_j} l_i x_i = \sum_{i \in B_j} l_i t_i \textnormal{ for all } B_j \in \underline{B} \big \}.  
\end{equation*}
Up to an integral translation, $W^t_{\underline{B}}$ supports faces of $Z_{\Gamma}^\sigma$, and so the affine span of any such face contains a lattice point if and only if $W^t_{\underline{B}}$ does.
The following lemma generalizes \cite[Lemma 4.5]{ArdilaSV2020} and \cite[Proposition 3.12]{EKM2022}.
\begin{lem} \label{lemma:integral_in_quotient}
Let $\underline{S} \in \Fl^\sigma$ be a flat such that $\pi(\underline{S}) \geq \underline{B}$.
The affine subspace supporting a face of the zonotope $Z_{\Gamma_{\underline{S}}}^\sigma$ and corresponding to the flat $\underline{B}$ has a lattice point if and only if $A_{\underline{B}, \sigma}^{\leq \underline{S}} = \emptyset$.
\end{lem}
\begin{proof}
    A face of the zonotope $\pi(Z_{\Gamma}^\sigma) \simeq Z_{\Gamma}^\sigma$ is determined by a flat $\underline{B} \in \Fl(\Gamma_{\underline{S}}/\sigma)$ and an acyclic orientation of $\Gamma_{\underline{S}}^{\underline{B}}$. Up to an integral translation, its supporting affine space depends only on the flat $\underline{B}$, and it   
    satisfies the equations
    \[ \sum_{i \in B_j} l_i x_i = \sum_{i \in B_j} \sum_{ \substack{\{a,b\} \in \sigma_i \\ a\sim_{\underline{S}} b}} y_{ab} \]
    for all $B_j \in \underline{B}$.
    This linear system has integer solution if and only if 
    \[ \gcd_{i \in B_j}(l_i) \Bigg\vert \sum_{i \in B_j} \sum_{ \substack{\{a,b\} \in \sigma_i \\ a\sim_{\underline{S}} b}} y_{ab}  \]
    for all $j$.
    By \eqref{eq:tj}, this is equivalent to
    \[ \gcd_{i \in B_j}(l_i) \Bigg\vert \sum_{\substack{i \in B_j \\ l_i \text{ even} \\ a \sim_{\underline{S}} \sigma^{l_i/2}(a)}} \frac{l_i}{2} 
    y_{a \sigma^{l_i/2}(a)}\]
    for all $j$ and any $a \in \sigma_i$.
    This happens if and only if
    \[ m_j = v_2(\gcd_{i \in B_j}(l_i)) < v_2 \Bigl( \sum_{\substack{i \in B_j \\ l_i \text{ even} \\ a \sim_{\underline{S}} \sigma^{l_i/2}(a)}} l_i y_{a \sigma^{l_i/2}(a)} \Bigr), \]
    which is equivalent to $\lvert A_{\underline{B},\sigma,j}^{\leq \underline{S}} \rvert$ being even.
    Using \Cref{lem:all_or_nothing} and the definition of $A_{\underline{B}, \sigma}$, we obtain the desired result.
\end{proof}

\begin{lem} \label{lemma:B_to_pi_S}
 Let $\underline{S} \in \Fl^\sigma$ be a flat such that $\pi(\underline{S}) \geq \underline{B}$.
 If $ A_{\underline{B},\sigma}^{\leq \underline{S}} = \emptyset $, then:
 \begin{itemize}
    \item  $A_{\pi(\underline{S}),\sigma}^{\leq \underline{S}} = \emptyset$;
    \item $\lvert A_{\pi(\underline{S}),\sigma} \rvert \equiv \lvert A_{\underline{B},\sigma} \rvert \mod 2$.
\end{itemize} 
\end{lem}
\begin{proof}
    Both statements are additive on the blocks of $\pi(\underline{S})$. Therefore we may assume $\pi(\underline{S})= \hat{1}$ and $\Gamma/\sigma$ is connected.
    Let $m \coloneqq v_2(\gcd(l_i \colon i=1, \dots, |\sigma|))= \min(m_i \colon i =1, \dots, \ell(\underline{B}))$. By definition, we have
    \begin{itemize}
        \item $A_{\hat{1},\sigma,1}= \bigsqcup_{m_j=m} A_{\underline{B},\sigma,j}$;
        \item $A_{\pi(\underline{S}), \sigma}=A_{\hat{1}, \sigma}=\begin{cases} \{\vee A_{\hat{1},\sigma,1}\} \quad & \text{ if }| A_{\hat{1},\sigma,1}|\text{ is odd},\\
        \emptyset \quad & \text{ if }| A_{\hat{1},\sigma,1}|\text{ is even}.\\
        \end{cases}$
    \end{itemize}
 In particular, we have
    \begin{equation}\label{eq:mod1}
        \lvert A_{\pi(\underline{S}),\sigma} \rvert \equiv \lvert A_{\hat{1},\sigma,1} \rvert = \sum_{m_j=m} \lvert A_{\underline{B},\sigma,j} \vert \quad \mod 2.
    \end{equation}
    If there exists $j'$ such that $m_{j'}>m$ and $A_{\underline{B},\sigma,j'} \neq \emptyset$, then choose:
    \begin{itemize}
        \item $\sigma_k$ such that the type of $\sigma_k^{l_k/2}$ is in $A_{\underline{B},\sigma,j'}$,
        \item $\sigma_i$ such that $v_2(l_i)=m$,
        \item $a\in \sigma_i$ and $b\in \sigma_k$ such that $a\sim_{\underline{S}}b$.
    \end{itemize} 
    We have $\sigma^{\lcm(l_i,l_k)/2}(b) \neq b$ and $\sigma^{\lcm(l_i,l_k)/2}(a) = a$, hence \[b\sim_{\underline{S}} a = \sigma^{\lcm(l_i,l_k)/2}(a) \sim_{\underline{S}} \sigma^{\lcm(l_i,l_k)/2}(b).\]
 As a result, the type of $\sigma_k^{l_k/2}$ is smaller than $\underline{S}$, and by the arbitrariness of $\sigma_k$ we have $\vee A_{\underline{B},\sigma,j'} \leq \underline{S}$ for all $j'$ with $m_{j'}>m$.
    The hypothesis $ A_{\underline{B},\sigma}^{\leq \underline{S}} = \emptyset $ implies that $\lvert A_{\underline{B},\sigma,j'} \rvert$ is even for $m_{j'} >m$.
    We obtained that 
     \begin{equation}\label{eq:mod2}  \sum_{m_j=m} \lvert A_{\underline{B},\sigma,j} \rvert 
    \equiv \sum_{j=1}^{\ell(\underline{B})} \lvert A_{\underline{B},\sigma,j} \rvert \equiv \lvert A_{\underline{B},\sigma} \vert \quad \mod 2. \end{equation}
    Then \eqref{eq:mod1} and \eqref{eq:mod2} gives $\lvert A_{\pi(\underline{S}),\sigma} \rvert \equiv \lvert A_{\underline{B},\sigma} \rvert \mod 2$.

    To prove the first statement, assume by contradiction that $A_{\pi(\underline{S}),\sigma}^{\leq \underline{S}} \neq \emptyset$, i.e.\  $\lvert A_{\hat{1},\sigma,1} \rvert$ is odd and $\vee A_{\hat{1},\sigma,1} \leq \underline{S}$.
    Then there exists $j_\circ$ such that $m_{j_\circ}=m$ and $\lvert A_{\underline{B},\sigma,j_\circ} \rvert$ is odd.
    Notice that $\vee A_{\underline{B},\sigma,j_\circ} \leq \vee A_{\hat{1},\sigma,1} \leq \underline{S}$, which contradicts the hypothesis $ A_{\underline{B},\sigma}^{\leq \underline{S}} = \emptyset $. 
\end{proof}

\begin{lem}\label{lem:orientation} For any $\sigma \in \Aut(\Gamma)$, we write 
\begin{equation}\label{eq:oreintformula} o_{\Gamma}(\sigma)=(-1)^{\sum_i \frac{1}{l_i}\sum_{a \in \sigma_i} y_{a \sigma^{l_i/2}(a)}} = \begin{cases} -1 & \quad  \text{ if }\sum_{i} t_i \in \big(\frac{1}{2} \ZZ \big)\setminus \ZZ,\\
1 & \quad  \text{ if } \sum_{i} t_i \in \ZZ.\end{cases}
\end{equation}
In particular, $o_{\Gamma}(\sigma)=1$ if $\gcd_i(l_i)$ is odd.
\end{lem}
\begin{proof}
Orient the edges between $\sigma_i$ and $\sigma_j$ (with $i<j$) from $\sigma_i$ to $\sigma_j$. Since $\sigma$ preserves the orientation of all edges between different cycles, we have
\begin{equation}\label{eq:orientationsplit}
    o_{\Gamma}(\sigma)= \prod_i o_{\sigma _i}(\sigma).
\end{equation}

By the multiplicativity \eqref{eq:multiplicativity} of $o_{\Gamma}$, we can further suppose that $\sigma$ acts transitively on $V(\Gamma)$. The stabilizers of a pair of vertices in $V(\Gamma)$ for the action $\sigma \cdot \{a,b\} = \{\sigma(a), \sigma(b)\}$  is either trivial, or $\ZZ/2\ZZ \simeq \langle \sigma^{l/2}\rangle$, in which case $l$ is even and $b=\sigma^{l/2}(a)$. Edges whose pairs of endpoints have trivial stabilizer can be oriented in such a way that $\sigma$ is orientation-preserving: fix the orientation of a representative of the orbit, and induce the orientation on the other elements in the orbit by imposing that the map $\sigma|_{\sigma^k(e)} \colon \sigma^{k}(e) \to \sigma^{k+1}(e)$ is orientation-preserving. Orient the remaining edges as follows: identify $V(\Gamma)$ and $\ZZ/l\ZZ$ as $\langle \sigma\rangle=\ZZ/l\ZZ$-modules, and orient the edges between $j$ and $j+l/2$, with $0 \leq j<l/2$, from $j$ to $j+l/2$. In this way, the map $\sigma|_e \colon e \to \sigma(e)$ is orientation-preserving for all edges $e \in E(\Gamma)$ except for the edges from $l/2-1$ to $l-1$. Therefore, we obtain \[o_{\sigma_i}(\sigma_i)=(-1)^{y_{l/2-1, l-1}}=(-1)^{\frac{1}{l_i}\sum_{a \in \sigma_i} y_{a \sigma^{l_i/2}(a)}}.\] The second equality \eqref{eq:oreintformula} follows from \cref{rmk:half_int}.
\end{proof}
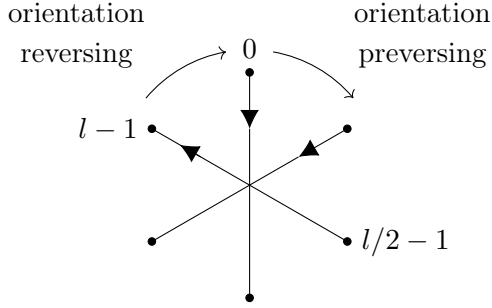
\begin{figure}
    \begin{center}
    \begin{tikzpicture}
   \newdimen\R
   \R=1.5cm
   \draw[-{Latex[length=2.5mm, width=2.5mm]}] (90:\R) -- (90:\R/2);
   \draw (90:\R/2)--(270:\R);
   \draw[-{Latex[length=2.5mm, width=2.5mm]}] (30:\R) -- (30:\R/2);
   \draw (30:\R/2)--(30+180:\R);
   \draw[-{Latex[length=2.5mm, width=2.5mm]}] (330:\R) -- (150:\R/1.4);
   \draw (150:\R/2)--(150:\R);
   \node[inner sep=1pt,circle,draw,fill] at (30:\R) {};
   \node[inner sep=1pt,circle,draw,fill] at (210:\R) {};
   \node[inner sep=1pt,circle,draw,fill] at (270:\R) {};
   \node[inner sep=1pt,circle,draw,fill,label={above : $0$}] at (90:\R) {};
   \node[inner sep=1pt,circle,draw,fill,label={left : $l-1$}] at (150:\R) {};
   \node[inner sep=1pt,circle,draw,fill,label={right : $l/2-1$}] at (330:\R) {};
   \draw[->] (80:{1.2*\R}) arc (80:40:{1.2*\R});
   \draw[->] (140:{1.2*\R}) arc (140:100:{1.2*\R});
   \node[label={left:}, align=center] at (-2.3,2) {orientation\\reversing};
   \node[label={right:}, align=center] at (2.3,2) {orientation\\preversing};
\end{tikzpicture}
    \caption{Illustration of the proof of \cref{lem:orientation} for $l=6$.}
    \label{fig:my_label}\end{center}
\end{figure}

\begin{lem} \label{lemma:alfa_S_sigma}
    Let $\underline{S} \in \Fl^\sigma$ be a flat such that $\pi(\underline{S}) \geq \underline{B}$. If $A_{\underline{B}, \sigma}^{\leq \underline{S}} = \emptyset$, then
    \[ \alpha_{\underline{S}}(\sigma)= (-1)^{\ell(\sigma\colon \underline{S} \to \underline{S})+ \lvert A_{\pi(\underline{S}), \sigma}\rvert}. \] 
\end{lem}
\begin{proof}
    By \Cref{lemma:B_to_pi_S}, $A_{\pi(\underline{S}), \sigma}^{\leq \underline{S}} = \emptyset$.
    We first reduce to the case $\pi(\underline{S})=\hat{1}$ and $\Gamma/\sigma$ connected. This follows from the definition of $\alpha_{\underline{S}}$, and the fact that \begin{equation}\label{eq:or2}
        o_{\Gamma^{\underline{S}}}(\sigma)= \prod_{i} o_{q^{-1}(\Gamma_{T_i})^{\underline{S}_{T_i}}} (\sigma),
    \end{equation}
    where $\pi(\underline{S})=(T_1, \ldots, T_k)$, $q\colon \Gamma \to \Gamma/\sigma$ is the quotient map, and $\underline{S}_{T_i} \coloneqq q^{-1}(\Gamma_{T_i}) \cap \underline{S}$. The proof of \eqref{eq:or2} is identical to that of \eqref{eq:orientationsplit}. 

Assume then that $\pi(\underline{S})=\hat{1}$ consists of a single connected block. We 
   want to prove the equality \[o_{\Gamma^{\underline{S}}}(\sigma)= (-1)^{\lvert A_{\hat{1}, \sigma}\rvert}=(-1)^{|A_{\hat{1}, \sigma, 1}|}.\]
    Note that $\sigma$ acts transitively on the blocks $S_i$ of $\underline{S}$ which intersect all cycles $\sigma_i$ (equivalently the fibers of $q$) in subsets of the same cardinality, since $\underline{S}$ is $\sigma$-invariant, and $q(\underline{S})=q(S_i)=\hat{1}$. 
    We obtain that $\ell (\underline{S}) \mid l_i$ for all $i=1, \dots, |\sigma|$, and so $v_2(\ell(\underline{S})) \leq m \coloneqq \min(m_i)$. 
    
    We distinguish the following cases.
    \begin{description}
        \item[$m=v_2(\ell(\underline{S}))=0$] then $A_{\hat{1}, \sigma} = \emptyset$, and $o_{\Gamma^{\underline{S}}}(\sigma)=1$ by \cref{lem:orientation}.
        \item[$m>v_2(\ell(\underline{S}))=0$] then $o_{\Gamma^{\underline{S}}}(\sigma)=1$ by \cref{lem:orientation}, since $\ell(\underline{S})$ is odd. 
        Moreover, the pairs of vertices in $V(\Gamma)$ of the form $\{a, \sigma^{l_i/2}(a)\}$, with $a\in \sigma_i$, are all contained in the fibers of the contraction map $\Gamma \to \Gamma^{\underline{S}}$, since they are $\sigma^{l_i/2}$-orbits and $\ell(\underline{S})$ is odd. Therefore, the type of $\sigma_i^{l_i/2}$ is smaller or equal than $\underline{S}$ for all $i=1, \dots, |\sigma|$, and so $A_{\hat{1}, \sigma} = A_{\hat{1}, \sigma}^{\leq \underline{S}}$, which is empty by assumption. Hence, $o_{\Gamma^{\underline{S}}}(\sigma)= 1=(-1)^{\lvert A_{\hat{1}, \sigma}\rvert}$.
        \item[$m>v_2(\ell(\underline{S}))>0$] the number of edges in $\Gamma^{\underline{S}}$ between $S_0$ and $S_{\ell(\underline{S})/2}:=\sigma^{\ell(\underline{S})/2}(S_0)$ is even, which implies $o_{\Gamma^{\underline{S}}}(\sigma)=1$ by \cref{lem:orientation}. Indeed, let $e_{ab}$ be an edge whose endpoints $a, b$ lie respectively in $\sigma_i$ and $\sigma_j$.        Recall that the length of the orbit of $e_{ab}$ is
        \[
        |\langle \sigma \rangle \cdot e_{ab}|= \begin{cases}
        \lcm(l_i,l_j) \quad & \text{ if }i\neq j,\\
        l_i &  \text{ if }i=j, b \neq \sigma^{l_i/2}(a),\\
        l_i/2 &  \text{ if }i=j, b = \sigma^{l_i/2}(a),
        \end{cases}
        \]
        while the length of the orbit of a edge between $S_0$ and $S_{\ell(\underline{S})/2}$ in $\Gamma^{\underline{S}}$ is $\ell(\underline{S})/2$. Thus we write
        \[\frac{\ell(\underline{S})}{2}y_{S_0 S_{\ell(\underline{S})/2}}= \sum_{0\leq i < \ell(\underline{S})/2} y_{S_i \sigma^{\ell(\underline{S})/2}(S_i)}= \sum_{\substack{\{a,b\}\subseteq V(\Gamma) \\ q(b)=\sigma^{\ell(\underline{S})/2}(q(a))}}y_{ab}=\sum_{[\{a,b\}]\in \mathcal{O}}|\langle \sigma \rangle \cdot e_{ab}| y_{ab},\]
        where $\mathcal{O}$ is the set of orbits of pairs $\{a,b\}\subseteq V(\Gamma)$ such that $q(b)=\sigma^{\ell(\underline{S})/2}(q(a))$.
        In all cases, the ratio $|\langle \sigma \rangle \cdot e_{ab}|/(\ell(\underline{S})/2)$ is even, and $y_{S_0 S_{\ell(\underline{S})/2}}$ is even too, as we wanted. 
        
        Further, since $2\ell(\underline{S})| l_i$, $\sigma^{l_i/2}$ acts trivially on $\underline{S}$. As above, the type of $\sigma_i^{l_i/2}$ is smaller or equal than $\underline{S}$ for all $i$, and $A_{\hat{1}, \sigma} = A_{\hat{1}, \sigma}^{\leq \underline{S}} = \emptyset$. Hence, $o_{\Gamma^{\underline{S}}}(\sigma)= 1=(-1)^{\lvert A_{\hat{1}, \sigma}\rvert}$.
        \item[$m=v_2(\ell(\underline{S}))>0$] the parity of the number of edges between $S_0$ and $S_{\ell(\underline{S})/2}$ coincides with the parity of  
        \[ \sum_{v_2(l_i)= v_2(\ell(\underline{S}))} y_{a \sigma^{l_i/2}(a)} \]
        for any $a \in \sigma_i$. Hence
        \[ o_{\Gamma^{\underline{S}}}(\sigma)= (-1)^{\sum_{v_2(l_i)= v_2(\ell(\underline{S}))} y_{a \sigma^{l_i/2}(a)}} = (-1)^{\lvert A_{\hat{1}, \sigma,1}\rvert}. \]
    \end{description}
    This completes the proof.
\end{proof}

Let us define 
\[\omega_{\underline{B}}= \Bigl( \sum_{i \in B_j} l_i \omega_i \Bigr)_j = \Bigl( \sum_{k \in \cup_{i \in B_j} \sigma_i} \omega_k \Bigr) _j \in \RR^{\ell(\underline{B})} \]
and
\[ \Half_{\omega}= \Bigl\{ \vee A_{\underline{B},\sigma,j} \mid \frac{(\omega_{\underline{B}})_j}{\gcd_{i \in B_j} (l_i)} \in \frac{1}{2}\ZZ \setminus \ZZ  \Bigr\} \subseteq \Fl^{\sigma}.\]

Define $\underline{R}= \underline{R}(F) \in \Fl^{\sigma}(\Gamma)$ as the finest flat containing all the lifts of the edges of the forest $F$ via the quotient map $q\colon \Gamma \to \Gamma/\sigma$.

\begin{lem}
    \label{lemma:number_blocks_B}
    The inverse image $\pi^{-1}(B_j)$ of any block $\underline{B}_j $ of $\underline{B}(F)$ under the map $\pi \colon \Fl^{\sigma}(\Gamma) \to \Fl(\Gamma/\sigma)$ is the disjoint union of $\gcd_{i \in B_j} (l_i)$ blocks of $\underline{R}$ over which $\sigma$ acts transitively.
\end{lem}
\begin{proof}
    Working block-wise, we can assume $\underline{B}=\{ B_1\} =\hat{1}$ and the statement becomes $\ell(\underline{R})=\gcd_{i \in B_1} (l_i)$.
    Since $\sigma$ acts transitively on the blocks of $\underline{R}$, then $\ell(\underline{R}) \mid \gcd_i(l_i)$.
    
    By contradiction suppose that $\ell(\underline{R})< \gcd_i(l_i)$. 
    For each $i$, each vertex $a\in \sigma_i$ and each integer $k$ coprime with $\gcd_i(l_i)/\ell(\underline{R})$ there is a path in $\pi^{-1}(F)$ from $a$ to $\sigma^{k \ell(\underline{R})}(a)$.
    Choose $i$, $a\in \sigma_i$ and $k$ such that the path has minimal length among all considered paths.
    The assumption $\ell(\underline{R})< \gcd_i(l_i)$ implies that the path is non-constant.
    Since $F$ is a forest, the path starts with an edge with endpoints $\{a,b\}$, and it terminates with an edge with endpoints $\{\sigma^h(a), \sigma^h(b)\}$ in the same orbit of the pair $\{a,b\}$.
    Notice that $h \equiv k\ell(\underline{R}) \mod l_i$, and so $h=k' \ell(\underline{R})$ with $k'$ coprime to $\gcd_i(l_i)/\ell(\underline{R})$. Eliminating the first and last edges of the path, we obtain a shorter path from $b$ to $\sigma^{k'\ell(\underline{R})}(b)$ contrary to the minimality assumption on $a$.
\end{proof}

\begin{lem}\label{lem:RHintegral}
Let $H \subseteq A_{\underline{B},\sigma}$.
The flat $\underline{R} \vee \bigvee H$ is $\omega$-integral if and only if $\gcd_{i \in B_j} (l_i) \mid 2 (\omega_{\underline{B}})_j$ for all $j=1, \dots, \ell(\underline{B})$ and $H \supseteq \Half_\omega$.
\end{lem}
\begin{proof}
    We can reduce to a single block of $\underline{B}=\pi(\underline{R})$, hence we assume $\underline{B}=\hat{1}$.    
    By \Cref{lemma:number_blocks_B} we have $\ell(\underline{R}) = \gcd_i(l_i)$ and  $\omega_{\underline{B}}= \ell(\underline{R}) (\omega_{\underline{R}})_1= \gcd_i(l_i) (\omega_{\underline{R}})_1$.
    There are two cases.
    \begin{description}
        \item[$H=\emptyset$] then $(\omega_{\underline{R}})_1 \in \ZZ$, $\gcd_i(l_i) \mid \omega_{\underline{B}}$ and $\Half_{\omega}= \emptyset$.
        \item[$H=\vee A_{\hat{1},\sigma,1}$] then $\underline{R} \vee H$ has $\gcd_i(l_i)/2$ blocks and $(\omega_{\underline{R} \vee H})_1= 2(\omega_{\underline{R}})_1$.
        Therefore, $2(\omega_{\underline{R}})_1 \in \ZZ$, $\gcd_i(l_i) \mid 2\omega_{\underline{B}}$ and $H \supseteq \Half_{\omega}$.
    \end{description}
\end{proof}

\begin{lem}\label{lem:integralHals} 
The affine subspace supporting a face of the zonotope  $Z_\Gamma^\sigma + \omega$ corresponding to the flat $\underline{B}$ has a lattice point  if and only if $\gcd_{i \in B_j} (l_i) \mid 2 (\omega_{\underline{B}})_j$ for all $j=1, \dots, \ell(\underline{B})$ and $A_{\underline{B},\sigma} = \Half_\omega$.
\end{lem}
\begin{proof}
    The proof is similar to the one of \Cref{lemma:integral_in_quotient}.
    The affine subspace corresponding to $\underline{B}$ has a lattice point if and only if the linear system
    \[ \sum_{i \in B_j} l_i x_i = \sum_{i \in B_j} \Bigl( l_i \omega_i + \sum_{a,b \in \sigma_i } y_{ab} \Bigr) \]
    for $j=1,\dots, \ell(\underline{B})$ has integral solutions.
    This is equivalent to
    \[\gcd_{i \in B_j}(l_i) \Bigg\vert \sum_{i \in B_j} l_i \omega_i + \sum_{\substack{i \in B_j \\ l_i \text{ even}}} \frac{l_i}{2} 
    y_{a \sigma^{l_i/2}(a)}  \]
    for all $j=1,\dots, \ell(\underline{B})$.
    In particular $\gcd_{i \in B_j} (l_i) \mid 2 (\omega_{\underline{B}})_j$.

    In addition, $\gcd_{i \in B_j} (l_i) \nmid  (\omega_{\underline{B}})_j$ if and only if $\gcd_{i \in B_j} (l_i) \nmid \sum_{\substack{i \in B_j \\ l_i \text{ even}}} \frac{l_i}{2} y_{a \sigma^{l_i/2}(a)}$.
    The former is equivalent to $\vee A_{\underline{B},\sigma,j} \in \Half_{\omega}$ and the latter to $\vee A_{\underline{B},\sigma,j} \in A_{\underline{B},\sigma}$.
\end{proof}

\subsection{Proof of Theorem \ref{thm:repr}}

\begin{proof}[Proof of \Cref{thm:repr}]
Denote the $\Stab(\underline{S})$-representation 
$\alpha_{\underline{S}} \otimes \widetilde{H}_{\ell(\underline{S})-3}(\Delta(\overline{\Fl}_{\omega, \geq \underline{S}})) \otimes \mathcal{C}(Z_{\Gamma_{\underline{S}}})$ simply by $V_{\underline{S}}$.
Then the character of the second summand of the right hand side of \eqref{eq:representations} is given by:
    \begin{align*}
        \chi_{C(Z_\Gamma+\epsilon)}(\sigma)-\chi_{C(Z_\Gamma)}(\sigma)
        &= \sum_{\underline{S} \in \Fl_{\omega}/\Aut(\Gamma)} \chi_{\Ind_{\Stab (\underline{S})}^{\Aut (\Gamma)} V_{\underline{S}}}(\sigma) \\
        &=  \sum_{\underline{S}  \in \Fl_{\omega}} \frac{1}{\lvert \mathcal{O}_{\underline{S}} \rvert} \chi_{\Ind_{\Stab (\underline{S})}^{\Aut (\Gamma)} V_{\underline{S}}}(\sigma) \\
        &= \sum_{\underline{S} \in \Fl_{\omega}}  \frac{1}{\lvert \mathcal{O}_{\underline{S}} \rvert}  \frac{1}{\lvert \Stab(\underline{S}) \rvert} \sum_{ \substack{g \in \Aut(\Gamma) \\ g^{-1}\sigma g \in \Stab(\underline{S})}} \chi_{V_{\underline{S}}}(g^{-1}\sigma g) \\
        &= \frac{1}{\lvert \Aut(\Gamma) \rvert} \sum_{\underline{S} \in \Fl_{\omega}} \sum_{ \substack{g \in \Aut(\Gamma) \\ g^{-1}\sigma g \in \Stab(\underline{S})}} \chi_{V_{\underline{S}}}(g^{-1}\sigma g) \\
        &= \frac{1}{\lvert \Aut(\Gamma) \rvert}\sum_{g \in \Aut(\Gamma)} \sum_{ \substack{\underline{S} \in \Fl_{\omega}\\ g^{-1}\sigma g \in \Stab(\underline{S})}} \chi_{V_{\underline{S}}}(g^{-1}\sigma g) \\
        &= \frac{1}{\lvert \Aut(\Gamma) \rvert}\sum_{g \in \Aut(\Gamma)} \sum_{ \substack{\underline{S} \in \Fl_{\omega}\\ \sigma \in \Stab(g\underline{S})}} \chi_{V_{g\underline{S}}}(\sigma) \\
        &= \frac{1}{\lvert \Aut(\Gamma) \rvert}\sum_{g \in \Aut(\Gamma)} \sum_{ \substack{\underline{T} \in \Fl_{\omega}\\ \sigma \in \Stab(\underline{T})}} \chi_{V_{\underline{T}}}(\sigma) \\
        &= \sum_{\underline{T} \in \Fl_{\omega}^\sigma} \chi_{V_{\underline{T}}}(\sigma)
    \end{align*}
    
Applying \Cref{lemma:inv_order_complex}, the identity \eqref{eq:representations} is equivalent to
\begin{equation}
    \label{eq:num_identity_invariant}
     C(Z_{\Gamma}^{\sigma}+\omega)- C(Z_{\Gamma}^{\sigma})= \sum_{\underline{S} \in \Fl_\omega^\sigma} \alpha_{\underline{S}}(\sigma) (-1)^{\ell(\underline{S})-3} \mu_{\Fl_{\omega}^\sigma} (\underline{S},\hat{1}) C(Z_{\Gamma_{\underline{S}}}^{\sigma})
\end{equation}

Following \Cref{thm:Ehrhart_qp}, we rewrite the LHS of \eqref{eq:num_identity_invariant} as follows
\[C(Z^{\sigma}_\Gamma+\omega)-C(Z_\Gamma^\sigma)= \sum_{F \textnormal{ forest of } \Gamma/\sigma} (-1)^{\dim (Z_\Gamma^\sigma)+\lvert F \rvert} ( \delta_{\underline{B}(F) \, (\omega_{\sigma}+t_{\sigma})\textnormal{-integral}}- \delta_{\underline{B}(F) \, t_{\sigma}\textnormal{-integral}}) Q_F \]
where \[Q_F \coloneqq \sum_{\substack{W \textnormal{ spanning}\\ \text{forest of }F
}} \mathrm{vol}(W);\] cf \S \ref{rmk:2zonoto} and \cref{defn:newflat}.

Note that if $\underline{R}(F)$ is $\omega$-integral, then $\omega_{\underline{B}} \in \Z^{\ell(\underline{B})}$ and $\gcd_{i \in B_j}(l_i) \mid (\omega_{\underline{B}})_j$ by \cref{lemma:number_blocks_B}.
In this case, $\delta_{\underline{B}(F) \, (\omega_{\sigma}+t_{\sigma})\textnormal{-integral}}= \delta_{\underline{B}(F) \, t_{\sigma}\textnormal{-integral}}$.
Therefore, we obtain
\begin{align*}
    \mathrm{LHS} \eqref{eq:num_identity_invariant} &= \sum_{F \textnormal{ forest of } \Gamma/\sigma } (-1)^{\dim (Z_\Gamma^\sigma)+\lvert F \rvert} ( \delta_{\underline{B}(F) \, (\omega_{\sigma}+t_{\sigma})\textnormal{-integral}}- \delta_{\underline{B}(F) \, t_{\sigma}\textnormal{-integral}}) Q_F \\
    &= \sum_{\substack{F \textnormal{ forest of } \Gamma/\sigma \\ \underline{R}(F) \in \Fl_\omega^\sigma}} (-1)^{\dim (Z_\Gamma^\sigma)+\lvert F \rvert} ( \delta_{\underline{B}(F) \, (\omega_{\sigma}+t_{\sigma})\textnormal{-integral}}- \delta_{\underline{B}(F) \, t_{\sigma}\textnormal{-integral}}) Q_F
\end{align*}

Observe that $\ell(\underline{S})= \ell(\pi(\underline{S})) + \ell(\sigma \colon \underline{S} \to \underline{S})$, and $\dim(Z^\sigma_{\Gamma_{\underline{S}}})= \dim(Z^\sigma_{\Gamma})-(\ell(\pi(\underline{S}))-1)$.

On the other hand, \Cref{lemma:alfa_S_sigma} implies that 
\begin{align*}
    \mathrm{RHS}\eqref{eq:num_identity_invariant}
    &= \sum_{\underline{S} \in \Fl_\omega^\sigma} (-1)^{\ell(\underline{S})-3+ \ell(\sigma \colon \underline{S} \to \underline{S})+  \lvert A_{\pi(\underline{S}), \sigma}\rvert } \mu_{\Fl_{\omega}^\sigma}(\underline{S}, \hat{1}) C(Z_{\Gamma_{\underline{S}}}^\sigma) \\
    &= \sum_{\underline{S} \in \Fl_\omega^\sigma} (-1)^{\ell(\pi(\underline{S}))-3 +  \lvert A_{\pi(\underline{S}), \sigma}\rvert} \mu_{\Fl_{\omega}^\sigma}(\underline{S}, \hat{1}) C(Z_{\Gamma_{\underline{S}}}^\sigma)\\
    &= \sum_{\underline{S} \in \Fl_\omega^\sigma} (-1)^{\ell(\pi(\underline{S}))-3+  \lvert A_{\pi(\underline{S}), \sigma}\rvert} \mu_{\Fl_{\omega}^\sigma}(\underline{S}, \hat{1}) \sum_{F \textnormal{ forest of } \Gamma_{\underline{S}}/\sigma} (-1)^{\dim(Z_{\Gamma_{\underline{S}}}^{\sigma})+ \lvert F \rvert} Q_F \delta_{\underline{B}(F) \, t_{\sigma}(\Gamma_{\underline{S}})\textnormal{-integral}} \\
    &= \sum_{F \textnormal{ forest of } \Gamma/\sigma} \left( \sum_{\substack{ \underline{S} \in \Fl_\omega^\sigma \\ \underline{S} \geq \underline{R}(F)}} (-1)^{ \lvert A_{\pi(\underline{S}), \sigma}\rvert} \mu_{\Fl_{\omega}^\sigma}(\underline{S}, \hat{1}) \delta_{\underline{B}(F) \, t_{\sigma}(\Gamma_{\underline{S}})\textnormal{-integral}} \right)  (-1)^{ \dim (Z_{\Gamma}^\sigma) + \lvert F \rvert}  Q_F.
\end{align*}
So it is enough to prove that
\begin{equation}
\label{eq:coef_repr_formula}
    \delta_{\underline{B}(F) \, (\omega_{\sigma}+t_{\sigma})\textnormal{-integral}}- \delta_{\underline{B}(F) \, t_{\sigma}\textnormal{-integral}} = 
    \sum_{\substack{ \underline{S} \in \Fl_\omega^\sigma \\ \underline{S} \geq \underline{R}(F)}} (-1)^{ \lvert A_{\pi(\underline{S}), \sigma}\rvert} \mu_{\Fl_{\omega}^\sigma}(\underline{S}, \hat{1}) \delta_{\underline{B}(F) \,\, t_{\sigma}(\Gamma_{\underline{S}})\textnormal{-integral}}
\end{equation}
for all $F$ such that $\underline{R}(F) \in \Fl_{\omega}^{\sigma}$.

Applying lemmas from the previous section we have:
\begin{alignat*}{2}
    \mathrm{RHS}\eqref{eq:coef_repr_formula}& =\sum_{ \underline{S} \in \Fl^\sigma_{\omega,\geq \underline{R}}} (-1)^{ \lvert A_{\pi(\underline{S}), \sigma}\rvert} \mu_{\Fl_{\omega}^\sigma}(\underline{S}, \hat{1}) \delta_{A_{\underline{B}, \sigma}^{\leq \underline{S}} = \emptyset} && \text{\cref{lemma:integral_in_quotient}}\\
    & = \sum_{ \underline{S} \in \Fl^\sigma_{\omega,\geq \underline{R}}} (-1)^{ \lvert A_{\underline{B}, \sigma}\rvert} \mu_{\Fl_{\omega}^\sigma}(\underline{S}, \hat{1}) \delta_{A_{\underline{B}, \sigma}^{\leq \underline{S}} = \emptyset} && \text{\cref{lemma:B_to_pi_S}}\\
    &= (-1)^{ \lvert A_{\underline{B}, \sigma}\rvert} \sum_{ \underline{S} \in \Fl^\sigma_{\omega,\geq \underline{R}}} \mu_{\Fl_{\omega}^\sigma}(\underline{S}, \hat{1}) \sum_{H\subseteq A_{\underline{B}, \sigma}^{\leq \underline{S}}} (-1)^{\lvert H \rvert} && (\dagger)\\
    &= (-1)^{ \lvert A_{\underline{B}, \sigma}\rvert} \sum_{H\subseteq A_{\underline{B}, \sigma}} (-1)^{\lvert H \rvert} \sum_{ \substack{ \underline{S} \in \Fl_\omega^\sigma \\ \underline{S} \geq \underline{R} \vee \bigvee H}} \mu_{\Fl_{\omega}^\sigma}(\underline{S}, \hat{1}) \\
    &= -\delta_{A_{\underline{B}, \sigma}= \emptyset} + (-1)^{\lvert A_{\underline{B}, \sigma}\rvert} \sum_{H\subseteq A_{\underline{B}, \sigma}} (-1)^{\lvert H \rvert} \left(1+\sum_{ \substack{ \underline{S} \in \Fl_\omega^\sigma \\ \underline{S} \geq \underline{R} \vee \bigvee H}} \mu_{\Fl_{\omega}^\sigma}(\underline{S}, \hat{1}) \right) \qquad && (\dagger)\\
    &= -\delta_{A_{\underline{B}, \sigma}= \emptyset} + (-1)^{\lvert A_{\underline{B}, \sigma}\rvert} \sum_{H\subseteq A_{\underline{B}, \sigma}} (-1)^{\lvert H \rvert} \delta_{\underline{R} \vee \bigvee H \, \omega\textnormal{-integral}} && \text{\cref{lem:RHintegral}}\\
    &= -\delta_{A_{\underline{B}, \sigma}= \emptyset} + (-1)^{\lvert A_{\underline{B}, \sigma}\rvert} \prod_{j =1 }^{\ell(\underline{B})} \delta_{\gcd_{i \in B_j} (l_i) \mid 2 (\omega_{\underline{B}})_j} \sum_{\Half_{\omega} \subseteq H\subseteq A_{\underline{B}, \sigma}} (-1)^{\lvert H \rvert} && (\dagger)\\
    &= -\delta_{A_{\underline{B}, \sigma}= \emptyset} + (-1)^{\lvert A_{\underline{B}, \sigma}\rvert}(-1)^{\lvert \Half_{\omega}\rvert} \delta_{\Half_{\omega} = A_{\underline{B}, \sigma}} \prod_{j =1 }^{\ell(\underline{B})} \delta_{\gcd_{i \in B_j} (l_i) \mid 2 (\omega_{\underline{B}})_j} && \text{\cref{lem:integralHals} and \ref{lemma:integral_in_quotient}} \\
    &= \delta_{\underline{B} \, \Gamma+\omega/\sigma \text{-integral}} - \delta_{\underline{B} \, \Gamma/\sigma \text{-integral}}=\mathrm{LHS}\eqref{eq:coef_repr_formula},
\end{alignat*}
where $(\dagger)$ accounts for the elementary property that $\sum_{C \subseteq B \subseteq A} (-1)^{\lvert B \rvert} =(-1)^{\lvert C \rvert}\delta_{A=C}$ for any finite sets $A,B,C$. This concludes the proof.
\end{proof}

\section{Hitchin systems and compactified Jacobians}\label{sec:Hitchincompactified}
The goal of this section is to determine a bijection between the irreducible components of the fibers of the Hitchin fibration and the lattice points of certain graphical zonotopes; see \cref{thm:geomvscombin}.

Let $C$ be a compact Riemann surface of genus $g$ with canonical bundle $\omega_C$.
\begin{defn}
The \textbf{Dolbeault moduli space} $\MDol(C, n,d)$ is the coarse moduli space which parametrises {$S$-equivalence classes of} semistable Higgs bundles on $C$ of rank $n$ and degree $d$, i.e.\ polystable pairs $(E, \phi)$ consisting of a vector bundle $E$ of rank $n$ and degree $d$ and a section $\phi \in H^0(C, \mathcal{E}nd(\mathcal{E})\otimes \omega_C)$; see \cite{Simpson1994}.
\end{defn} 

Let $2c(n,g) \coloneqq 2(g-1)n^2+2$ be the dimension of $\MDol(C,n,d)$. In the following we drop the dependence from $C,n,d$, or some of them when it does not cause any confusion. The Dolbeault moduli space is endowed with a proper map called Hitchin fibration. 

\begin{defn}
The \textbf{Hitchin fibration} 
\begin{align}\label{Hitchinfibration}
    \chi(n,d)\colon  \MDol(n,d) & \to A_{n}\coloneqq \bigoplus^{n}_{i=1} H^0(C, \omega_C^{\otimes i}), \\ (\mathcal{E}, \phi) & \mapsto \operatorname{char}(\phi) = (\mathrm{tr}(\phi), \mathrm{tr}(\Lambda^2 \phi), \ldots, \mathrm{tr}(\Lambda^i \phi), \ldots, \det(\phi)), \nonumber
\end{align}
assigns to $(E, \phi)$ the characteristic polynomial $\operatorname{char}(\phi)$ of the Higgs field $\phi$.
\end{defn}

The fibers of the Hitchin fibration have a classical modular interpretation that we recall here. Any point $a \in A_n$ is the characteristic polynomial of some Higgs field which can be seen as the defining equation of a curve $C_{a} \subset \operatorname{Tot}(\omega_C)$ of degree $n$ over $C$, called \textbf{spectral curve}. The fiber $\chi(n,d)^{-1}(a)$ is a compactified Jacobian of the spectral curve $C_{a}$ with respect to the canonical polarization $\qb^a$, defined as follows.

\begin{defn}
 A \textbf{polarization} on a connected curve $X = \bigcup_{i \in I} X_i$ is a tuple of rational numbers $\qb =(q_{X_i})$, one for each irreducible component $X_i$ of $X$, such that $\sum_{i}q_{X_i} \in \ZZ$. 
 
 Given any proper subscheme $Y \subseteq X$ of pure dimension 1, we set $\qb_Y \coloneqq \sum_j q_{X_j} \in \QQ$, where the sum runs over all the irreducible components $X_j$ of $Y$.
 
 A polarization is \textbf{general} if $\qb_Y \not\in \ZZ$ for any proper subscheme $Y \subseteq X$ of pure dimension 1 such that $Y$ and the complementary subcurve $Y^{c} \coloneqq \overline{X \setminus Y}$ are connected.
\end{defn}

\begin{defn}
 Let $\qb$ be a polarization of $X$. A rank 1 torsion free sheaf $\mathcal{I}$ on a curve $X$ with $\chi(I)=\qb_X$ is \textbf{$\qb$-semistable} (resp. $\qb$-stable) if for any proper subscheme $Y \subset X$ of pure dimension 1 we have that
 \begin{equation*} 
\chi(Y, \mathcal{I}_Y) \geq \qb_{Y} \qquad \text{(resp. }\chi(Y, \mathcal{I}_Y) > \qb_{Y})
 \end{equation*}
 where $\mathcal{I}_Y$ is the biggest torsion free quotient of the restriction $\mathcal{I}|_Y$ of $\mathcal{I}$ to $Y$. 

 Then the \textbf{compactified Jacobian}  $\overline{J}_{\qb}(X)$ is 
the (projective) coarse moduli space of $\qb$-semistable rank 1 torsion free sheaves; see for instance \cite{Seshadri82}. 
\end{defn}

\begin{prop}[BNR correspondence]\label{prop:BNR} \cite{Hitchin1987a, BNR89, Schaub1998, MRV2019}
Let $a \in A_n$ and consider the associated spectral curve $\pi_{a} \colon C_{a}=\bigcup_{i} C_i \to C$. The \textbf{canonical polarization} assigns to any proper subscheme $Y \subseteq C_{b}$ of pure dimension 1 the rational number
\[
\qb^{a}_{Y} \coloneqq \chi(I) \frac{\deg(\pi_a|_{Y})}{n},
\]
where $\deg(\pi_a|_{Y})$ is the degree of the restriction $\pi_{a}|_{Y} \colon Y \to C$. 

There exists an algebraic isomorphism
\begin{align*}
    \overline{J}_{\qb^a}(C_a) \to \chi^{-1}(n,d)(a), \qquad \text{given by }
    (C_{a}, \mathcal{I}) \mapsto (\mathcal{E} \coloneqq \pi_{a,*}\mathcal{I}, \phi \coloneqq t \, \cdot),
\end{align*}
where the Higgs field $\phi$ is induced by  multiplication by the tautological section $t$ of the line bundle $\pi^*_{a}\omega_C$ on $C_{a} \subset \operatorname{Tot}(\omega_C)$.
\end{prop}
From now on, we restrict our attention to reduced spectral curve. We denote by  $A^{\mathrm{red}}_n \subset A_n$ the locus of reduced spectral curves. For any partition $\underline{n}=\{n_i\}\vdash n$, $S^{\circ}_{\underline{n}}$ is the locus of characteristic polynomials whose irreducible factors have degree $n_i$, and they gives a stratification of $A^{\mathrm{red}}_n$; see also \S \ref{sec:mainresults_proof}.
\begin{prop}\label{prop:compJacobiansI}
Let $a \in A^{\mathrm{red}}_n$. The following facts hold.
\begin{enumerate}
    \item\label{item} The subvariety parametrising $\qb^a$-stable line bundles on the spectral curve $C_{a}$ is dense in $\chi(n,d)^{-1}(a)$;
    \item \label{item2}The irreducible components of $\chi(n,d)^{-1}(a)$ are indexed by the degrees of the restrictions of the $\qb^a$-stable line bundles to the irreducible components of $C_a$. 
    \item \label{item3}The number of irreducible components of $\chi(n,d)^{-1}(a)$ is constant for $a \in S^{\circ}_{\underline{n}}$.
\end{enumerate}
\end{prop}
\begin{proof}
The statement \eqref{item} holds for reduced curves with respect to a general polarization by \cite[Cor. 2.20]{MRV2017}. 
We show \eqref{item} for the (not necessarily general) canonical polarization $\qb^{a}$, via a small perturbation of the polarization $\qb^{a}$ to a general one, denoted $\underline{q}'$. 
Indeed, $\underline{q}'$-stable sheaves are $\qb^{a}$-semistable, and by the formation of GIT quotients there exists a specialization map
\[\overline{J}_{\underline{q}'}(C_a) \to \overline{J}_{\qb^a}(C_a)\]
 which is an isomorphism over the locus of $\qb^{a}$-stable sheaves. 
 Since $\underline{q}'$ is general, the locus of $\underline{q}'$-stable line bundles is open and dense in $\overline{J}_{\underline{q}'}(C_a)$ by \cite[Cor. 2.20]{MRV2017}. 
 It is then sufficient to prove that the $\qb^{a}$-stable locus is dense in $\overline{J}_{\qb^a}(C_a)$. 
 To this end, note that its complement $\overline{J}_{\qb^a}(C_a)^{\neq s}$ is the intersection of $\overline{J}_{\qb^a}(C_a) \simeq \chi^{-1}(n,d)(a)$ with the locus of polystable not stable Higgs bundles, i.e., the singular locus of $M(n,d)$.
 Any irreducible component of $\overline{J}_{\qb^a}(C_a)$, with its reduced structure, is Lagrangian in $M(n,d)$ and so of dimension $\dim M(n,d)/2$. On the other hand, $\overline{J}_{\qb^a}(C_a)^{\neq s}$ is dominated by a union of Lagrangian subvarieties contained in the normalization of the singular locus of $M(n,d)$ by \cite[Thm 3.1]{Matsushita} and \cite{Matsushita2000}, hence 
\[\dim \overline{J}_{\qb^a}(C_a)^{\neq s}
\leq \dim \mathrm{Sing}(M(n,d))/2< \dim M(n,d)/2.\] We obtain that $\overline{J}_{\qb^a}(C_a)^{\neq s}$ does not contain any irreducible component of $\overline{J}_{\qb^a}(C_a)$, which concludes the proof of \eqref{item}.

Now let $\mathcal{L}$ be a $\qb^a$-stable line bundle on $C_{a}$, and $J(C_{a}, \deg_{\bullet}(\mathcal{L}))$ be the Jacobian of line bundles $\mathcal{L}'$ on $C_{a}$ with $\deg_{Y} (\mathcal{L}') = \deg_{Y}(\mathcal{L})$ for all $Y \subseteq C_{a}$, which is an irreducible variety. Since any line bundle $\mathcal{L}'$ deformation equivalent to $\mathcal{L}$ is $\qb^a$-stable, we have $J(C_{a}, \deg_{\bullet}(\mathcal{L})) \subset \overline{J}_{\qb^a}(C_a)$. Because the two spaces have the same dimension, the closure of $J(C_{a}, \deg_{\bullet}(\mathcal{L}))$ is an irreducible component of $\overline{J}_{\qb^a}(C_a)$. This implies \eqref{item2}.

Statement \eqref{item3} now follows from \eqref{item2} and the fact that the stability of a line bundle is determined by numerical inequalities involving the degree of the irreducible components of $C_{a}$, and the sheaf of these irreducible components is locally constant on $S^{\circ}_{\underline{n}}$.
\end{proof}

\begin{notation}
Recall that the \textbf{dual graph} of a nodal curve $X = \bigcup_{i \in I} X_i$ with smooth irreducible components is a graph with $|I|$ vertices corresponding to the irreducible components $X_i$ of $X$, and $|X_i \cap X_j|$ edges between the vertices $i$ and $j$. In particular, for any partition $\underline{n}$ of $n$ we denote by $\Gamma[{\underline{n}}]$ the dual graph of the general spectral
curve in $S^{\circ}_{\underline{n}}$, i.e.\ $\Gamma[{\underline{n}}]$ is the graph with $|\underline{n}|$ vertices corresponding to
the irreducible components of the curve and $n_in_j (2g-2)$ edges between the vertices $i$, $j$,
corresponding to the intersection points of the components.
\end{notation}
Recall that $\omega_{\underline{n}}(d)$ is the rational vector in $\QQ^{|\underline{n}|}$ whose $i$-th component is $dn_i/n$. As a corollary of \cref{prop:compJacobiansI}.\eqref{item2}, the automorphism group $\Aut(\Gamma[{\underline{n}}])$ of the graph $\Gamma[{\underline{n}}]$ exchanges the irreducible components of the fibre $\chi(n,d)^{-1}(a)$, and acts by permutation on the translated graphical zonotope $Z_{\Gamma[{\underline{n}}]} + \omega_{\underline{n}}(d)$.

\begin{thm}[Hitchin system and zonotopes]\label{thm:geomvscombin} Let $a \in S^{\circ}_{\underline{n}}$. The following $\Aut(\underline{n})$-representations are isomorphic \begin{align*}
    (R^{2c(n,g)} \chi(n,d)_* \QQ_{\MDol(n,d)})_a & \simeq \mathcal{C}(Z_{\Gamma[\underline{n}]} + \omega_{\underline{n}}(d)).
\end{align*}
In particular, the number of irreducible components of the fibre $\chi(n,d)^{-1}(a)$ equals the number of lattice points contained in the zonotope $Z_{\Gamma[{\underline{n}}]} + \omega_{\underline{n}}(d)$, in symbols \begin{align*}
    \dim (R^{2c(n,g)} \chi(n,d)_* \QQ_{\MDol(n,d)})_a & = C(Z_{\Gamma[\underline{n}]} + \omega_{\underline{n}}(d)).
\end{align*}
\end{thm}
\begin{proof}
By \cref{prop:compJacobiansI} we can suppose that the spectral curve $C_{a}=\bigcup_{i \in I}C_i$ is nodal with smooth irreducible components. Denote by $C_K = \bigcup_{i \in K} C_i \subseteq C_a$ the subcurve corresponding to $K \subseteq I$. Pick a $\qb^a$-stable line bundle $\mathcal{I}$ in $\overline{J}_{\qb^a}(C_a)$, and let $\mathcal{I}_{K}$ be its restriction to $C_K$. By pushing forward from the normalization of $C_a$, we have the standard short exact sequence 
\[0 \to \mathcal{I}_K \to \bigoplus_{k \in K} \mathcal{I}_k \to \bigoplus_{\substack{J=\{i,j\} \subseteq K, \\|J|=2}} \mathcal{O}_{C_i \cap C_j} \to 0.\]
The stability condition reads
\begin{align*}
    & \sum_{i \in I} \chi(\mathcal{I}_i) - \sum_{J=\{i,j\} \subseteq I, |J|=2}|C_i \cap C_j| = \chi(\mathcal{I}) = d+n(1-g),\\
    & \sum_{i \in K} \chi(\mathcal{I}_i) - \sum_{J=\{i,j\} \subseteq K, |J|=2}|C_i \cap C_j|  = \chi(\mathcal{I}_K) > [d+n(1-g)]\sum_{k \in K}\frac{n_k}{n}.\\
\end{align*}
 For any pair $J=\{i,j\} \subseteq I$, the number of edges of $\Gamma(C_a)$ between the vertices corresponding to $C_i$ and $C_j$, denoted $y_J$, equals $|C_i \cap C_j|$. Set $x_i \coloneqq \chi(\mathcal{I}_i) - (1-g)n_i$, and $z_K = \sum_{J \subseteq K} y_J$. We can rewrite the previous system of equality and inequalities as
\begin{equation*}
  \sum_{i \in I} x_i = z_{I} + d, \quad \sum_{i \in K} x_i > z_K + \sum_{k \in K}\frac{d n_k}{n}\quad \forall K \subset I.
\end{equation*} 
As a result, the possible multidegrees of $\mathcal{I}$ are in correspondence with the lattice points of the zonotope $Z_{\Gamma(C_{a})}+\omega$ by \eqref{eq:inequequ}, and the bijection is $\Aut(\Gamma[{\underline{n}}])$-equivariant.
\end{proof} 

\section{Notation: intersection cohomology}

Given a $\QQ$-local system $L$ on a Zariski-open subset of a complex variety $X$, we denote the intersection cohomology of $X$ with coefficients in $L$ by $\IH^*(X, L) \coloneqq H^*(\IC(X, L))$, where $\IC(X, L)$ is the perverse intersection cohomology complex of $X$ with coefficients in $L$ shifted by $-\dim X$. See \cite{GM80, deCataldoMigliorini09, KirwanWoolf06} for an account. We will consider $\IC(X, L)$ as an element of the derived category $D^b(X, \QQ)$ of sheaves of $\QQ$-vector spaces on $X$ with bounded algebraically
constructible cohomology, or as an element of the bounded derived category of algebraic mixed Hodge modules $D^bMHM_{\text{alg}}(X)$.

\section{Ng\^{o} strings for the Hitchin system}\label{sec:Ngo}

In this section we recall how the cohomology of the Dolbeault moduli space $M(n,d)$ decomposes in Ng\^{o} strings, and we carefully study the stalk of these strings at the general point of their support. 

For any partition $\underline{n}$ of $n$ of length $r$, we define $S_{\underline{n}}$ as the image of the multiplication map
\begin{equation}\label{eq:S_n}
    \mathrm{mult}_{\underline{n}}: A_{\underline{n}} \coloneqq \prod^r_{i=1} A_{n_i} \to A_{n}
\end{equation}
given by $\mathrm{mult}_{\underline{n}}(a_1, \ldots, a_r) = \prod^r_{i=1} a_i$. There exists a partition $\underline{S}(\nbar) = \{S(\nbar)_k\}^{s}_{k=1}$ of $[r]$, maximal with respect to the order $<$, with the property that there exists $i(k) \in S(\nbar)_k$ such that $n_j=n_{i(k)}$ for all $j \in S(\nbar)_k$. In other words, we group together the $n_i$ which are equal. Then the multiplication $\mathrm{mult}_{\underline{n}}$ factors as follows
\begin{equation}\label{eq:normalizationSn}
    A_{\underline{n}} \xrightarrow{\eta_{\nbar}} A_{\underline{S}(\nbar)} \coloneqq\prod^t_{k=1}  \mathrm{Sym}^{|S(\nbar)_k|}A_{n_{i(k)}} \xrightarrow{\nu_{\nbar}} S_{\nbar}.
\end{equation}
The map $\nu_{\nbar}$ is the normalization of $S_{\underline{n}}$, since $\mathrm{mult}_{\underline{n}}$ is a finite map and $\nu_{\nbar}$ is generically 1:1 with normal domain. In particular, the composition is \'{e}tale over \[S^{\circ}_{\underline{n}} \coloneqq (S_{\nbar} \cap A^{\mathrm{red}})\setminus \bigcup_{\mbar < \nbar} S_{\mbar}\] with Galois group  the automorphism  $\mathrm{Stab}(\underline{n}) \coloneqq \prod^n_{i=1} \mathfrak{S}_{\alpha_{i}}$ of the partition $\underline{n} = 1^{\alpha_1} 2^{\alpha_2} \cdots n^{\alpha_n}$, i.e.\ the subgroup of the symmetric group $\mathfrak{S}_{r}$ stabilizing $\underline{n}$, or equivalently the automorphism group $\Aut(\Gamma[{\underline{n}}])$ of the general spectral curve in $S_{\underline{n}}$.

We define also the following spaces:
\begin{enumerate}[label=(\roman*)]
    \item $S^\times_{\underline{n}}\subseteq S^{\circ}_{\underline{n}}$ is the locus parametrising reducible nodal curves having smooth irreducible components of degree $n_i$ over $C$;
    \item $\pi_{\underline{n}} \colon \mathcal{C}^{\times}_{\underline{n}} \to S^\times_{\underline{n}}$ is the universal spectral curve over $S^\times_{\underline{n}}$;
    \item $\mathrm{Pic}^0(\mathcal{C}^{\times}_{\underline{n}}) \to S^\times_{\underline{n}}$ is the (relative) Jacobian of the universal spectral curve over $S^\times_{\underline{n}}$;
    \item $g_{\underline{n}}: \mathscr{A}^{\times}_{\underline{n}} \to S^\times_{\underline{n}}$ is the maximal abelian proper quotient of $\mathrm{Pic}^0(\mathcal{C}^{\times}_{\underline{n}})$, or equivalently the Jacobian of the normalization of $\mathcal{C}^{\times}_{\underline{n}}$;
    \item $\Lambda^l_{\underline{n}} \coloneqq R^lg_{\underline{n}, *} \QQ_{\mathscr{A}^{\times}_{\underline{n}}}= \Lambda^l R^1g_{\underline{n}, *} \QQ_{\mathscr{A}^{\times}_{\underline{n}}}= \Lambda^l R^1\pi_{\underline{n}, *} \QQ_{\mathscr{A}^{\times}_{\underline{n}}}$; see \cite[Lemma 1.3.5]{deCataldoHauselMigliorini2012}.
\end{enumerate}

\begin{thm}[Ng\^{o} strings from Hitchin system] \label{thm:Ngostring} \cite[Thm 3.6, Thm 3.11]{MM2022} There exists an isomorphism in $D^bMHM_{\mathrm{alg}}(A^{\mathrm{red}}_n)$ (resp. $D^b(A^{\mathrm{red}}_n, \QQ)$) 
\[ R \chi(n,d)_* \IC(\MDol(n,d), \QQ)|_{A^{\mathrm{red}}_n} \simeq \bigoplus_{\underline{n}\vdash n} \mathscr{S}(\mathscr{L}_{\underline{n}}(d))|_{A^{\mathrm{red}}_n} ,
\]
where the Ng\^{o} string $\mathscr{S}(\mathscr{L}_{\underline{n}}(d))$ is the mixed Hodge module supported on $S_{ \underline{n}}$ given by
\[
\mathscr{S}(\mathscr{L}_{\underline{n}}(d)) \coloneqq \bigoplus^{2 \dim S_{ \underline{n}}}_{l=0} \IC(S_{\underline{n}}, \Lambda^l_{\underline{n}} 
\otimes \mathscr{L}_{ \underline{n}}(d))[-l-2\codim S_{\underline{n}}](-\codim S_{\underline{n}}), 
\]
where $\mathscr{L}_{ \underline{n}}(d)$ is a polarizable variation of pure Hodge structures of weight zero and of Hodge-Tate type supported on an open set of $S_{\underline{n}}$. Moreover, $R^{2c(n,g)} \chi(n,d)_*\QQ_{M(n,d)}$ admits the following direct summand
\[
\bigoplus_{\underline{n} \in I} i_{\underline{n}, *}\mathscr{L}_{ \underline{n}}(d),
\]
where $i_{\underline{n}} \colon S^{\circ}_{\underline{n}} \to A_{n}$ is the natural immersion.
\end{thm}
\begin{rmk}
If $\underline{n}=\{n\}$, then $\mathscr{L}_{\{n\}}(d)$ is the trivial local system, and $\mathscr{S}(\mathscr{L}_{\{n\}}(d))$ is independent of $d$. In this case we simply write $\mathscr{S}_{n}$ for $\mathscr{S}(\mathscr{L}_{\{n\}}(d))$.
\end{rmk}

\begin{rmk}\label{rmk:represent}
The pullback of $\mathscr{L}_{ \underline{n}}(d)$ to the preimage of $S^\times_{\underline{n}}$ under the map $A_{\underline{n}} \to S_{\underline{n}}$ has trivial monodromy. This is proved in \cite[Cor. 6.14]{deCataldoHeinlothMigliorini19}  for $\gcd(n,d)=1$, but the same proof of \cite[Cor. 6.14]{deCataldoHeinlothMigliorini19} works in arbitrary degree using \cref{prop:compJacobiansI}. This means that the monodromy of the local system $\mathscr{L}_{\underline{n}}(d)$ is a representation of $ \Aut(\Gamma[{\underline{n}}])\simeq \mathrm{Stab}(\underline{n})$.
\end{rmk}
From now on, let $\mbar$ be a partition of $n$, and $a$ be a general closed point of the stratum $S_{\underline{m}}$.
\begin{lem}\label{lem:localmodelforLn}
Let $r(\underline{n}, d)$ be the rank of $\mathscr{L}_{ \underline{n}}(d)$. For any $\nbar \geq \mbar$, then $(i_{\nbar, *} \mathscr{L}_{ \underline{n}}(d))_a \simeq \QQ_a^{r(\underline{n}, d)}$.
\end{lem}
\begin{proof}
 By \cref{rmk:represent}, the monodromy of $\mathscr{L}_{ \underline{n}}(d)$ is non-trivial only along the loops winding around the branch locus $B_{\underline{n}}$ of $A_{\underline{n}} \to S_{\underline{n}}$, i.e.\ the locus of polynomials with coefficients in $H^0(C, K^{\otimes i}_C)$ with non-distinct irreducible factors. 
But the point $a \in S_{\nbar}$ does not lay on $B_{\underline{n}}$ 
for any $\nbar \geq \mbar$. So $\mathscr{L}_{ \underline{n}}(d)$ extends through $a$ as a local system, and $i_{\nbar, *}\mathscr{L}_{ \underline{n}}(d)_a \simeq \QQ_a^{r(\nbar, d)}$.
\end{proof}

\begin{prop}\label{prop:stalkstrings}
Let $N_{\nbar/\mbar}$ be a general slice transverse to $S_{\mbar}$ in $S_{\nbar}$ passing through $a$. Set $N^{\circ}_{\nbar/\mbar} \coloneqq N_{\nbar/\mbar} \cap S^{\circ}_{\nbar}$. Then we have
\begin{align*}
    (R^{2c} \chi(n, &1)_*  \QQ_{M(n,1)})_a  \simeq \bigoplus_{\nbar\geq \mbar} \bigoplus^{\codim_{S_{\nbar}} S_{\mbar}-1}_{l=0} \mathcal{H}^l_a(\IC(N_{\nbar/\mbar}, \Lambda^l_{\nbar}|_{N^{\circ}_{\nbar/\mbar}}))^{r(\underline{n}, d)}\\
    & = \bigoplus_{\nbar\geq \mbar} \bigg[ \QQ^{r(\underline{n}, d)}_a \oplus  \bigoplus^{\codim_{S_{\nbar}} S_{\mbar}-1}_{l=1} \mathcal{H}^l_a(\IC(N_{\nbar/\mbar}, \Lambda^l_{\nbar}|_{N^{\circ}_{\nbar/\mbar}}))^{r(\underline{n}, d)} \bigg].\\
\end{align*}
\end{prop}
\begin{proof}
By \cref{thm:Ngostring} we have
\begin{align*}
    \mathcal{H}^{2c}_a(\mathscr{S}(\mathscr{L}_{\underline{n}}(d))) = \bigoplus^{2\dim S_{\nbar}}_{l=0} \mathcal{H}^{2\dim(S_{\nbar})-l}_a(\IC(S_{\nbar}, \Lambda^l_{\nbar}\otimes  \mathscr{L}_{ \underline{n}}(d))).
\end{align*}
Relative Hard Lefschetz gives $\Lambda^l_{\nbar} \simeq \Lambda^{2 \dim S_{\nbar}-l}_{\nbar}$. By \cref{lem:localmodelforLn} we also have $\mathcal{H}^0_a(\IC(S_{\nbar}, \mathscr{L}_{ \underline{n}}(d)))= (i_{\nbar, *}L_{\nbar})_a \simeq \QQ^{r(\underline{n},d)}_a$. All together we obtain
\begin{align*}
    \mathcal{H}^{2c}_a(\mathscr{S}_{\nbar}) = \QQ^{r(\underline{n},d)}_a \oplus \bigoplus^{2\dim S_{\nbar}}_{l=1} \mathcal{H}^{l}_a(\IC(S_{\nbar}, \Lambda^l_{\nbar}))^{r(\underline{n},d)}.
\end{align*}
By stratification theory it is clear that 
\[\mathcal{H}^{l}_a(\IC(S_{\nbar}, \Lambda^l_{\nbar})) \simeq \mathcal{H}^{l}_a(\IC(N_{\nbar/\mbar}, \Lambda^l_{\nbar}|_{N^{\circ}_{\nbar/\mbar}})).\]
Further, by the strong support condition for intersection cohomology we obtain
\[ \mathcal{H}^{l}(\IC(N_{\nbar/\mbar}, \Lambda^l_{\nbar}|_{N^{\circ}_{\nbar/\mbar}}))=0 \quad \text{ for all }l>\dim N_{\nbar/\mbar} -1 = \codim_{S_{\nbar}} S_{\mbar}-1.\]
\end{proof}

\begin{cor}\label{cor:vanishingresult}\emph{\cite[Lemma 6.9]{deCataldoHeinlothMigliorini19}} $\mathcal{H}_a^l(\IC(S_{\nbar}, \Lambda^{k}_{\nbar}))=0$ for $l>k$.
\end{cor}
\begin{proof}
The same proof of \cref{prop:stalkstrings} gives that $\mathcal{H}_a^l(\IC(S_{\nbar}, \Lambda^{k}_{\nbar}))$ is a direct summand of $(R^{2c + (l-k)} \chi(n, 1)_*  \QQ_{M(n,1)})_a=0$.
\end{proof}

Let $\mbar$ be a partition of $n$ of length $s$. Any partition $\underline{S}$ of the set $[s]$ of length $r$ corresponds to a partition $\nbar \geq \mbar$ of $n$ of length $r$. Indeed, given $\underline{S}$, we define \ $\nbar=\{\nbar_1, \ldots, \nbar_{r}\}$ with $\underline{n}_i = \sum_{j \in S_i} m_j$; see also \cref{defn:partition}.
 
The map $j_{\underline{S}} \colon A_{\mbar} \hookrightarrow A_{\nbar}$, given by
$
j_{\underline{S}}(a_1, \ldots, a_r)=(\prod_{j \in S_1} a_j, \ldots, \prod_{j \in S_r} a_j),
$
obviously lifts the inclusion $S_{\mbar} \hookrightarrow S_{\nbar}$. 

Let $\tilde{a}$ be a lift in $A_{\mbar}$ of a general point $a$ in $S_{\mbar}$.
Since $A_n$ and $A_m$ are smooth, a slice of $A_n$ transverse to $j_{\underline{S}}(A_m)$ passing through $\tilde{a}$ is smooth, and it can be identified with the product
\begin{equation}\label{eq:slice}
    N_{S_1} \times \ldots \times N_{S_r} \coloneqq N_{\{n_1\}/\nbar_1} \times \ldots \times N_{\{n_r\}/\nbar_r}.
\end{equation}
Since the map $\eta_{\nbar}$ in \eqref{eq:normalizationSn} is \'{e}tale at $\tilde{a}$, we can identify the normalization of a branch of $N_{\nbar/\mbar}$, denoted $N_{\underline{S}}$, with the product above. To summarize, there exists a bijection between the branches of $N_{\nbar/\mbar}$ and the partitions of $[s]$.

For brevity we denote $\Lambda^l_{\Sbar} \coloneqq \Lambda^l_{\nbar}|_{N_{\Sbar} \cap S^{\circ}_{\nbar/\mbar}}$ and $\Lambda^l_{S_i} \coloneqq \Lambda^l_{\{n_i\}}|_{N_{S_i} \cap S^{\circ}_{\{n_i\}/\nbar_i}}$. 

\begin{prop}[Reduction to lower rank]\label{prop:reductiontolowerrank}
Under the identification $N_{\Sbar} \simeq N_{S_1} \times \ldots \times N_{S_r}$ above, we have
\begin{align}
\mathcal{H}^{2c(n,g)}_a(\mathscr{S}(\mathscr{L}_{\underline{n}}(d)) & \simeq \bigoplus_{N_{\Sbar} \subseteq N_{\nbar/\mbar}}\mathcal{H}^{2c(n,g)}_a(\mathscr{S}(\mathscr{L}_{\underline{n}}(d)|_{N_{\Sbar}}) \label{eq: decnormalization}\\
    \mathcal{H}^{2c(n,g)}_a(\mathscr{S}(\mathscr{L}_{\underline{n}}(d))|_{N_{\Sbar}}) & \simeq \bigg(\bigotimes^r_{i=1} \mathcal{H}_a^{2c(n_i, g)}(\mathscr{S}_{\{n_i\}}|_{N_{\{n_i\}/\nbar_i}}) \bigg) \otimes \QQ^{r(\underline{n}, d)}. \label{eq:Kunnethdecom}
\end{align}
\end{prop}
\begin{proof}
\cref{eq: decnormalization} follows from the proof of \cref{prop:stalkstrings} and from the behaviour of intersection cohomology under normalization
\[
\mathcal{H}^l_a(\IC(N_{\nbar/\mbar}, \Lambda^l_{\nbar}|_{N^{\circ}_{\nbar/\mbar}})) \simeq \bigoplus_{N_{\Sbar} \subseteq N_{\nbar/\mbar}} \mathcal{H}^l_a(\IC(N_{\Sbar}, \Lambda^l_{\Sbar})).
\]
\cref{eq:Kunnethdecom} boils down to the K\"{u}nneth decomposition for intersection cohomology. In order to prove this identity, consider the spectral curve $\pi_b: C_b= \bigcup^r_{i=1} C_i \to C \times b$ over the generic point $b$ of $N_{\underline{S}}$. The map $\pi_b$ restricts to a finite map $C_i \to C \times b$ of degree $n_i$ on the smooth irreducible component $C_i$ of genus $g_i$. In particular, over the general point $\bar{b}$ we have 
\[
\Jacob(C^{\circ, \nu}_{\bar{b}}) \simeq \Jacob(C^{\circ}_{1, \bar{b}}) \times \ldots \times \Jacob(C^{\circ}_{r, \bar{b}}). 
\]
Since the point $a$ does not lay on the branch locus of $A_{\underline{n}} \to S_{\underline{n}}$, the monodromy of $\pi_b$ over $N_{\underline{S}}$ preserves the irreducible components $C_i \times \bar{b}$, and this implies
\[
\Lambda^{l}_{\Sbar} \simeq (\Lambda^{\bullet}_{S_1} \boxtimes \ldots \boxtimes \Lambda^{\bullet}_{S_r})^{l} = \bigoplus_{\lbar \vdash l \colon \ell(\lbar)=r}\Lambda^{l_1}_{S_1} \boxtimes \ldots \boxtimes \Lambda^{l_r}_{S_r}.
\]
where the summation index runs over the partition $\lbar$ of $l$ of lenght $r$. Then the K\"{u}nneth decomposition for intersection cohomology gives
\begin{align*}
    \mathcal{H}^l_a (\IC(N_{\Sbar}, \Lambda^l_{\Sbar})) & \simeq   \bigoplus_{\kbar \vdash l \colon \ell(\kbar)=r} \mathcal{H}^l_a(\IC(N_{S_1} \times \ldots \times N_{S_r}, \Lambda^{k_1}_{S_1} \boxtimes \ldots \boxtimes \Lambda^{k_r}_{S_r}))\\
   & \simeq   \bigoplus_{\kbar \vdash l  \colon \ell(\kbar)=r} \:\: \bigoplus_{\lbar \vdash l  \colon \ell(\lbar)=r} \mathcal{H}^{l_1}_a(\IC(N_{S_1}, \Lambda^{k_1}_{S_1})) \otimes \ldots \otimes \mathcal{H}^{l_r}_a(\IC(N_{S_r}, \Lambda^{k_r}_{S_r})),
\end{align*}
with $l=k$ (i.e.\ the two summation indexes both running over the same set of partition of $l$ of length $r$). For each pair \{$\kbar$, $\lbar$\} of partitions of $l$ there are two options: either $l_i = k_i$ for all $i \in [r]$, or there exists $i_0 \in [r]$ such that $l_{i_0} > k_{i_0}$. In the latter case, $\mathcal{H}^{l_{i_0}}_a(\IC(N_{S_{i_0}}, \Lambda^{k_{i_0}}_{S_{i_0}}))=0$ by \cref{cor:vanishingresult}. Thus we obtain
\begin{align*}
    \mathcal{H}^l_a(\IC(N_{\Sbar}, \Lambda^l_{\Sbar})) \simeq  \bigoplus_{\lbar \vdash l \colon \ell(\lbar)=r} \mathcal{H}^{l_1}_a(\IC(N_{S_1}, \Lambda^{l_1}_{S_1})) \otimes \ldots \otimes \mathcal{H}^{l_r}_a(\IC(N_{S_r}, \Lambda^{l_r}_{S_r})).
\end{align*}
We conclude that
\begin{align*}
     \mathcal{H}^{2c(n,g)}_a  (\mathscr{S}(\mathscr{L}_{\underline{n}}(d)|_{N_{\Sbar}}) & \simeq \bigoplus_{l\geq 0}  \mathcal{H}^l_a(\IC(N_{\Sbar}, \Lambda^l_{\Sbar}))\otimes \QQ^{r(\nbar, d)}\\
     & \simeq \bigg( \bigotimes^r_{i=1}  \bigoplus_{l_i\geq 0}  \mathcal{H}^{l_i}_a(\IC(N_{S_i}, \Lambda^{l_i}_{S_i}))\bigg)\otimes \QQ^{r(\nbar, d)}\\
     & \simeq \bigg(\bigotimes^r_{i=1} \mathcal{H}_a^{2c(n_i, g)}(\mathscr{S}_{\{n_i\}}|_{N_{n_i/\nbar_i}})\bigg) \otimes \QQ^{r(\nbar, d)}.
\end{align*}
\end{proof}

\section{Proof of \texorpdfstring{\cref{thm:fullsupport}}{Theorem \ref{thm:fullsupport}} and \texorpdfstring{\cref{thm:combchar}}{Theorem \ref{thm:combchar}}}\label{sec:mainresults_proof} 

\begin{proof}[Proof of \cref{thm:fullsupport}]
By \cref{thm:Ngostring} it suffices to check that $(R^{2c(n,g)}\chi(n,0)_*\QQ_{\MDol(n,0)})_{a}$ has no direct summand properly supported at the general point $a \in S_{\mbar}$ for any $\mbar$. We prove it by induction on the length of $\mbar$. If $\mbar$ has length $1$, then $a \in A^{\mathrm{reg}}$ avoids any potential proper support of the decomposition theorem, so there is nothing to prove. 

Suppose now that $\mbar$ is a permutation of length $s>1$, and no summands of $R^{2c(n,g)}\chi(n,0)_*\QQ_{\MDol(n,0)}|_{A^{\mathrm{red}}}$ are supported on $S_{\nbar} \cap A^{\mathrm{red}}$, for any $n \in \ZZ_{>0}$ and for any partition $\nbar$ of $n$ such that $\mbar< \nbar< \{n\}$. We show that $(R^{2c(n,g)}\chi(n,0)_*\QQ_{\MDol(n,0)})_{a}$ has no summands at $a \in S_{\mbar}$.
Otherwise, there exists a non-zero $\QQ$-vector space $\mathcal{F}_a$ such that 
\[(R^{2c(n,g)}\chi(n,0)_*\QQ_{\MDol(n,0)})_{a} \simeq \mathcal{H}^{2c(n,g)}_a(\mathscr{S}_{\{n\}}) \oplus \mathcal{F}_a.\]

Let $e$ be an integer with $\gcd(e,n)=1$. \cref{prop:reductiontolowerrank} gives
\begin{align*}
(R^{2c(n,g)}&\chi (n,e )_*  \QQ_{\MDol(n,e)})_{a} \simeq \bigoplus_{\nbar \geq \mbar}\mathcal{H}^{2c(n,g)}_a(\mathscr{S}(\mathscr{L}_{\nbar}(e))) \\
& \simeq \bigoplus_{\Sbar \vdash [s]} \bigg( \bigotimes^{\ell(\Sbar)}_{i=1} \mathcal{H}_a^{2c(\sigma_i, g)}(\mathscr{S}_{\{\sigma_i\}}|_{N_{S_i}}) \bigg) \otimes \QQ^{r(\ell(\Sbar), e)}\\
& \simeq \mathcal{H}^{2c(n,g)}_a(\mathscr{S}_{\{n\}}) \oplus
 \bigoplus_{\substack{ \Sbar \vdash [s] \\ \Sbar \neq \{[s]\}}} \bigg(\bigotimes^{\ell(\Sbar)}_{i=1} \mathcal{H}_a^{2c(\sigma_i, g)}(\mathscr{S}_{\{\sigma_i\}}|_{N_{S_i}})\bigg) \otimes \QQ^{r(\ell(\Sbar), e)},
\end{align*}
where we set $\sigma_i \coloneqq \sum_{j \in S_i} m_j$. By induction hypothesis we have that for any $\Sbar \vdash [s]$ 
\[
\mathcal{H}_a^{2c(\sigma_i, g)}(\mathscr{S}_{\{\sigma_i\}}|_{N_{S_i}}) \simeq \mathcal{H}_a^{2c(\sigma_i, g)}(\mathscr{S}_{\{\sigma_i\}}) \simeq 
(R^{2c(\sigma_i,g)}\chi (\sigma_i,0 )_*  \QQ_{\MDol(\sigma_i,0)})_{a}.
\]
Recall that the vertices of the dual graph $\Gamma[\underline{m}]$ are labelled by $[s]$. Let $\Gamma(S_i)$ be the maximal subgraph of $\Gamma[\underline{m}]$ on the vertex set $S_i$. \cref{thm:geomvscombin} gives 
\[
\dim \mathcal{H}_a^{2c(\sigma_i, g)}(\mathscr{S}_{\{\sigma_i\}}|_{N_{S_i}}) = C(Z_{\Gamma(S_i)}).
\]
Now set $\omega_{\underline{m}}(e) =(m_1e/n, \ldots, m_se/n) \in \QQ^{s}$ with $\underline{m}=\{m_1, \ldots, m_s\}$, and write
\begin{align*}
C(Z_{\Gamma[\underline{m}]}+ \omega_{\underline{m}}(e)) & = \dim (R^{2c(n,g)}\chi (n,e )_*  \QQ_{\MDol(n,e)})_{a}\\
& = \dim \mathcal{H}^{2c(n,g)}_a(\mathscr{S}_{\{n\}})  + \sum_{\substack{ \Sbar \vdash [s] \\ \Sbar \neq \{[s]\}}}  r(\ell(\Sbar), e) \cdot \prod^{|\Sbar|}_{i=1} \dim \mathcal{H}_a^{2c(\sigma_i, g)}(\mathscr{S}_{\{\sigma_i\}}|_{N_{S_i}})\\
& = \dim \mathcal{H}^{2c(n,g)}_a(\mathscr{S}_{\{n\}})  + \sum_{\substack{ \Sbar \vdash [s] \\ \Sbar \neq \{[s]\}}}  r(\ell(\Sbar), e) \cdot \prod^{|\Sbar|}_{i=1} C(Z_{\Gamma(S_i)})\\
& = \dim \mathcal{H}^{2c(n,g)}_a(\mathscr{S}_{\{n\}})  + \sum_{\substack{ \Sbar \vdash [s] \\ \Sbar \neq \{[s]\}}}  (\ell(\Sbar)-1)! \cdot \prod^{|\Sbar|}_{i=1} C(Z_{\Gamma(S_i)}).
\end{align*}
Here the fourth equality follows from an explicit computation of the cohomology of the Cattani--Kaplan--Schmid complex for a versal family of $C_a$; see \cite[Cor.~6.20]{deCataldoHeinlothMigliorini19}. 
Together with the combinatorial identity of \cref{rmk:commentformula}.(i)
\[C(Z_{\Gamma[\underline{m}]}+ \omega_{\underline{m}}(e)) = C(Z_{\Gamma[\underline{m}]}) + \sum_{\substack{ \Sbar \vdash [s] \\ \Sbar \neq \{[s]\}}}  (\lvert \underline{S} \rvert-1)! \prod_{j=1}^{\lvert \underline{S} \rvert} C(Z_{\Gamma(S_j)}),\]
we conclude that 
\[\dim \mathcal{H}^{2c(n,g)}_a(\mathscr{S}_{\{n\}})=C(Z_{\Gamma[\underline{m}]})=\dim(R^{2c(n,g)}\chi(n,0)_*\QQ_{\MDol(n,0)})_{a},\]
which implies $\mathcal{F}_a =0$, i.e.\ $S_{\mbar}$ is not a support of  $R^{2c(n,g)}\chi(n,0)_*\QQ_{\MDol(n,0)}$.
\end{proof}

\begin{proof}[Proof of \cref{thm:combchar}] Let $a$ be a general point in $S_{\mbar}$ for any partition $\mbar =\{m_1, \ldots, m_s\}$.
\cref{prop:reductiontolowerrank} gives
\begin{equation}\label{eq:decthm}
(R^{2c(n,g)}\chi (n,d )_*  \QQ_{\MDol(n,d)})_{a} \simeq \bigoplus_{\Sbar=\{S_i\} \vdash \underline{m}} \mathscr{L}_{ \{|S_i|\}}(d)_{a} \otimes \bigg( \bigotimes^{\ell(\Sbar)}_{i=1} \mathcal{H}_a^{2c(\sigma_i, g)}(\mathscr{S}_{\{\sigma_i\}}|_{N_{S_i}}) \bigg).
\end{equation}
The LHS of \eqref{eq:decthm} is an $\Aut(\underline{m})$-representation, while the summands of the RHS are $\Aut(\underline{m})$-representations induced by the injection $\Stab(\underline{S}) \hookrightarrow \Aut(\underline{m})$. For convenience, write $\omega$ for the vector $ \omega_{\underline{m}}(d)$ defined in \eqref{eq:vectoromega}. By \cref{thm:geomvscombin} and \cref{thm:repr}, the following $\Aut(\underline{m})$-representations are isomorphic
\begin{align*}
(R^{2c(n,g)} & \chi (n,d )_*  \QQ_{\MDol(n,d)})_{a}  \simeq \mathcal{C}(Z_{\Gamma[{\underline{m}}]} + \omega_{\underline{m}}(d)) \\ \simeq & \,\mathcal{C}(Z_{\Gamma[{\underline{m}}]}) \oplus \bigoplus_{\underline{S} \in \Pi_\omega(\underline{m}) /\Aut(\Gamma[\underline{m}])} \Ind_{\Stab (\underline{S})}^{\Aut (\Gamma[\underline{m}])} \big( \sgn \otimes \widetilde{H}_{\ell(\underline{S})-3}(\Delta(\overline{\Pi}_{\omega, \geq \underline{S}})) \otimes \mathcal{C}(Z_{\Gamma_{\underline{S}}})\big),\\
\simeq & \,(R^{2c(n,g)}\chi(n,0)_*\QQ_{\MDol(n,0)})_{a} \oplus\\ & \oplus \bigoplus_{\underline{S} \in \overline{\Pi}_\omega(\underline{m}) /\Aut(\Gamma[\underline{m}])} \Ind_{\Stab (\underline{S})}^{\Aut (\Gamma[\underline{m}])} \big( \sgn \otimes \widetilde{H}_{\ell(\underline{S})-3}(\Delta(\overline{\Pi}_{\omega_{\underline{S}}})) \otimes \bigg( \bigotimes^{\ell(\Sbar)}_{i=1} \mathcal{H}_a^{2c(\sigma_i, g)}(\mathscr{S}_{\{\sigma_i\}}|_{N_{S_i}}) \bigg) \big),\end{align*}
Comparing with \eqref{eq:decthm}, we conclude that the $\Aut(\Gamma[\underline{n}])$-representations
\[\mathscr{L}_{ \underline{n}}(d)_{a} \simeq \sgn \otimes \widetilde{H}_{\ell(\underline{n})-3}(\Delta(\overline{\Pi}_{\omega_{\underline{n}}(d)}))\] are isomorphic by induction on the length of $\mbar$. 
\end{proof}
\bibliographystyle{plain}
\bibliography{construction}
\end{document}